\documentclass[pdflatex,sn-mathphys-num]{sn-jnl}

\usepackage{graphicx}%
\usepackage{multirow}%
\usepackage{amsmath,amssymb,amsfonts}%
\usepackage{amsthm}%
\usepackage{mathrsfs}%
\usepackage[title]{appendix}%
\usepackage{xcolor}%
\usepackage{textcomp}%
\usepackage{manyfoot}%
\usepackage{booktabs}%
\usepackage{algorithm}%
\usepackage{algorithmicx}%
\usepackage{algpseudocode}%
\usepackage{listings}%
\usepackage{url}

\theoremstyle{thmstyleone}%
\newtheorem{theorem}{Theorem}[section]
\newtheorem{proposition}[theorem]{Proposition}%

\theoremstyle{thmstyletwo}%
\newtheorem{remark}{Remark}[section]%

\theoremstyle{thmstylethree}%
\newtheorem{definition}{Definition}[section]%
\newtheorem{lemma}{Lemma}[section]
\newtheorem{condition}{Condition}
\newtheorem{corollary}{Corollary}[section]
\newtheorem{assumption}{Assumption}
\begin{document}

\title[Article Title]{Error Bound and Convergence Guaranteed Inexact PALM Algorithm for Low-Rank Matrix Optimization with Factorized Schatten-q Quasi-Norm Regularization}


\author[1]{\fnm{Yongjun} \sur{Chen}}\email{yongjun-24@mails.tsinghua.edu.cn}

\author[2]{\fnm{Defeng} \sur{Sun}}\email{defeng.sun@polyu.edu.hk}

\author*[1]{\fnm{Liping} \sur{Zhang}}\email{lipingzhang@mail.tsinghua.edu.cn}



\affil*[1]{\orgdiv{Department of Mathematical Sciences}, \orgname{Tsinghua University}, \orgaddress{\city{Beijing}, \postcode{100084},  \country{China}}}

\affil[2]{\orgdiv{Department of Applied Mathematics}, \orgname{The Hong Kong Polytechnic University}, \orgaddress{\street{Hung Hom}, \city{Kowloon}, \state{Hong Kong}}}



\abstract{The Schatten-$q$ quasi-norm is a widely used nonconvex rank surrogate and matrix factorization is an effective approach to reduce computational cost. In this paper, we consider the equivalent group-sparse factorized reformulation of Schatten-$q$ norm regularized low-rank matrix recovery problem. Though this factorized model exhibits favorable performance, two issues remain: (i) the error bound of critical points is unexplored; (ii) the proximal operator of $\|\cdot\|_2^q$ lacks a closed-form solution for general $q$, limiting algorithms to adopt fixed $q$ like $1/2$ or $2/3$. This paper addresses both issues. We investigate the properties of critical points for the factorized problem and show that, compared to nuclear norm, the Schatten-$q$ norm implicitly endows critical points with column orthogonality. From this insight, we introduce the notion of S-critical points under mild conditions that ensure column orthogonality with easily operable criterion for identifying. We show that global optimal points must be S-critical points and we derive an error bound between S-critical points and the true matrix. We further present an inexact proximal alternating linearized minimization method for the factorized problem, along with practically computable inexact proximal operator for $\|\cdot\|_2^q$ and criteria to find solutions satisfying inexactness conditions, and we establish the whole sequence convergence and a convergence rate guarantee under Kurdyka–\L ojasiewicz condition. Moreover, we prove that the factorized model with least-squares loss has KL exponent $1/2$ at S-critical points, then the iteration converges linearly under suitable condition. Extensive numerical experiments validate the effectiveness of our algorithm and confirm the theoretical properties of the factorized model.}

\keywords{Factorized low-rank matrix recovery problem, Schatten-$q$ quasi-norm, Error bound for critical point, Kurdyka-\L ojasiewicz property, Inexact proximal alternating linearized minimization}


\pacs[MSC Classification]{90C26, 49J52, 15A60, 65K05}

\maketitle
\section{Introduction}\label{Introduction}
Low-rank matrix recovery is a fundamental problem in signal and image processing, machine learning, statistics, recommendation systems, quantum state tomography, etc, \cite{overview6,overview1,nuclear,overview3,overview4}. The objective is to reconstruct an underlying low-rank matrix $X^*$ from a limited number of linear observations, often corrupted by noise or outliers. When the true rank $r^*$ (or a tight estimate $r$) is known, the problem can be formulated as the rank-constrained model
\begin{equation}\label{rank-constrained}
\min_{X\in \mathbb{R}^{m\times n}} \bigl\{F(X) \;\; \text{s.t.} \;\; \text{rank}(X)\le r\bigr\},
\end{equation}
where $F$ is a loss function. In many scenarios, the true rank is unknown, one may consider the rank-regularized model
\begin{equation}\label{rank-regularized}
\min_{X\in \mathbb{R}^{m\times n}} \bigl\{F(X) + \lambda\,\text{rank}(X)\bigr\},
\end{equation}
which promotes low-rank solutions with a suitable $\lambda$. However, it is NP-hard to solve \eqref{rank-regularized}.  An effective approach is to adopt the convex relaxation technique. One popular way is the nuclear norm \cite{cai2010singular, nuclear, ma2011fixed, zhang2010nuclear}, the tightest convex surrogate of the rank function, leading to the convex problem
\begin{equation}
\min_{X\in \mathbb{R}^{m\times n}} \bigl\{F(X) + \lambda\,\|X\|_*\bigr\}.
\end{equation}
Beyond the nuclear norm, various non-convex surrogates have been proposed, e.g., truncated nuclear norm \cite{nonconvex3}, capped-$\ell_1$ \cite{yu2022smoothing}, minimax concave penalty \cite{nonconvex7}, Schatten-$q$ norm \cite{sq1,sq2}, logarithmic norm \cite{nonconvex4}, (truncated) $\ell_{1-2}$ \cite{ge2022dantzig,ma2017truncated}, among others \cite{nonconvex9}. In particular, a commonly used non-convex surrogate is the Schatten-$q$ norm (see Definition \ref{Schatten-q}), which is a tighter approximation to the rank function than the nuclear norm, yielding the non-convex formulation
\begin{equation}
\min_{X\in \mathbb{R}^{m\times n}} \bigl\{F(X) + \lambda\,\|X\|_{s_q}\bigr\}.
\end{equation}

In practice, solving the aforementioned relaxed models generally requires an economy SVD per iteration, limiting scalability. The Burer–Monteiro factorization \cite{BM} instead optimizes the bi-factor form $UV^T$ to exploit low rank structure, leading many works to reformulate low‑rank recovery models accordingly (see, e.g., references in \cite{overview6}). With an exact rank estimate, the rank‑constrained model (\ref{rank-constrained}) is equivalently transformed into a factorized form:
\begin{equation}\label{rank-constrained1}
\underset{U\in \mathbb{R}^{m\times r},V\in \mathbb{R}^{n\times r}}{\min} F(UV^T).
\end{equation}
When the true rank is unknown, one may use the factorized form of the rank-regularized model or its relaxation. Studies show that such matrix factorizations effectively reduce dimensionality and improve computational efficiency \cite{nonconvex4,BM1}.

The rank function admits the following equivalent factorization, if $\operatorname{rank}(X)\le d$, 
\begin{equation}
\operatorname{rank}(X)= \min_{X=UV^T} \frac{1}{2}\bigl(\|U\|_{2,0}+\|V\|_{2,0}\bigr), ~U \in \mathbb{R}^{m \times d},~ V \in \mathbb{R}^{n \times d},
\end{equation}
where $\|\cdot\|_{2,0}$ denotes the number of nonzero columns of the matrix. Recht et al. \cite{BM2} propose a factorized form of the nuclear norm, if $\text{rank}(X)\le d$,
\begin{equation}\label{nuclear}
\|X\|_*= \min_{X = UV^T} \frac{1}{2}\bigl(\|U\|_F^2+\|V\|_F^2\bigr), ~U \in \mathbb{R}^{m \times d},~ V \in \mathbb{R}^{n \times d}.
\end{equation}
Furthermore, Shang et al. \cite{BM1} extended the factorization to the Schatten-$q$ norm for $q = 1/2$ and $2/3$. Fan et al. \cite{FGSR} then proposed factor group-sparse regularizer equivalent to the Schatten-$q$ norm for $q = \frac{1}{k+1}, \frac{2}{2k+1}$ with $k\in\mathbb{N}_+$. Recently, Jia et al. \cite{jia2020generalized} introduced a general group-sparse factorization form of Schatten-$q$ norm, and then the Schatten-$q$ norm regularized low-rank optimization problem,
\begin{equation}\label{M1}
\underset{X\in \mathbb{R}^{m\times n}}{\min}F(X)+\frac{\lambda}{t}\|X\|_{S_t}^t,
\end{equation}
where $t\in(0,1]$, is equivalent to the group-sparse factorized model, if $\text{rank}(X)\le d$,
\begin{equation}\label{M2}
\underset{U\in \mathbb{R}^{m\times d},V\in \mathbb{R}^{n\times d}}{\min} F(UV^T)+\frac{\lambda}{q} \|U\|_{2,q}^q + \frac{\lambda }{p}\|V\|^p_{2,p},
\end{equation}
for any $q,p\in(0,2]$ satisfying $\frac{1}{t}=\frac{1}{q}+\frac{1}{p}$. Here, the column-wise $q$-norm of a matrix $A \in \mathbb{R}^{k \times d}$ is $\|A\|_{2,q}=
(\sum_{i=1}^{n}\|A_{:,i}\|_2^q)^{\frac{1}{q}}$, where the symbol $A_{:,i}$ stands for the $i$-th column of the matrix $A$. In addition, Yu et al. \cite{yu2025efficient} proposed a flexible group sparse regularizer (FLGSR) that can group an arbitrary number of columns into a single unit.

Factorized models involve fewer variables and avoid per-iteration SVD, and are therefore widely adopted in large-scale low-rank matrix optimization \cite{L2,L3}. However, the bilinear structure of the factorized models induces non-convexity, giving rise to many non-global critical points. Consequently, characterizing the nonconvex geometric landscape and deriving error bounds between critical points (or local minima) and the true matrix (or global optimum) have become important research directions.

Hence, in this paper, we consider the factorized model (\ref{M2}), and study the property and derive the error bound of the critical points.

The nonconvex geometric landscape and strict saddle property of the factorized low-rank problem (\ref{rank-constrained1}) have been extensively studied. For symmetric matrix sensing, Bhojanapalli et al. \cite{g1} showed that $UV^T$ factorization introduces no spurious local minima under incoherent and noiseless measurements, and showed that in noisy or approximately low-rank settings, all local minima lie near a global minimum. Park et al. \cite{Park} extended this to the asymmetric case under the Restricted Isometry Property (RIP). Ge et al. \cite{g2} further proved that for matrix sensing, completion, and robust PCA, all local minima are global and every saddle point has a strictly negative Hessian eigenvalue. This was further unified by Zhu et al. \cite{g5}, who demonstrated that for general factorized low-rank problems, spurious local minima are absent provided the objective function has restricted strong convexity and smoothness property.

While previous work focused exclusively on the unregularized model, recent studies have analyzed properties of critical point and error bounds for factorized models incorporating rank surrogates regularizer. For the factorized model with a rank regularizer plus a tiny nuclear norm term,
\begin{equation}\label{rank-nuclear}
\min_{U\in\mathbb{R}^{m\times d},\,V\in\mathbb{R}^{n\times d}} \bigl\{F(UV^T)+\lambda(\|U\|_{2,0}+\|V\|_{2,0})+\mu(\|U\|_F^2+\|V\|_F^2)\bigr\},
\end{equation}
Tao et al. \cite{r1} proved that under a mild assumption on $F$, any critical point associated with suitable $\lambda$ and tiny $\mu$ recovers the true rank. Li et al. \cite{r2} considered the rank-regularized problem with bound constraints, strengthened the optimality condition for stationary points, and proved equivalence between the factorized formulation and its non-convex relaxation in terms of global minimizers and strong stationary points. For the factorized model with nuclear norm regularizer,
Li et al. \cite{non-convex} demonstrated that, under restricted strong convexity and smoothness, critical points are either globally optimal or possess a strict negative Hessian eigenvalue. Tao et al. \cite{Error_bound} further derived an error bound between critical points and the true matrix for this model when the Hessian is positive semi-definite.

Nonconvex surrogates approximate the rank function more closely than convex proxies such as the nuclear norm. In many practical scenarios, the true rank is unknown, making models with an automatic rank-reduction capability desirable. Although the factorized nuclear norm regularized model enjoys favorable smoothness, its rank-reduction ability remains relatively weak. In contrast, factorized models with nonconvex surrogates, owing to sharp geometric structure near the origin, often achieve stronger rank reduction, but the non-convexity introduces significant challenges in analyzing their critical points. Despite these challenges, the properties and error bounds of critical points in the factorized Schatten-$q$ norm (or other nonconvex surrogates) regularized low-rank matrix recovery problem have been rarely studied. This paper aims to fill this gap to some extent.

For general nonconvex or nonsmooth problems, optimization algorithms often guarantee only subsequential convergence. Establishing global convergence and convergence rates typically requires regularity properties of the objective, such as the Kurdyka–\L ojasiewicz property. In this paper, we show that the factorized Schatten-$q$ norm regularized model with a least-squares loss possesses a KL exponent of $1/2$ at critical points. Moreover, we aim to design a guaranteed algorithm that ensures global convergence and a convergence rate under KL property for the factorized model (\ref{M2}).


Note that the factorized model (\ref{M2}) can be generalized to nonconvex optimization problems of the form
\begin{equation}\label{form1}
\min_{X=(x_1,x_2,\ldots,x_s)} \Phi(X):=F(x_1,x_2,\ldots,x_s) + \sum_{i=1}^s r_i(x_i).   
\end{equation}
This model has a wide range of applications, including $q$-quasi-norm regularized sparse regression problems \cite{lai2013improved}, sparse dictionary learning \cite{aharon2006k}, matrix rank minimization \cite{BM2}, matrix factorization with nonnegativity/sparsity/orthogonality regularization \cite{hoyer2004non,lee1999learning,paatero1994positive}, (nonnegative) tensor decomposition \cite{kolda2009tensor,welling2001positive}, and (sparse) higher-order principal component analysis \cite{allen2012sparse}, etc. Block coordinate decent (or block coordinate update) is a general and widely used method for solving both convex and nonconvex problems of the form (\ref{form1}) with multiple variable blocks \cite{bonettini2011inexact,lin2015accelerated,lu2015complexity,qin2013efficient,richtarik2014iteration,tseng2009coordinate}.

Under a suitable smoothness assumption on $F$, Tseng \cite{tseng2001convergence} proved that every limit point of BCD is a critical point when $\Phi$ is pseudo-convex, and a coordinate-wise minimum when $\Phi$ is quasi-convex and hemivariate. For differentiable block multi-convex $f$ and extended-value convex $r_i$, Xu and Yin \cite{bcd} proposed a generalized BCD framework (covering original, proximal, and prox-linear updates), showed that limiting points satisfy Nash equilibrium conditions, and established global convergence and convergence rates under the KL property. Extending to general differentiable $F$ and proximable (possibly nonconvex or non-differentiable) $r_i$, Xu and Yin \cite{xu2017globally} developed a prox-linear BCD method, proved that all limiting points are critical points, and derived global convergence and convergence rates under KL property. Ahookhosh et al. \cite{ahookhosh2021multi} analyzed global convergence and rates of a Bregman-distance-based prox-linear alternating minimization algorithm, and further incorporated inertial forces \cite{ahookhosh2021block}. Phan et al. \cite{phan2023inertial} presented a general inertial block-coordinate update framework with corresponding global convergence and rate guarantees.

Nevertheless, for BCD-type algorithms, exact solution of subproblems may be unavailable or lack a closed form, motivating inexact algorithms \cite{chouzenoux2016block,daneshmand2015hybrid,razaviyayn2014parallel}. For $F$ with a block coordinate Lipschitz gradient and convex $r_i$, Tappenden et al. \cite{tappenden2016inexact} proposed an inexact BCD method with randomized variable selection and provided a complexity guarantees. Chouzenoux et al. \cite{chouzenoux2016block} considered differentiable $F$ with an $L$-Lipschitzian gradient on $\operatorname{dom}\Phi$ and proper lsc $r_i$, proposing a block coordinate variable metric forward–backward algorithm allowing inexact update that ensures whole sequence convergence and a convergence rate. Yang et al. \cite{yang2019inexact} showed subsequential convergence to a critical point for inexact BCD, without requiring multi-convexity of $F$ or block Lipschitz continuity of $\nabla F$. Wang and Song \cite{wang2025hybrid} introduced a hybrid inexact proximal alternating method that, under well‑designed parameter conditions and an implementable stopping criterion, generates a Cauchy sequence with a convergence rate guarantee. Most of these inexact BCD-type algorithms contain inexact criteria with impractical conditions, and \cite{wang2025hybrid} also recognized this point, although they present a practical criterion, it requires the function $r_i$ to be proximable. Moreover, we note that most of these algorithms with whole sequence convergence guarantees rely on the sufficient descent property of the objective function, and therefore their proof methods are not applicable to iteration schemes accelerated by extrapolation.

Therefore, for inexact BCD-type algorithms, there are two important directions that merit further in-depth study: how to ensure, in practice, the accuracy requirement of inexact design on subproblem solves; and how to establish the whole sequence convegence and convergence rate guarantee with iteration schemes accelerated by extrapolation under the KL condition. In this paper, for $r_i$ of the form $\|\cdot\|_2^q$, we provide these theoretical guarantee. 

The main contributions of this paper are summarized as follows.
\begin{itemize}
\item We investigate the properties of critical points for the factorized Schatten-$q$ norm regularized low-rank matrix recovery problem (\ref{M2}) in the case that $q=p$. We show that, in contrast to the nuclear norm (\ref{nuclear}), the factorization of Schatten-$q$ norm implicitly endows critical points with column orthogonality, see Theorem \ref{Lemma10}. With this insight, we introduce the notion of S-critical points (Definition \ref{S-critical}) with mild conditions that ensure column orthogonality and and establish that global minimizer must be S-critical points. We also provide an easily operable criterion for identifying S-critical points, see Remark \ref{Remark}. Furthermore, we derive an error bound (Theorem \ref{error-theorem}) between S-critical points and the truth matrix (or global optimal), thereby addressing a gap in the relatively sparse literature on error bounds for factorized low-rank matrix recovery problems with nonconvex surrogates for rank function. Furthermore, we also present the error bound for two specific model by Theorem \ref{error-theorem}, i.e., matrix sensing problem (Proposition \ref{exampleMS}), and weighted principle component analysis problem (Proposition \ref{examplePCA}).

\item We prove that for the low-rank matrix recovery problem with a least-squares loss function (\ref{p1}), under appropriate conditions on the parameter $\lambda$, the objective function satisfies the KL property with exponent $1/2$ at S-critical points (Theroem \ref{KL}).

\item We provide an approximate proximal operator for 
$\|\cdot\|_2^q$	with arbitrary parameter $q\in(0,2]$ (Proposition \ref{inexactprox}), which is practically computable (Remark \ref{computeprox}), and present an inexact proximal alternating linearlized minimization method (Algorithm \ref{PAMM}), incorporating an automatically rank-adjustment technique and allowing acceleration step in update, for solving the factorized low-rank matrix recovery problem (\ref{M2}). We prove that this inexact algorithm obtain the subsequence convergence guarantee (Theorem \ref{theorem4.1}) and the whole sequence convergence (Theorem \ref{theoremKL}) and convergence rate (Theorem \ref{theoremrate}) guarantee under KL condition. Furthermore, for factorized model with least-square loss, we show that the sequence generated by Algorithm \ref{PAMM} converges linearly under suitable condition (Proposition \ref{Convergence_result}).
\end{itemize}

The remainder of this paper is organized as follows. In Section \ref{section2}, we introduce necessary notation, definitions, and preliminaries. In Section \ref{section3}, we define S-critical points for the low-rank matrix recovery problem (\ref{M2}) in the case that $q=p$ and establish the error bound between S-critical points and the truth matrix. Section \ref{section4} demonstrates that, under certain conditions on $\lambda$, the objective function of problem (\ref{p1}) possesses the KL property with exponent $1/2$ at S-critical points. In Section \ref{section5}, we propose an inexact proximal alternating linearlized minimization method and analyze its convergence properties. Section \ref{section6} reports on numerical experiments that validate the proposed algorithm and support our theoretical findings. Finally, Section \ref{section7} draws a conclusion.

\section{Notation and Preliminaries}\label{section2}
In this paper, we denote $\mathbb{R}^{m \times n}$ as the vector space of all $m \times n$ real matrices, and $\mathbb{C}^{m \times n}$ as the vector space of all $m \times n$ complex matrices, equipped with the trace inner product $\langle X, Y \rangle = \operatorname{trace}(X^* Y)$ for $X, Y \in \mathbb{C}^{m \times n}$ and its induced Frobenius norm $\|\cdot\|_F$. Let $\mathbb{Q}$ denote the set of rational numbers and $\mathbb{N}_+$ the set of positive integers.  Let $I_r$ denote an
identity matrix with dimention $r$. For a matrix $X \in \mathbb{R}^{m \times n}$, denote $\sigma(X)$ as the singular value vector of $X$ arranged in a non-increasing order. $\sigma_k(X)$ means the value of the $k$-th largest singular values of $X$. $\sigma_{\max}(X)$ and $\sigma_{\min}(X)$ mean the largest and the smallest nonzero singular values of $X$, respectively. We denote by $\|X\|_2$, $\|X\|_*$ and $X^\dagger$ the spectral norm, the nuclear norm and the pseudo-inverse of $X$, respectively. We denote $\odot q$ as the element-wise power of order $q$, which means, for a matrix $X=(x_{ij})$,  $X^{\odot q}=(x_{ij}^q)$. We denote the column-wise q-norm of a matrix $X \in \mathbb{R}^{m \times n}$ as $\|X\|_{2,q}=
(\sum_{i=1}^{n}\|X_{:,i}\|_2^q)^{\frac{1}{q}}$, where the symbol $X_{:,i}$ stands for the $i$-th column of the matrix $X$. And $\|X\|_{2,0}$ denotes the number of nonzero columns in $X$.

We have the following commonly used inequality for the inner product of matrices.
\begin{lemma}\label{R1}
Let $A\in \mathbb{R}^{m\times l},B\in \mathbb{R}^{l\times n}$ and $\operatorname{rank}(B)\le k$, then 
\begin{equation}
|\langle A,B \rangle|\le \sqrt{k}\|A\|_2\|B\|_F.
\end{equation}
\end{lemma}

First, we present the definition of the Schatten‑q quasi-norm (see, e.g., \cite{sq1,nie2012low}).
\begin{definition}[Schatten-q quasi-norm]\label{Schatten-q} Suppose that $\sigma_1(A)\ge \sigma_2(A)\ge \ldots\ge \sigma_{d}(A)$ are singular values of matrix $A\in \mathbb{C}^{m\times n}$, where $d=\min\{m,n\}$. Given $q\in(0,1]$, then the Schatten-q quasi-norm of the matrix $A$ is defined as 
\begin{equation}
\|A\|_{S_q}:=\big(\sum_{k=1}^{d}\sigma^q_k(A)\big)^{\frac{1}{q}}.
\end{equation} 
\end{definition}

Next we introduce two operators defined on $2\times 2$ block matrices (see, e.g., \cite{non-convex,Error_bound}):
\begin{equation}
\mathcal{P}_{\text{on}} \left( \begin{bmatrix} A_{11} & A_{12} \\ A_{21} & A_{22} \end{bmatrix} \right) := 
\begin{bmatrix} A_{11} & 0 \\ 0 & A_{22} \end{bmatrix},~\mathcal{P}_{\text{off}} \left( \begin{bmatrix} A_{11} & A_{12} \\ A_{21} & A_{22} \end{bmatrix} \right) := 
\begin{bmatrix} 0 & A_{12} \\ A_{21} & 0 \end{bmatrix},
\end{equation}
for any matrices  $A_{11} \in \mathbb{R}^{m \times m}, A_{12} \in \mathbb{R}^{m \times n}, A_{21} \in \mathbb{R}^{n \times m}, A_{22} \in \mathbb{R}^{n \times n}$. 

For any $W=\begin{bmatrix} U\\ V \end{bmatrix}$, where $U\in \mathbb{R}^{m\times d}$ and $V\in \mathbb{R}^{n\times d}$, donote $\hat{W}=\begin{bmatrix} U\\ -V \end{bmatrix}$. 
Then by direct computation, it holds 
\begin{equation}
\mathcal{P}_{\text{on}}(W_1W_2^T)=\frac{1}{2}(W_1W_2^T+\hat{W_1}\hat{W_2}^T),~~\mathcal{P}_{\text{off}}(W_1W_2^T)=\frac{1}{2}(W_1W_2^T-\hat{W_1}\hat{W_2}^T).
\end{equation}

Concerning this operations $\mathcal{P}_{\text{on}}$ and $\mathcal{P}_{\text{off}}$, we present two useful lemmas \cite{non-convex} that will facilitate the study of critical points in Section \ref{section3}.

\begin{lemma}\label{o1}
Let $U\in \mathbb{R}^{m\times d}$ and $V\in \mathbb{R}^{n\times d}$, satisfying $U^TU=V^TV$, denote $W=\begin{bmatrix} U\\ V \end{bmatrix}$, then for any $\Delta=\begin{bmatrix} \Delta_U\\ \Delta_V \end{bmatrix}$ for proper dimension, it holds \begin{equation}
\|\mathcal{P}_{\text{on}}(\Delta W^T)\|_F=\|\mathcal{P}_{\text{off}}(\Delta W^T)\|_F.
\end{equation}
\end{lemma}

\begin{lemma}\label{o2}
Let $U_1,U_2\in \mathbb{R}^{m\times d}$ and $V_1,V_2\in \mathbb{R}^{n\times d}$, satisfying $U_1^TU_1=V_1^TV_1$ and $U_2^TU_2=V_2^TV_2$, denote $W_1=\begin{bmatrix} U_1\\ V_1 \end{bmatrix}$ and $W_2=\begin{bmatrix} U_2\\ V_2\end{bmatrix}$, then it holds \begin{equation}
\|\mathcal{P}_{\text{on}}(W_1W_1^T-W_2W_2^T)\|_F\le\|\mathcal{P}_{\text{off}}(W_1W_1^T-W_2W_2^T)\|_F.
\end{equation}
\end{lemma}

Next, we present the definitions of the limiting Fréchet subdifferential and the critical point \cite{kruger2003frechet, Rockafellar2009}, as well as of Restricted Strong Convexity (RSC) and Restricted Strong Smoothness (RSS). These restricted curvature conditions are standard requirements for loss functions in the context of low-rank matrix recovery problem (see, e.g., \cite{g5,zhu2021global,non-convex,Error_bound}).

\begin{definition}[Limiting Fréchet subdifferential]
A vector $g$ is a Fréchet subgradient of a lower semicontinuous function $\Phi$ at $x \in \operatorname{dom}(F)$ if
\begin{equation}
\liminf_{y \to x, y \neq x} \frac{\Phi(y) - \Phi(x) - \langle g, y - x \rangle}{\|y - x\|} \geq 0.
\end{equation}
The set of Fréchet subgradients of $\Phi$ at $x$ is called the Fréchet subdifferential and denoted as $\hat{\partial}\Phi(x)$. If $x \notin \operatorname{dom}(\Phi)$, then $\hat{\partial}\Phi(x) = \emptyset$. The limiting Fréchet subdifferential is denoted by $\partial \Phi(x)$ and defined as
\begin{equation}
\partial \Phi(x) = \bigl\{ g : \exists~ x_m \to x,~ g_m \in \hat{\partial}\Phi(x_m)~ \text{s.t.}~ g_m \to g \bigr\}.  
\end{equation}  If $F$ is differentiable at $x$, then $\partial \Phi(x) = \hat{\partial}\Phi(x) = \{\nabla \Phi(x)\}$.
\end{definition}

\begin{definition}
A point $x^*$ is called a critical point of $F$ if $0\in \partial \Phi(x^*)$.
\end{definition}

\begin{definition}
A twice continuously differentiable function $F:\mathbb{R}^{m\times n}\to \mathbb{R}$ is said to satisfy
the $(r_1,r_2)$-RSC of modulus $\alpha$ and the $(r_1,r_2)$-RSS of modulus $\beta$, respectively, where $0<\alpha\le\beta$, if it holds
\begin{equation}
\alpha\|H\|_F^2\le \nabla^2F(X)(H,H)\le \beta\|H\|_F^2,
\end{equation} for any $X, H \in \mathbb{R}^{m\times n}$ with $\operatorname{rank}(X)\le r_1$ and $\operatorname{rank}(X)\le r_2$.
\end{definition}

This work relies on the following restricted well-conditioned property to establish an error bound for the critical points. 
\begin{condition}\label{Condition1}
Function $F$ has the $(2r,4r)$-RSC of modulus $\alpha$ and the $(2r,4r)$-RSS of modulus $\beta$, with $\beta/\alpha \le 1.38$.
\end{condition}

If function $F$ holds the RSC and RSS property, then the following result holds \cite{non-convex}:
\begin{lemma}\label{RSC}
Suppose that the twice continuously differentiable function $F$ has the $(2r,4r)$-RSC of modulus $\alpha$ and the $(2r,4r)$-RSS of modulus $\beta$. Then 
\begin{equation}
\Big|\frac{2}{\alpha+\beta}\nabla^2F(X)(G,H)-\langle G,H\rangle\Big|\le \frac{\beta-\alpha}{\alpha+\beta}\|G\|_F\|H\|_F,
\end{equation} holds for any matrices $X, G, H $ of rank at most $2r$.
\end{lemma}

In this paper, we denote $X^*\in \mathbb{R}^{m \times n}$ as the true matrix of the low-rank matrix recovery problem, and $\operatorname{rank}(X^*)=r$. Let 
\begin{equation}
\mathcal{E}^*=\{(P\Sigma^{\frac{1}{2}},Q\Sigma^{\frac{1}{2}})|X^*=P\Sigma Q\},
\end{equation} where $P\Sigma Q$ is the singular value decomposition of $X^*$ and $\Sigma\in \mathbb{R}^{r \times r}$. We can find that $\Sigma^{\frac{1}{2}}P^TP\Sigma^{\frac{1}{2}}=I_r=\Sigma^{\frac{1}{2}}Q^TQ\Sigma^{\frac{1}{2}}$. Therefore, all elements in $\mathcal{E}^*$ satisfy the conditions of Lemma \ref{o1} and Lemma \ref{o2}.

In the following, we introduce the Kurdyka-\L ojasiewicz (KL) property \cite{KL} and the notion of the KL exponent.
\begin{definition}[Kurdyka-\L ojasiewicz property]
Let  $f: \mathbb{R}^n \to \mathbb{R} \cup \{+\infty\}$ be a proper lower semicontinuous function. The function $f$ is said to have the Kurdyka-\L ojasiewicz property at a point $\bar{x} \in \operatorname{dom} \partial f$ if there exist $\eta \in (0, +\infty]$, a neighborhood $U=\{x~|~\|x-\bar{x}\|_2\le \rho\}$ of $\bar{x}$, and a continuous concave function $\varphi: [0, \eta) \to [0, +\infty)$ such that:
\begin{enumerate}
\item $\varphi(0) = 0$, and $\varphi$ is $C^1$ on $ (0, \eta) $ with $ \varphi'(s) > 0 $ for all $ s \in (0, \eta) $;
\item For every $ x \in U \cap \{x: f(\bar{x}) < f(x) < f(\bar{x}) + \eta\} $, the following inequality holds: 
\begin{equation}
\varphi'\bigl(f(x) - f(\bar{x})\bigr) \cdot \operatorname{dist}\bigl(0, \partial f(x)\bigr) \geq 1, 
\end{equation} where $ \operatorname{dist}(0, \partial f(x)) = \inf\{\|v\| : v \in \partial f(x)\} $.
\end{enumerate}
If the function $\varphi$ in the KL property can be chosen of the form $\varphi(s) = \mu s^{1-\theta}$ for some $\mu > 0$ and $\theta \in [0, 1)$, then $f$ is said to have the KL property at $\bar{x}$ with exponent $\theta$.
\end{definition}


\section{Critical Points and Error Bound}\label{section3}
The factorized Schatten-$q$ norm regularized low-rank matrix recovery problem is nonconvex, and one can generally expect to obtain a critical point in practice, which promotes us to study the property and error bound of critical points. Therefore, in this section, we investigate the properties of the critical points of  model (\ref{M2}) in the case $q=p$, and we further introduce the notion of S-critical points under mild conditions, which guarantee column orthogonality (in the sense that there exists an equivalent column-orthogonal critical point). Finally, we establish an error bound between such S-critical points and the true matrix. And we also present the error bound for two specific model, i.e., matrix sensing problem, and weighted principle component analysis problem.

Consider the factorized model (\ref{M2}) with $q=p$, where $q\in(0,2)$,

\begin{equation}\label{M3}
\underset{U\in \mathbb{R}^{m\times d},V\in \mathbb{R}^{n\times d}}{\min} \Phi_\lambda (U,V)
 :=F(UV^T)+\frac{\lambda}{q}(\|U\|_{2,q}^q+\|V\|_{2,q}^q).
\end{equation}
Note that the model (\ref{M3}) is equivalent to 
\begin{equation}\label{M4}
\underset{X\in \mathbb{R}^{m\times n}}{\min} \bar{\Phi}_\lambda (X):=F(X)+\frac{2\lambda}{q}\|X\|_{S_{\frac{q}{2}}}^{\frac{q}{2}}.
\end{equation}.

\subsection{Property of Critical Points and S-critical Points}

Given $A\in \mathbb{R}^{k\times d}$ and $q\in (0,2)$, define the diagonal matrix $D_A$ by $$(D_A)_{ii}=\begin{cases}
\|A_{:,i}\|_2^{q-2}, &\text{if}~ \|A_{:,i}\|_2\ne 0,  \\ 
0, &\text{if} ~ \|A_{:,i}\|_2= 0. 
\end{cases}$$
By direct calculation, if $1<q<2$, then the gradient of $\Phi_\lambda (U,V)$ at $(U, V)$ takes the form of 
\begin{equation}
\nabla\Phi_\lambda (U,V)=\begin{bmatrix}
\nabla_U\Phi_\lambda (U,V)\\
\nabla_V\Phi_\lambda (U,V)
\end{bmatrix}=\begin{bmatrix}
\nabla F(UV^T)V+\lambda UD_U\\
\nabla F(UV^T)^TU+\lambda VD_V
\end{bmatrix}.
\end{equation}
If $0<q\le1$, the limiting Fréchet subdifferential of $\Phi_\lambda (U,V)$ at $(U, V)$ takes the form of 
\begin{equation}
\partial \Phi_\lambda (U,V)=\begin{bmatrix}
\nabla F(UV^T)V+\lambda UD_U \\
\nabla F(UV^T)^TU+\lambda VD_V
\end{bmatrix}+\begin{bmatrix}
\mathcal{B}(U) \\
\mathcal{B}(V)
\end{bmatrix},
\end{equation}
where given $A\in\mathbb{R}^{k\times d}$
$$(\mathcal{B}(A))_{:,i}=\begin{cases}
\mathbb{B}^k\mathbf{1}_{\{A_{:,i}=0\}}, &\text{if} ~ q=1,  \\ 
\mathbb{R}^k\mathbf{1}_{\{A_{:,i}=0\}}, &\text{if} ~ 0<q<1, 
\end{cases}$$
here $\mathbb{B}^k$ is the unit ball in $\mathbb{R}^k$ and $\mathbf{1}$ is the indicator function.
Note that when $0< p \le 1$, and when a certain column of $U$ or $V$ is zero, the function $\Phi_\lambda$ is not differentiable at that point; in the remaining cases, it is differentiable.

We employ an adaptive rank adjustment technique in Algorithm \ref{PAMM} in the subsequent experimental section. Specifically, if a column in the matrices $U$ or $V$ is zero, we remove the corresponding columns from $U$ and $V$. The following lemma ensures that if the original $(U,V)$ are critical points, then the new $(U,V)$ remain critical points as well. Consequently, in the analysis below, we may restrict our attention to such $U$ and $V$ where every column is a nonzero vector.

\begin{lemma}\label{L3.1}
Let $(U,V)$ be a critical point of $\Phi_\lambda$.  Denote $J_{U}=\{~i~|~\|U_{:,i}\|_2 \ne 0\}$ and $J_{V}=\{~i~|~\|V_{:,i}\|_2 \ne 0\}$. Then $J_U=J_V$ and $(U|_{J_U},V|_{J_V})$ is also a critical point of $\Phi_\lambda$.
\end{lemma}
\begin{proof}
For any $i\in J_U$, it holds $$\nabla_{U_{:,i}} \Phi_\lambda
(U,V)=\nabla F(UV^T)V_{:,i}+ \lambda \|U_{:,i}\|_2^{p-2}U_{:,i}=0.$$ Since for $i\in J_U$, we have $\lambda \|U_{:,i}\|_2^{p-2}U_{:,i}\ne 0$, then $V_{:,i}\ne 0$. This implied $i\in J_V$ and $J_U \subseteq J_V$. Similarly, we have $J_V \subseteq J_U$. This implies $J_U=J_V$. 

Denote $\bar{U}=U|_{J_U}$ and $\bar{V}=V|_{J_V}$. Note that $UV^T=U|_{J_U}(V|_{J_V})^T$. Then it holds 
\begin{align*}
\nabla F(\bar{U}\bar{V}^T)\bar{V}+ \lambda \bar{U}D_{\bar{U}}=\nabla F(UV^T)V|_{J_U}+ 
\lambda U|_{J_U}D_{U|_{J_U}} =0,\\
\nabla F(\bar{U}\bar{V}^T)^T\bar{U}+ \lambda \bar{V}D_{\bar{V}}=\nabla F(UV^T)^TU|_{J_V}+ 
\lambda V|_{J_V}D_{V|_{J_V}} =0.
\end{align*}
This shows that $(U|_{J_U},V|_{J_V})$ is also a critical point of $\Phi_\lambda$.
\end{proof}

Next, we present an important property that the critical points of $\Phi_\lambda $ possess.

\begin{theorem}\label{Lemma10}
Any critical point of $\Phi_\lambda$ belongs to the set
\begin{equation}
\mathcal{E}_\lambda :=\left \{ (U,V)\in \mathcal{T}(\mathbb{R}^{m\times d})\times 
\mathcal{T}(\mathbb{R}^{n\times d})~|~U^TU=V^TV \right \},
\end{equation}
where \begin{equation}
\mathcal{T}(\mathbb{R}^{k\times d}):=\left \{ A\in \mathbb{R}^{k\times d}~|~A_{:,i}^TA_{:,j}=0, ~~\text{if}~~   \|A_{:,i}\|_2\ne \|A_{:,j}\|_2 ~~\text{and}~~ i\ne j \right \}.
\end{equation}
\end{theorem}
\begin{proof}
Let $(U,V)$ be a critical point of $\Phi_\lambda$. By Lemma \ref{L3.1}, the index sets of the nonzero columns of $U$ and $V$ are the same, denoted by $J$. Moreover, we note that it suffices to verify that $U|_J$ and $V|_J$ satisfy the conclusion of the theorem. To this end, we may assume without loss of generality that  $U$ and $V$  have no nonzero columns; then the function $\Phi_\lambda$ is differentiable at $(U,V)$. Then, it holds 
$$\nabla F(X)V+\lambda U D_U=0, \nabla F(X)^TU+\lambda V D_V=0.$$
This implies $$\lambda U^TUD_U=-U^T\nabla F(X)V=\lambda D_VV^TV.$$
Since $$\|U_{:,i}\|_2^{q}=(U^TUD_U)_{ii}=(D_VV^TV)_{ii}=\|V_{:,i}\|_2^{q},$$ then $D_U=D_V$. Denote $$D=D_U=D_V=\operatorname{diag}(\delta_1,\ldots,\delta_d),$$ we have $D^{-1}U^TUD=V^TV$. This implies $$\frac{\delta_j}{\delta_i}(U^TU)_{ij}=(V^TV)_{ij}=(V^TV)_{ji}=\frac{\delta_i}{\delta_j}(U^TU)_{ji}=\frac{\delta_i}{\delta_j}(U^TU)_{ij}.$$ Then, for $i\ne j $, if $\delta_i\ne \delta_j$, it holds $(U^TU)_{ij}=(V^TV)_{ij}=0$. And if $\delta_i= \delta_j$, it holds $(U^TU)_{ij}=(V^TV)_{ij}$. This implies $ (U,V)\in \mathcal{E}_\lambda$.
\end{proof}


For the factorized nuclear norm regularized low‑rank matrix problem, a critical point $(U,V)$ only satisfies $U^TU = V^TV$, (see, e.g., \cite{Error_bound}). Theorem \ref{Lemma10} shows that if the Schatten-$q$ norm regularizer is employed, the critical points possess a stronger property, namely column‑wise orthogonality among columns with distinct $\ell_2$‑norms.

\begin{corollary}\label{Corollary3.1}
Let $(U,V)$ be a critical point of $\Phi_\lambda$. If $\|U_{:,i}\|_2$ are distinct for all nonzero column index $i$, then $UU^T=VV^T$, which is a diagonal matrix. Furthermore, it holds $$\|UV^T\|^\frac{q}{2}_{s_{\frac{q}{2}}}=\frac{1}{q}(\|U\|_{2,q}^q+\|V\|_{2,q}^q).$$ In this case, we have $\Phi_\lambda(U,V)=\bar{\Phi}_\lambda(UV^T)$.
\end{corollary}


To further exploit the column orthogonality induced by the Schatten-$q$ norm, we introduce the following notion of S-critical points that a critical point is called an S-critical point if there exists an equivalent column-orthogonal critical point. We show that the global minimizer must be a S-critical point.

\begin{definition}\label{S-critical}
Let $(U,V)$ be a critical point of $\Phi_\lambda$. If there exists a column-orthonormal pair $(\bar{U},\bar{V})$ satisfying $UV^T = \bar{U}(\bar{V})^T$ such that $(\bar{U},\bar{V})$ is also a critical point of $\Phi_\lambda$, then $(U,V)$ is called an S-critical point of $\Phi_\lambda$. 
\end{definition}

From Corollary \ref{Corollary3.1}, we obtain the following sufficient condition for a critical point to be S-critical.

\begin{proposition}
Let $(U,V)$ be a critical point of $\Phi_\lambda$. If $\|U_{:,i}\|_2$ are distinct for all nonzero
column index $i$, then $(U,V)$ is an S-critical point of $\Phi_\lambda$.
\end{proposition}

In the following, we also present an equivalent characterization of S-critical points, and prove that any global minimizer is necessarily S-critical. The proof of Theorem \ref{Theorem4.4} can be found in Appendix \ref{APTheorem4.4}.

\begin{theorem}\label{Theorem4.4}
Let $(U,V)$ be a critical point of $\Phi_\lambda$. Then there exists a matrix $E\in \mathbb{R}^{d\times d}$ such that $UV^T=\bar{U}\bar{V}^T$ and both $\bar{U}$ and $\bar{V}$ are column-wise orthogonal, where $(\bar{U},\bar{V})=(UE,VE)$. Furthermore, either $(\bar{U},\bar{V})$ is also a critical point of $\Phi_\lambda$, or $(U,V)$ is not a global minimizer of $\Phi_\lambda$. Moreover, $(\bar{U},\bar{V})$ is a critical point of $\Phi_\lambda$ if and only if $D_U = D_{\bar{U}}$.
\end{theorem}

Furthermore, we provide an easily operable equivalent criterion for determining whether a critical point is S-critical.
\begin{remark}\label{Remark}
Let $(U,V)$ be a critical point of $\Phi_\lambda$. Let $\hat{U}\hat{\Sigma}\hat{V}^T$ be the singular value decomposition of $UV^T$, and define $\bar{U} = \hat{U}\hat{\Sigma}^{\frac{1}{2}}$ and $\bar{V} = \hat{V}\hat{\Sigma}^{\frac{1}{2}}$. Then to determine whether a given critical point $(U,V)$ is an S-critical point, we have a straightforward operational criterion: it suffices to check whether $(\bar{U},\bar{V})$ is a critical point, i.e., to verify that $0\in \partial \Phi_\lambda(\bar{U},\bar{V})$.
\end{remark}

In what follows we will focus exclusively on studying S-critical points, since the global minimizer must be a S-critical points and it is natural and reasonable to require that $(\bar{U},\bar{V})$ also be a critical point. Moreover, by exploiting the properties of S-critical points, it suffices to study column‑orthonormal pairs $(U,V)$.

\subsection{Error Bound for S-critical Points}
In this subsection, we provide an error bound between the S-critical points and the true matrix for the model (\ref{M3}). Let $W=(U; V)$ be an S-critical point and $X^*=U^*(V^*)^T$ be the truth matrix.  Denote $W^*=(U^*; V^*)$. Note that the off-diagonal blocks of the matrix $WW^T - W^*(W^*)^T$ correspond precisely to $UV^T - U^*(V^*)^T$, then to establish an error bound between S-critical point $(U,V)$ and the truth matrix $X^*=U^*(V^*)^T$ (or optimal solution), denoted as $\|UV^T - U^*(V^*)^T\|_F$, we can estimate the upper bound of $\|WW^T - W^*(W^*)^T\|_F$. Let $Q$ be a column-orthogonal matrix that spans the range of $W$, such that, $\operatorname{span}(\operatorname{col}(Q))=\operatorname{span}(\operatorname{col}(W))$ and $Q^TQ$ is an identity matrix. Then, we introduce the term $(WW^T - W^*(W^*)^T)QQ^T$, which represents the projection of the error matrix $WW^T - W^*(W^*)^T$ onto the range of $W$.

To quantify the discrepancy between two matrices, we utilize the distance as follows. Let $W_1, W_2 \in \mathbb{R}^{k \times d}$, the distance is defined as (see, e.g., \cite{zhu2021global,g5, g2,non-convex,Error_bound}):
\begin{equation}
    d(W_1, W_2) = \min_{R_1, R_2 \in \mathbb{O}^d} \|W_1R_1 - W_2R_2\|_F = \min_{R \in \mathbb{O}^d} \|W_1 - W_2R\|_F.
\end{equation}

Naturally induced by the distance, we introduce a specific direction $\Delta$ (see, e.g., \cite{zhu2021global,g5,g1, g2,non-convex,Error_bound}), which serves as a promising descent direction. Here we make a minor generalization by allowing the column dimensions $d$ and $r$ of $W$ and $W^*$, respectively, to differ. Specifically, let
\begin{equation} \label{Delta}
    \Delta := W - W^*R^*,
\end{equation}
where when $d\le r$, $R^*$ denote the matrix consisting of the first $k$ columns of $\bar{R} \in \arg\min_{R \in \mathbb{O}^r} \|[W \ \ 0] - W^*R\|_F$, and when $d\ge r$,
$R^*$ denote the matrix consisting of the first $r$ rows of $\bar{R}\in \arg\min_{R \in \mathbb{O}^d} \|W - [W^* \ \ 0]R\|_F$.
This direction is instrumental in examining whether the Hessian matrix at the current point possesses potential negative eigenvalues.

The relationship between the three critical matrices identified above, the direction $\Delta$, the projected error $(WW^T - W^*(W^*)^T)QQ^T$, and the total error $WW^T - W^*(W^*)^T$, which play an important role in the error bound analysis for the factorized low-rank matrix recovery problems \cite{zhu2021global,g5,g1, g2,non-convex,Error_bound,Park,li2020non}, is established in the following inequality, which can be deduced from {\cite[Lemma 5]{non-convex}}.
\begin{equation}\label{connect1}
  \|W\Delta^T\|_F^2\leq\frac{1}{8} \|WW^T-W^*{W^*}^T\|_F^2
  +\frac{7+\sqrt{2}}{2}\big\|(WW^T-W^*{W^*}^T)QQ^T\big\|_F^2.
 \end{equation}

Consequently, to bound the total error $\|WW^T - W^*(W^*)^T\|_F$, it is necessary to derive an upper bound for $\|(WW^T - W^*(W^*)^T)QQ^T\|_F$ and a lower bound for $\|W\Delta^T\|_F$, both related to $\|WW^T - W^*(W^*)^T\|_F$. To this end, we propose the following two lemmas to provide these estimates. The proof draws on the methods used in the proofs of Papers mentioned above in this subsection.

In the theorems and lemmas below, for a given S-critical point $(U,V)$, let $W=(U;V)$ and $X = UV^T$, $Q$ is an arbitary column-orthogonal matrix that spans the range of $W$, such that, $\operatorname{span}(\operatorname{col}(Q))=\operatorname{span}(\operatorname{col}(W))$ and $Q^TQ$ is an identity matrix; for a given $(U^*,V^*)\in\mathcal{E}^*$, we set $W^*=(U^*;V^*)$ and $X^*=U^*(V^*)^T$. 

We provide an upper bound for $\|(WW^T - W^*(W^*)^T)QQ^T\|_F$ in the following lemma.
\begin{lemma}\label{upperbound}
Suppose that function $F$ satisfies Condition \ref{Condition1}. Let $(U,V)$ be an S-critical point of $\Phi_{\lambda}$ and fix any $(U^*,V^*)\in\mathcal{E}^*$. If ${\rm rank}(X)\leq r$ and $\sigma_{\min}(X)\ge \frac{1}{10}\sigma_{\min}(X^*)$, then it holds
\begin{align*}
\|YQQ^{T}\|_F^2\le \gamma_1\|Y\|_F^2+\gamma_2 r\|\nabla F(X^*)\|_2^2+\lambda(\gamma_3\sigma^{\frac{q}{2}-1}_{\min}(X^*)\|X^*\|_*-\gamma_4\|X\|_{S_{q/2}}),    
\end{align*}
where $Y=WW^T-W^*{W^*}^T$ and $\gamma_1=\frac{32(\beta -\alpha)^2}{15(\alpha+\beta)^2}$, $\gamma_2=\frac{4096}{15(\alpha+\beta)^2}$, $\gamma_3=\frac{256}{15(\alpha+\beta)}10^{1-\frac{q}{2}}$, $\gamma_4=\frac{256}{15(\alpha+\beta)}$ are constant depend only on the parameter $\alpha$, $\beta$ and $q$.
\end{lemma}
\begin{proof}
Since $(U,V)$ is a S-critical point of $\Phi_{\lambda}$, it holds 
$\langle \nabla\Phi_{\lambda}(U,V),Z\rangle=0$,
for any $Z=(Z_U;Z_V)$, where $(Z_U,Z_V)\in\mathbb{R}^{m\times d}\times\mathbb{R}^{n\times d}$. Denote $D=D_U(=D_V)$, then
\begin{align*} 
0=&\langle \nabla F (X)V+\lambda UD,Z_U\rangle+\langle \nabla F (X)^TU+\lambda VD,Z_V\rangle
\\=&\langle \nabla F (X)-\nabla F (X^*),Z_UV^T+UZ_V^T\rangle+\langle \nabla F (X^*),Z_UV^T+UZ_V^T\rangle
\\&+\langle \lambda D,U^TZ_U+V^TZ_V\rangle
\\=&\int_0^1 \nabla^2F(tX+(1-t)X^*)(X-X^*,Z_U V^T+ U Z_V^T)dt
\\& +\langle \nabla F (X^*),Z_UV^T+UZ_V^T\rangle+\langle \lambda D,U^TZ_U+V^TZ_V\rangle
\end{align*} 
Since ${\rm rank}(X-X^*)\le 2r$ and ${\rm rank}(Z_U V^T+ U Z_V^T)\le 2r$, by Lemma \ref{RSC}, we have
\begin{align*}
  &\Big|\frac{2}{\alpha+\beta}\nabla^2F(tX+(1-t)X^*)(X-X^*,Z_U V^T+ U Z_V^T)\\&-
  \langle X-X^*, Z_U V^T+ U Z_V^T\rangle\Big|\\ &
  \leq  \frac{\beta-\alpha}{\alpha+\beta}\big\|UV^T-U^*{V^*}^T\big\|_F
  \big\|Z_U V^T+UZ_V^T\big\|_F.
\end{align*}
It follows that
\begin{align*}\label{II}
  \underbrace{\langle X-X^*,Z_U V^T+UZ_V^T\rangle}_{I_1}
  &\le -\frac{2}{\alpha+\beta}\big(\underbrace{\langle \nabla F (X^*),Z_UV^T+UZ_V^T\rangle}_{I_2}\\&+\underbrace{\langle \lambda D,U^TZ_U+V^TZ_V\rangle}_{I_3}\big)\\&+
  \frac{\beta-\alpha}{\alpha+\beta}\underbrace{\|X-X^*\|_F}_{I_4}\underbrace{\|Z_U V^T+UZ_V^T\|_F}_{I_5}
\end{align*}
Now take $$Z=Y(W^{T})^{\dagger}=(WW^T-W^*(W^*)^T)(W^{T})^{\dagger}.$$ Since the column orthonormal matrix $Q$ spans the subspace ${\rm col}(W)$, then $$(W^{T})^{\dagger}W^{T}=QQ^{T}.$$

We now estimate the terms $I_1, I_2, I_3, I_4, I_5$.
For the term $I_1$, we have
\begin{align*}
I_1 &=\langle UV^T-U^*(V^*)^T,Z_U V^T+UZ_V^T\rangle=\langle \mathcal{P}_{\rm off}(WW^T-W^*(W^*)^T), ZW^T \rangle
\\&=\frac{1}{2} \langle WW^T-W^*(W^*)^T, ZW^T \rangle
-\frac{1}{2} \langle \hat{W}\hat{W}^T-\hat{W}^*(\hat{W}^*)^T, ZW^T \rangle
\\&\ge \frac{1}{2} \langle WW^T-W^*(W^*)^T, ZW^T \rangle
\\&=\frac{1}{2} \langle WW^T-W^*(W^*)^T, (WW^T-W^*(W^*)^TQQ^T \rangle
\\&= \frac{1}{2}\|(WW^T-W^*(W^*)^T)Q\|_F^2=\frac{1}{2}\|(WW^T-W^*(W^*)^T)QQ^{T}\|_F^2\\&=\frac{1}{2}\|YQQ^{T}\|_F^2,
\end{align*}
where the inequality follows from
\begin{align*}
&\langle \hat{W}\hat{W}^T-\hat{W}^*(\hat{W}^*)^T, ZW^T \rangle
\\=&\langle \hat{W}\hat{W}^T-\hat{W}^*(\hat{W}^*)^T, (WW^T-W^*(W^*)^T(W^{T})^{\dagger}W^T \rangle
\\=&\langle \hat{W}\hat{W}^TW-\hat{W}^*(\hat{W}^*)^TW, (WW^T-W^*(W^*)^T(W^{T})^{\dagger} \rangle
\\=& -\langle \hat{W}^*(\hat{W}^*)^TW, (WW^T-W^*(W^*)^T(W^{T})^{\dagger} \rangle
\\=& -\langle \hat{W}^*(\hat{W}^*)^TW, WW^T(W^{T})^{\dagger} \rangle+\langle \hat{W}^*(\hat{W}^*)^TW, W^*(W^*)^T(W^{T})^{\dagger} \rangle
\\=& -\langle \hat{W}^*(\hat{W}^*)^T, WW^T(W^{T})^{\dagger}W\rangle+\langle (W^*)^T\hat{W}^*(\hat{W}^*)^TW, (W^*)^T(W^{T})^{\dagger} \rangle
\\=& -\langle \hat{W}^*(\hat{W}^*)^T, WW^T\rangle\le 0
\end{align*}

Denote $\Gamma(X^*)=\begin{bmatrix}
O  &  \nabla F(X^*)\\
\nabla F(X^*)^T  & O
\end{bmatrix}$.
Then, for the term $I_2$, it holds that
\begin{equation}\label{I1-ineq}
I_2=\langle \nabla F (X^*),Z_UV^T+UZ_V^T\rangle=\langle \Gamma (X^*),ZW^T\rangle =\langle \Gamma (X^*),YQQ^T\rangle.
\end{equation}
Then by Lemma \ref{R1}, $$|I_2|=|\langle \Gamma (X^*),YQQ^T\rangle|\le \sqrt{r}\|\nabla F (X^*)\|_2\|YQQ^T\|_F.$$

Denote $H=\operatorname{diag}(\|U_{:,1}\|_2,\ldots,\|U_{:,k}\|_2)$ and $H^*=\operatorname{diag}(\|U^*_{:,1}\|_2,\ldots,\|U^*_{:,r}\|_2)$. Let $\widetilde{W}=\frac{1}{\sqrt{2}}WH^{-1}$, and  $\widetilde{W}^*=\frac{1}{\sqrt{2}}W^*(H^*)^{-1}$.
Then $\widetilde{W}^T\widetilde{W}$ and $(\widetilde{W}^*)^T\widetilde{W}^*$ are identity matrices. Note that $W^TW=U^TU+V^TV=2H^2$, then $(W^{T})^{\dagger}=\frac{1}{2}WH^{-2}=\frac{1}{\sqrt{2}}\widetilde{W}H^{-1}$.
For the term $I_3$,   
\begin{align*} 
I_3&=\langle \lambda D,U^TZ_U+V^TZ_V\rangle=\langle \lambda D,W^TZ\rangle=\langle \lambda D,W^TY(W^{T})^{\dagger}\rangle\\&=\langle \lambda D,H\widetilde{W}^TY\widetilde{W}H^{-1}\rangle
=\langle \lambda H^{-1} D H,\widetilde{W}^TY\widetilde{W}\rangle\\&=\langle \lambda D,\widetilde{W}^TY\widetilde{W}\rangle=\langle \lambda D,\widetilde{W}^TWW^T\widetilde{W}\rangle-\langle \lambda D,\widetilde{W}^TW^*(W^*)^T\widetilde{W}\rangle
\\&=2\lambda \langle  D,H^2\rangle-2\lambda\langle  D,\widetilde{W}^T\widetilde{W}^*(H^*)^2(\widetilde{W}^*)^T\widetilde{W}\rangle
\\&\ge 2\lambda (\|X\|_{S_{q/2}}-\|D\|_2\|\widetilde{W}^T\widetilde{W}^*(H^*)^2(\widetilde{W}^*)^T\widetilde{W}\|_*).
\\&\ge 2\lambda (\|X\|_{S_{q/2}}-10^{1-q/2}\sigma^{q/2-1}_{\min}(X^*)\|X^*\|_*).
\end{align*}

For the term $I_4$, 
\begin{align*}
\|UV^T-U^*(V^*)^T\|_F&=\frac{1}{\sqrt{2}}\|\mathcal{P}_{\text{off}}(WW^T-W^*(W^*)^T)\|_F\nonumber
\\&\le \frac{1}{\sqrt{2}}\|WW^T-W^*(W^*)^T\|_F.
\end{align*}

For the term $I_5$,  by Lemma \ref{o1}
\begin{align}
I_5&=\|Z_U V^T+UZ_V^T\|_F\leq\sqrt{2}\|\mathcal{P}_{\rm off}(ZW^T)\|_F=\|ZW^{T}\|_F\nonumber=\|YQQ^T\|_F.
\end{align}

Now combining the above inequalities for items $I_1-I_5$ yields 
\begin{align*}
\frac{1}{2}\|YQQ^{T}\|_F^2&\le \frac{\beta -\alpha}{\sqrt{2}(\alpha+\beta)}\|Y\|_F\|YQQ^T\|_F+\frac{2\sqrt{r}}{\alpha+\beta}\|\nabla F(X^*)\|_2\|YQQ^T\|_F\\&
+\frac{4\lambda}{\alpha+\beta} (10^{1-q/2}\sigma^{q/2-1}_{\min}(X^*)\|X^*\|_*-\|X\|_{S_{q/2}})\\&
\le \frac{(\beta -\alpha)^2}{2(\alpha+\beta)^2}\|Y\|_F^2+\frac{1}{4}\|YQQ^T\|_F^2+\frac{64r}{(\alpha+\beta)^2}\|\nabla F(X^*)\|_2^2\\&+\frac{1}{64}\|YQQ^T\|_F^2+
\frac{4\lambda}{\alpha+\beta} (10^{1-q/2}\sigma^{q/2-1}_{\min}(X^*)\|X^*\|_*-\|X\|_{S_{q/2}}).
\end{align*}
This implies \begin{align*}
\|YQQ^{T}\|_F^2&\le \frac{32(\beta -\alpha)^2}{15(\alpha+\beta)^2}\|Y\|_F^2+\frac{4096r}{15(\alpha+\beta)^2}\|\nabla F(X^*)\|_2^2\\&+\frac{256\lambda}{15(\alpha+\beta)} (10^{1-q/2}\sigma^{q/2-1}_{\min}(X^*)\|X^*\|_*-\|X\|_{S_{q/2}}). 
\end{align*}
Then we obtain the desired result.
\end{proof}

The following lemma provide a lower bound for $\|W\Delta^T\|_F$.
 
\begin{lemma}\label{lowerbound}
Under the same assumptions as in Lemma \ref{upperbound}, and further assume that $\nabla^2\Phi_{\lambda}(U,V)(\Delta,\Delta)\ge 0$, where $\Delta=(\Delta_U;\Delta_V)$ is defined in \eqref{Delta}, then it holds that
 \begin{align*}
\|W\Delta^T\|_F^2&\ge \frac{\alpha}{2\beta} \|WW^T-W^*(W^*)^T\|^2_F+\frac{1}{\beta}\langle \Gamma(X^*),WW^T-W^*(W^*)^T\rangle\\&+\frac{\lambda}{\beta}(\|X\|_{S_{q/2}}-10^{1-q/2}\sigma^{q/2-1}_{\min}(X^*)\|X^*\|_*).
 \end{align*}
 \end{lemma}

\begin{proof}
Note that by direct computation,
\begin{align*} 
\nabla^2\Phi_\lambda(U,V)(\Delta,\Delta)&=\nabla^2F(X)(U\Delta_V^T+\Delta_U V^T,U\Delta_V^T+\Delta_U V^T)
 \\&+2\langle \nabla F(X) , \Delta_U\Delta_V^T \rangle +
\lambda \langle D, \Delta\Delta^T \rangle
\\&+\lambda (p-2)\langle D^2, (U^T\Delta_U)^{\odot 2}+(V^T\Delta_V)^{\odot 2}  \rangle.
\end{align*}
Denote $\Gamma(X)=\begin{bmatrix}
O  &  \nabla F(X)\\
\nabla F(X)^T  & O
\end{bmatrix}$. Since $\Delta\Delta^T=WW^T-W(W^*R^*)^T-W^*R^*W^T+W^*(W^*)^T$, it follows that
\begin{align*} 
2\langle \nabla F(X) , \Delta_U\Delta_V^T \rangle =&\big\langle\Gamma(X),WW^T-W(W^*R^*)^T-W^*R^*W^T+W^*(W^*)^T\big\rangle
\\ =& \big\langle\Gamma(X),2WW^T-W(W^*R^*)^T-W^*R^*W^T\big\rangle
\\& +\big\langle\Gamma(X),W^*(W^*)^T-WW^T\big\rangle
\\=& \big\langle\Gamma(X),W\Delta^T+\Delta W^T\big\rangle+\big\langle\Gamma(X),W^*(W^*)^T-WW^T\big\rangle
\\=& 2\big\langle\Gamma(X),W\Delta^T\big\rangle+\big\langle\Gamma(X),W^*(W^*)^T-WW^T\big\rangle
\\=& 2\operatorname{tr}(\nabla F(X)V\Delta_U^T+\nabla F(X)^TU\Delta_V^T)+\big\langle\Gamma(X),W^*(W^*)^T-WW^T\big\rangle
\\=& -2\lambda\operatorname{tr}(UD\Delta_U^T+VD\Delta_V^T)+\big\langle\Gamma(X),W^*(W^*)^T-WW^T\big\rangle
\\=& -2\lambda \operatorname{tr} (WD\Delta^T)+ \big\langle\Gamma(X),W^*(W^*)^T-WW^T\big\rangle.
\end{align*}
Let $L=\nabla^2 F(X)(U\Delta_V^T+\Delta_U V^T,U\Delta_V^T+\Delta_U V^T)$, since $\nabla^2\Phi_{\lambda}(U,V)(\Delta,\Delta)\ge 0$, then 
\begin{align}\label{III}
L&\ge 2\lambda \operatorname{tr} (WD\Delta^T)+ \langle\Gamma(X),WW^T-W^*(W^*)^T\rangle-
\lambda \langle D,\Delta^T\Delta\rangle \nonumber
\\& =\langle\Gamma(X),WW^T-W^*(W^*)^T\rangle+\lambda \operatorname{tr} (W^TWD)
-\lambda \langle D,(W-\Delta)^T(W-\Delta)\rangle \nonumber
\\&= \langle\Gamma(X),WW^T-W^*(W^*)^T\rangle+2\lambda \|X\|_{S_{q/2}}-\lambda\langle D,(R^*)^T(W^*)^TW^*R^*\rangle \nonumber
\\&\ge \langle\Gamma(X),WW^T-W^*(W^*)^T\rangle+\lambda(2 \|X\|_{S_{q/2}}-\|R^*D(R^*)^T\|_2\|(W^*)^TW^*\|_*), \nonumber
\\&\ge \langle\Gamma(X),WW^T-W^*(W^*)^T\rangle+2\lambda( \|X\|_{S_{q/2}}-\|D\|_2\|X^*\|_*), \nonumber
\\&\ge \langle\Gamma(X),WW^T-W^*(W^*)^T\rangle+2\lambda( \|X\|_{S_{q/2}}-10^{1-q/2}\sigma^{q/2-1}_{\min}(X^*)\|X^*\|_*).
\end{align}

According to the given assumption on $F$ and by Lemma \ref{o1}, it holds
\begin{align}\label{IIII}
\nabla^2 F(X)(U\Delta_V^T+\Delta_U V^T,U\Delta_V^T+\Delta_U V^T)
&\le\beta\|U\Delta_V^T+\Delta_UV^T\|_F^2,\nonumber\\
&\le2\beta(\|U\Delta_V^T\|^2_F+\|\Delta_UV^T\|_F^2)\nonumber \\&= 2\beta\|\mathcal{P}_{\text{off}}(\Delta W^T)\|^2_F=\beta\|W\Delta^T\|_F^2.
\end{align}
Since $U^{T}U=V^{T}V$, then we have $$\|U\Delta_V^T\|^2_F=\|V\Delta_V^T\|_F^2~~ \text{and} ~~\|V\Delta_U^T\|^2_F=\|U\Delta_U^T\|_F^2.$$
In addition, from the RSC module of $F$, it follows that
  \begin{align}\label{i5}
   &\langle\Gamma(X),WW^T-W^*(W^*)^T\rangle\nonumber\\
   &=\langle\Gamma(X)-\Gamma(X^*),WW^T-W^*(W^*)^T\rangle +\langle \Gamma(X^*),WW^T-W^*(W^*)^T\rangle\nonumber\\
   &=2\langle\nabla F(X)-\nabla F(X^*),X-X^*\rangle +\langle \Gamma(X^*),WW^T-W^*(W^*)^T\rangle\nonumber\\
   &=2\int_0^1\nabla^2F(X^*+t(X-X^*))(X-X^*,X-X^*)dt
     +\langle \Gamma(X^*),WW^T-W^*(W^*)^T\rangle\nonumber\\
   &\geq 2\alpha \|X-X^*\|_F^2+\langle \Gamma(X^*),WW^T-W^*(W^*)^T\rangle.
  \end{align}
Furthermore, from Lemma \ref{o2} it follows that
 \begin{align}\label{i6}
  \|WW^T-W^*(W^*)^T\|^2_F
  &=\|\mathcal{P}_{\rm on}(WW^T-W^*(W^*)^T)\|^2_F \nonumber
  \\&+\|\mathcal{P}_{\rm off}(WW^T-W^*(W^*)^T)\|^2_F \nonumber\\
  &\le 2\|\mathcal{P}_{\rm off}(WW^T-W^*(W^*)^T)\|^2_F\nonumber\\&= 4\|X-X^*\|_F^2,
\end{align}
Then together with inequalities \eqref{III}, \eqref{IIII}, \eqref{i5} and \eqref{i6},
we obtain
\begin{align}\label{lbound}
\beta\|W\Delta^T\|^2_F &
  \geq \frac{\alpha}{2}\|WW^T-W^*(W^*)^T\|^2_F+\langle \Gamma(X^*),WW^T-W^*(W^*)^T\rangle\nonumber\\&
  +2\lambda(\|X\|_{S_{q/2}}-10^{1-q/2}\sigma^{q/2-1}_{\min}(X^*)\|X^*\|_*) .\end{align}
Then we obtain the desired result.
\end{proof}

Using Lemma \ref{upperbound} and Lemma \ref{lowerbound}, we derive an error bound between the S-critical points and the true matrix.

\begin{theorem}\label{error-theorem}
Under the same assumptions as in Lemma \ref{upperbound}, and further assume that $\nabla^2\Phi_{\lambda}(U,V)$ is positive semi-definite, then the following inequality holds
\begin{equation}\label{bound}
  \|UV^{T}-X^*\|_F^2
  \le \gamma_1r\|\nabla F(X^*)\|_2^2+\gamma_2\lambda \sigma^{q/2-1}_{min}(X^*)\|X^*\|_*-\gamma_3\lambda\|X\|_{s_{q/2}},
\end{equation}
where $\gamma_1,\gamma_2,\gamma_3>0$ are constants depending only on $\alpha$ and $\beta$ and $q$. 
 \end{theorem}
 \begin{proof}
  By \cite[Lemma 3.6]{non-convex}, we have 
  \begin{equation}\label{connect}
  \|W\Delta^T\|_F^2\leq\frac{1}{8} \|WW^T-W^*{W^*}^T\|_F^2
  +\frac{7+\sqrt{2}}{2}\big\|(WW^T-W^*{W^*}^T)QQ^T\big\|_F^2.
 \end{equation}
Since $\nabla^2\Phi_{\lambda}(U,V)$ is PSD and by Lemma \ref{R1}, it holds
\begin{align*}
\langle \Gamma(X^*),WW^T-W^*{W^*}^T\rangle &\le 
\sqrt{2r}\|\nabla F(X^*)\|_2 \|WW^T-W^*{W^*}^T\|_F\\& \le \frac{64r}{\alpha}\|\nabla F(X^*)\|^2+\frac{\alpha}{128}\|WW^T-W^*{W^*}^T\|^2_F,
\end{align*}
then by Lemma \ref{lowerbound}, it holds
\begin{align*}
\|W\Delta^T\|_F^2 &\ge \frac{63\alpha}{128\beta}\|WW^T-W^*{W^*}^T\|^2-\frac{64r}{\alpha\beta}\|\nabla F(X^*)\|^2\\&-\frac{2\lambda}{\beta}(10^{1-q/2}\sigma^{q/2-1}_{\min}(X^*)\|X^*\|_*-\|X\|_{S_{q/2}}).
\end{align*}
Combining \eqref{connect}, we have
 \begin{align*}
 \frac{63\alpha}{128\beta} \|WW^T-W^*{W^*}^T\|_F^2
 &\le \frac{1}{8}\big\|WW^T-W^*{W^*}^T\big\|_F^2+\frac{64r}{\alpha\beta}\|\nabla F(X^*)\|^2\nonumber\\
 &\quad+\frac{7+\sqrt{2}}{2}\big\|(WW^T-W^*{W^*}^T)QQ^T\big\|_F^2\nonumber\\
&\quad +\frac{2\lambda}{\beta}(10^{1-q/2}\sigma^{q/2-1}_{\min}(X^*)\|X^*\|_*- \|X\|_{S_{q/2}} ).
\end{align*}
Then by Lemma \ref{upperbound}, this yields that
\begin{align*}
  &\Big(\frac{63\alpha}{128\beta}-\frac{1}{8}-\frac{16(7+\sqrt{2})(\beta-\alpha)^2}{15(\alpha+\beta)^2}\Big)
    \|WW^T-W^*{W^*}^T\|^2_F\\
& \le r(\frac{2048(7+\sqrt{2})}{15(\alpha+\beta)^2}+\frac{64}{\alpha\beta})\|\nabla F(X^*)\|^2\\&
 +\lambda (\frac{896+128\sqrt{2}}{15(\alpha+\beta)}+\frac{2}{\beta})(10^{1-q/2}\sigma^{q/2-1}_{min}(X^*)\|X^*\|_*- \|X\|_{S_{q/2}} ).
\end{align*}
Let $c_1=\frac{63\alpha}{128\beta}-\frac{1}{8}-\frac{16(7+\sqrt{2})(\beta-\alpha)^2}{15(\alpha+\beta)^2}$, $c_2=\frac{2048(7+\sqrt{2})}{15(\alpha+\beta)^2}+\frac{64}{\alpha\beta}$, $c_3=(\frac{896+128\sqrt{2}}{\alpha+\beta}+\frac{2}{\beta})10^{1-q/2}$ and $c_4=\frac{896+128\sqrt{2}}{15(\alpha+\beta)}+\frac{2}{\beta}$. And let $\gamma_1=\frac{c_2}{2c_1}$,  $\gamma_2=\frac{c_3}{2c_1}$ and $\gamma_3=\frac{c_4}{2c_1}$.
When $\frac{\beta}{\alpha} \le 1.38$, we have $c_1> 0$. Note that $2\|UV^{T}-X^*\|_F^2\le  \|WW^T-W^*{W^*}^T\|^2_F$.  This implied 
\begin{equation}
\|UV^{T}-X^*\|_F^2\le \gamma_1r\|\nabla F(X^*)\|_2^2+\gamma_2\lambda \sigma^{q/2-1}_{min}(X^*)\|X^*\|_*-\gamma_3\lambda\|X\|_{s_{q/2}}.
\end{equation}
\end{proof}

In Theorem \ref{error-theorem}, we established an error bound between the S-critical points and the true matrix. This conclusion also holds for the global optimum $\bar{X}$ of $\Phi_\lambda$, in which case we have the term $\|\nabla F(\bar{X})\|_F=0$.

\begin{remark}
By Theorem~\ref{error-theorem}, if $\nabla^2\Phi_{\lambda}(U,V)$ is positive semi-definite, the error bound $\|UV^{T}-X^*\|_F^2$ between an S-critical point $(U,V)$ and the true matrix $X^*$ is upper bounded by $$\gamma_1 r \|\nabla F(X^*)\|_2^2 + \gamma_2 \lambda \sigma_{\min}^{q/2-1}(X^*) \|X^*\|_*.$$ Furthermore, the error bound $\|UV^{T}-\bar{X}\|_F^2$ between $(U,V)$ and the global minimizer $\bar{X}$ of $\Phi_\lambda$ is upper bounded by $$\gamma_2 \lambda \sigma_{\min}^{q/2-1}(\bar{X}) \|\bar{X}\|_*.$$ Moreover, the error between $X^*$ and $\bar{X}$ is bounded by 
$$\|\bar{X}-X^*\|_F^2\le \gamma_1 r \|\nabla F(X^*)\|_2^2 + \gamma_2 \lambda \sigma_{\min}^{q/2-1}(X^*)-\gamma_3\lambda\|\bar{X}\|_{s_{q/2}}.$$
\end{remark}

\subsection{Specific Observation Models}\label{subsection3.3}
In this subsection, we show the result of Theorem \ref{error-theorem} for two specific observation models of (\ref{M3}), i.e., matrix sensing and weighted principle component analysis.

For the matrix sensing problem, we consider the loss function $$F(X)=\|\mathcal{A}(X)-y\|_2^2,$$ where $\mathcal{A}$ is the sampling operator and $y\in \mathbb{R}^l$ is the noisy observation from $y=\mathcal{A}(X^*)+\omega$. Here, the entries $\omega_1, \ldots, \omega_l$ of the noise vector $\omega$ are assumed to be i.i.d. sub-Gaussian with parameter $\sigma^2_\omega$, for sub-Gaussian distribution one can refer to \cite{vershynin2012introduction} for details, and Gaussian and Bernoulli distribution are typical examples. By Proposition 5.10 in \cite{vershynin2012introduction}, for any $\tau\in \mathbb{R}^l$, there exists an absolute constant $c>0$ such that with probability at least $1-\frac{1}{mn}$, it holds \begin{equation}
|\tau^T\omega| \leq c \sigma_\omega \sqrt{\ln(nm)} \|\tau\|_2.
\end{equation}
We introduce the definition of Restricted Isometry Property (RIP), (see e.g. \cite{g1,g2,g4}), which is also a commonly used condition in low-rank matrix problem, we show that, in the context of matrix sensing, the RSC and RSS property can be reduced to the RIP condition.
\begin{definition}[Restricted Isometry Property]
Measurement operator $\mathcal{A}: \mathbb{R}^{m \times n} \to \mathbb{R}^{l}$, with rows $A_i$, $i=1,\dots,l$, satisfies $(r,\delta_r)$ RIP if for any $m \times n$ matrix $H$ with rank $\le r$,
\begin{equation}
(1-\delta_r)\|H\|_F^2 \le \frac{1}{l}\sum_{i=1}^l \langle A_i, H \rangle^2 \le (1+\delta_r)\|H\|_F^2.   
\end{equation}
\end{definition}

Suppose that $F$ satisfies the $(r_1,r_2)$-RSC of modulus $\alpha$ and the $(r_1,r_2)$-RSS of modulus $\beta$. Note that $\nabla^2F(X)(H,H)=2\|\mathcal{A}(H)\|_2^2$, this implies that $$\alpha=\underset{\operatorname{rank}(H)\le r_2,\|H\|_F=1}{\min}2\|\mathcal{A}(H)\|_2^2,~~ \beta=\underset{\operatorname{rank}(H)\le r_2,\|H\|_F=1}{\max}2\|\mathcal{A}(H)\|_2^2.$$
Consequently, the $(r_1,r_2)$-RSC of modulus $\alpha=2l(1-\delta_{r_2})$ and the $(r_1,r_2)$-RSS of modulus $\beta=2l(1+\delta_{r_2})$ for some $\delta_{r_2}\in (0, 1)$ reduces to the $(r_2,\delta_{r_2})$  RIP condition.
Then if $F$ satisfies the Condition \ref{condition2}, it also satisfies the Condition \ref{Condition1}.
\begin{condition}\label{condition2}
The sampling operator $\mathcal{A}$ has the $4r$-RIP of constant $\delta_{4r}\in (0,\frac{19}{119})$.
\end{condition}
Furthermore, the following inequation holds with probability at least $1-\frac{1}{mn}$,
\begin{align*}
\|\mathcal{A}^*(\omega)\|_2&=\underset{\|u\|_2=\|v\|_2=1}{\sup}\left \langle u,\mathcal{A}^*(\omega) v \right \rangle 
=\underset{\|u\|_2=\|v\|_2=1}{\sup}\left \langle \omega,\mathcal{A}(uv^T) \right \rangle 
\\&\le c\sigma_\omega \sqrt{\ln(nm)}\underset{\|u\|_2=\|v\|_2=1}{\sup} \|\mathcal{A}(uv^T)\|_2
\\&\le c\sigma_\omega \sqrt{\ln(nm)}\underset{\|u\|_2=\|v\|_2=1}{\sup} \sqrt{1+\delta_{4r}} 
\|uv^T\|_F
\\&\le c\sigma_\omega \sqrt{\ln(nm)} \sqrt{1+\delta_{4r}}. 
\end{align*}

Note that $\nabla F(X^*)=2\mathcal{A}^*(\omega)$, then by Theroem \ref{error-theorem}, we have the following result.
\begin{proposition}\label{exampleMS}
Suppose that function $F(\cdot)=\|\mathcal{A}(\cdot)-y\|_2^2$ satisfies Condition \ref{condition2}. Let $(U,V)$ be an S-critical point of $\Phi_{\lambda}$ and fix any $(U^*,V^*)\in\mathcal{E}^*$. If ${\rm rank}(X)\leq r$, $\sigma_{\min}(X)\ge \frac{1}{10}\sigma_{\min}(X^*)$ and $\nabla^2\Phi_\lambda(U,V)$ is positive semi-definite, then it holds
\begin{equation}
\|UV^{T}-X^*\|_F^2\le \gamma_1r\sigma^2_\omega \ln(nm) (1+\delta_{4r})+\gamma_2\lambda \sigma^{q/2-1}_{min}(X^*)\|X^*\|_*-\gamma_3\lambda\|X\|_{s_{q/2}},
\end{equation}
with probability at least $1-\frac{1}{mn}$,
where $\gamma_1,\gamma_2,\gamma_3>0$ are constants depending only on $c$, $\alpha$ and $\beta$ and $q$. 
\end{proposition}

For the weighted PCA problem, we consider the loss function 
\begin{equation}
F(X)=\|H \circ (X-X^*-E)\|_F^2,
\end{equation}
where $H$ is the positive weight matrix, $E$ is the noise matrix, and $\circ$ denotes the Hadamard product of matrices. Assume that the entries $E_{ij}$ of $E$ are i.i.d. sub-Gaussian random variables of parameter $\sigma_E^2$. By Proposition 5.10 in \cite{vershynin2012introduction}, for any $T\in \mathbb{R}^{m\times n}$, there exists an absolute constant $\bar{c}>0$ such that with probability at least $1-\frac{1}{mn}$, it holds \begin{equation}
|\langle T, E\rangle| \leq \bar{c} \sigma_E \sqrt{\ln(nm)} \|T\|_F.\end{equation}
By direct calculation, for $X, \Delta \in \mathbb{R}^{m \times n}$, we have
\begin{equation}
\nabla F(X) = 2H \circ H \circ (X-X^*-E), \quad \nabla^2 F(X)[\Delta, \Delta] =2 \| H \circ \Delta \|_F^2.
\end{equation}
This implies $F$ satisfies the $(2r,4r)$-RSC of modulus $\alpha = 2\|H\|_{\min}^2$ and $(2r,4r)$-RSS of modulus $\beta = 2\|H\|_{\max}^2$, where $\|H\|_{\min}= \underset{i,j}{\min}~ H_{ij}$ and $\|H\|_{\max}= \underset{i,j}{\max}~ H_{ij}$. Note that
\begin{align*}
\| \nabla F(X^*) \|_F &= \| H \circ H \circ E \|_F = \sup_{\|u\|_2=\|v\|_2=1} \langle u, (H \circ H \circ E) v \rangle \\
&= \sup_{\|u\|_2=\|v\|_2=1} \langle E, H \circ H \circ (uv^T)  \rangle .
\end{align*}
Then with probability at least $1-\frac{1}{mn}$,
we have 
$$\| \nabla F(X^*) \|_F\le \bar{c}\sigma_E \sqrt{\ln(nm)}\sup_{\|u\|_2=\|v\|_2=1} \|H \circ H \circ (uv^T)\|_F\le 2\bar{c}\sigma_E \sqrt{\ln(nm)}\|H\|_{\max}^2.$$

Then by Theroem \ref{error-theorem}, we have the following result.
\begin{proposition}\label{examplePCA}
Set $F(\cdot)=\|H \circ (\cdot-X^*-E)\|_2^2$. Let $(U,V)$ be an S-critical point of $\Phi_{\lambda}$ and fix any $(U^*,V^*)\in\mathcal{E}^*$. If $\|H\|^2_{\max}/\|H\|^2_{\min}\le 1.38$,
 ${\rm rank}(X)\leq r$, $\sigma_{\min}(X)\ge \frac{1}{10}\sigma_{\min}(X^*)$ and $\nabla^2\Phi_\lambda(U,V)$ is positive semi-definite, then it holds
\begin{equation}
\|UV^{T}-X^*\|_F^2\le \gamma_1r\sigma^2_E \ln(nm) \|H\|_{\max}^4+\gamma_2\lambda \sigma^{q/2-1}_{min}(X^*)\|X^*\|_*-\gamma_3\lambda\|X\|_{s_{q/2}},
\end{equation}
with probability at least $1-\frac{1}{mn}$,
where $\gamma_1,\gamma_2,\gamma_3>0$ are constants depending only on $\bar{c}$, $\alpha$ and $\beta$ and $q$. 
\end{proposition}

\section{KL Property for the Factorized Low Rank Matrix Recovery Problem}\label{section4}

In this Section, we consider the following low rank matrix recovery problem 
\begin{equation}\label{p1}
\underset{U\in\mathbb{R}^{m\times r},V\in\mathbb{R}^{n\times r}}{\min}~\Phi_{\lambda}(U,V)=\|\mathcal{A}(UV^T)-y\|_2^2+\frac{\lambda}{q}(\|U\|_{2,q}^q+\|V\|_{2,q}^q),
\end{equation} 
i.e., we set $F(UV^T)=\|\mathcal{A}(UV^T)-y\|_2^2$, and $d=r$ in (\ref{M3}), where $\mathcal{A}$ is the sampling operator, and $y=\mathcal{A}(X^*)+\omega$, $X^*$ is the real data and $\omega$ is the noise. Next, we show the KL property of $\Phi_{\lambda}$ for this case.

First, we present a lemma that describes the perturbation properties of the function $f(x)=\|x\|_2^{q-2}x, x\in \mathbb{R}^{n}$, the proof can be found in Appendix \ref{APRD}.
\begin{lemma}\label{RD}
Denote $f(x)=\|x\|_2^{q-2}x$, where $0\ne x\in \mathbb{R}^{n}$ and $q\in(1,2)$. Then for any $y\in \mathbb{B}(x,\varepsilon)$ with $\varepsilon< \|x\|_2$, we have $$(q-1)(\|x\|_2+\varepsilon)^{q-2}\|y-x\|_2\le \|f(y)-f(x)\|_2\le (\|x\|_2-\varepsilon)^{q-2}\|y-x\|_2.$$
\end{lemma}

We define two operators
 $\Lambda _1:\mathbb{R}^{n \times r}\times\mathbb{R}^{m\times r}\to\mathbb{R}^{n\times r}$
 and $\Lambda _2:\mathbb{R}^{n\times r}\times\mathbb{R}^{m\times r}\to\mathbb{R}^{m\times r}$ by
 \begin{equation}\label{Upsilon12}
  \Lambda _1(U,V):=(\mathcal{A}^*\mathcal{A}(UV^T-X^*))V\ \ {\rm and}\ \
  \Lambda _2(U,V):=(\mathcal{A}^*\mathcal{A}(UV^T-X^*))^TU.
 \end{equation}
 Now we show that $\Phi_{\lambda}$ has the KL property of exponent $\frac{1}{2}$ at S-critical points for proper $\lambda$.
 
 \begin{theorem}\label{KL}
 Given an S-critical point $(\bar{U},\bar{V})$ of $\Phi_{\lambda}$ and $q\in(1,2)$. Suppose that there exists $0<\varepsilon<\underset{i}{\min}\|\bar{U}_{:,i}\|_2$ such that the calmness modulus of
  $\Lambda _1$ and $\Lambda _2$ on $\mathbb{B}((\bar{U},\bar{V}),\varepsilon)$,
  say $c_1$ and $c_2$, satisfies $\lambda \ge \frac{\sqrt{5}L_2}{L_1^2}   (2\bar{c}+\|\mathcal{A}^*(\omega)\|_2)$,
  where $\bar{c}=\max(c_1,c_2)$, $L_1=(q-1)(\underset{i}{\max}\|\bar{U}_{:,i}\|_2+\varepsilon)^{q-2}$, and  $L_2=(\underset{i}{\min}\|\bar{U}_{:,i}\|_2-\varepsilon)^{q-2}$. Then, $\Phi_{\lambda}$ has the KL property of
  exponent $\frac{1}{2}$ at $(\bar{U},\bar{V})$.
 \end{theorem}
 \begin{proof}
By the definition of calmness modulus \cite{Rockafellar2009}, for $i=1,2$, it holds 
$$\|\Lambda _i(U,V)-\Lambda _i(\bar{U},\bar{V})\|_F\le c_i\|(U,V)-(\bar{U},\bar{V})\|_F, ~\forall ~(U,V)\in\mathbb{B}((\bar{U},\bar{V}),\varepsilon).$$
 Note that $\nabla_{\bar{U}}\Phi_{\lambda}(\bar{U},\bar{V})=0$, and $$\nabla_U\Phi_{\lambda}(U,V)=\mathcal{A}^*(\mathcal{A}(UV^T-X^*)-\omega)V+\lambda U D_U=\Lambda_1(U,V)-\mathcal{A}^*(\omega)V+\lambda U D_U,$$
then for any $(U,V) \in \mathbb{B}((\bar{U},\bar{V}),\varepsilon)$, we have 
\begin{align*}
\|\nabla_U\Phi_{\lambda}(U,V)\|^2_F&=\big\|\nabla_U\Phi_{\lambda}(U,V)-\nabla_{\bar{U}}\Phi_{\lambda}(\bar{U},\bar{V})\big\|_F^2\\&=
\big\|\Lambda_1(U,V)-\Lambda_1(\bar{U},\bar{V})-\mathcal{A}^*(\omega)(V-\bar{V}) +\lambda (UD_U-\bar{U}D_{\bar{U}})\big\|_F^2\\
&=\|\Lambda_1(U,V)-\Lambda_1(\bar{U},\bar{V})\|^2_F+\|\mathcal{A}^*(\omega)(V-\bar{V})\|_F^2
\\&~~-2\langle\Lambda _1(U,V)-\Lambda _1(\bar{U},\bar{V}),\mathcal{A}^*(\omega)(V-\bar{V})-\lambda (UD_U-{\bar{U}}D_{\bar{U}})\rangle\\&~~+\lambda^2 \|UD_U-{\bar{U}}D_{\bar{U}}\|_F^2-2\lambda\langle\mathcal{A}^*(\omega)(V-\bar{V}), UD_U-\bar{U}D_{\bar{U}}\rangle.
\end{align*}
Similarly, we have
\begin{align*}
\|\nabla_V\Phi_{\lambda}(U,V)\|^2_F&=\|\Lambda_2(U,V)-\Lambda_2(\bar{U},\bar{V})\|^2_F+\|(\mathcal{A}^*(\omega))^T(U-\bar{U})\|_F^2
\\&~~-2\langle\Lambda_2(U,V)-\Lambda _2(\bar{U},\bar{V}),(\mathcal{A}^*(\omega))^T(U-\bar{U})-\lambda (VD_V-{\bar{V}}D_{\bar{V}})\rangle\\&~~+\lambda^2 \|VD_V-{\bar{V}}D_{\bar{V}}\|_F^2-2\lambda\langle(\mathcal{A}^*(\omega))^T(U-\bar{U}), VD_V-\bar{V}D_{\bar{V}}\rangle.
\end{align*}
By Lemma \ref{RD}, we have $$L_1\|U-\bar{U}\|_F\le\|UD_U-\bar{U}D_{\bar{U}}\|_F\le L_2\|U-\bar{U}\|_F,$$
$$L_1\|V-\bar{V}\|_F\le\|VD_V-\bar{V}D_{\bar{V}}\|_F\le L_2\|V-\bar{V}\|_F,$$
then it follows that 
 \begin{align}\label{gradPhi}
  &\|\nabla\Phi_{\lambda}(U,V)\|^2_F=\|\nabla_U\Phi_{\lambda}(U,V)\|^2_F+\|\nabla_V\Phi_{\lambda}(U,V)\|^2_F\nonumber\\
  &\ge\lambda^2 \|UD_U-\bar{U}D_{\bar{U}}\|_F^2+\lambda^2 \|VD_V-\bar{V}D_{\bar{V}}\|_F^2\nonumber\\&\quad -2\lambda\langle\mathcal{A}^*(\omega)(V-\bar{V}),UD_U-\bar{U}D_{\bar{U}}\rangle-
  2\lambda\langle(\mathcal{A}^*(\omega))^T(U-\bar{U}),VD_V-\bar{V}D_{\bar{V}}\rangle\nonumber\\
  &\quad-2\langle\Lambda _1(U,V)-\Lambda _1(\bar{U},\bar{V}),
     \mathcal{A}^*(\omega)(V-\bar{V})-\lambda (UD_U-\bar{U}D_{\bar{U}})\rangle\nonumber\\
  &\quad-2\langle\Lambda _2(U,V)-\Lambda _2(\bar{U},\bar{V}),(\mathcal{A}^*(\omega))^T(U-\bar{U})-\lambda (VD_V-\bar{V}D_{\bar{V}})\rangle,\nonumber\\
 &\ge\lambda^2 L_1^2\|U-\bar{U}\|_F^2+\lambda^2 L_1^2\|V-\bar{V}\|_F^2-4\lambda L_2\|\mathcal{A}^*(\omega)\|_2\|U-\bar{U}\|_F\|V-\bar{V}\|_F\nonumber\\
  &\quad-\underbrace{2\langle\Lambda _1(U,V)-\Lambda _1(\bar{U},\bar{V}),\mathcal{A}^*(\omega)(V-\bar{V})-\lambda (UD_U-\bar{U}D_{\bar{U}})\rangle}_{I_1}\nonumber\\
  &\quad-\underbrace{2\langle\Lambda _2(U,V)-\Lambda _2(\bar{U},\bar{V}),(\mathcal{A}^*(\omega))^T(U-\bar{U})-\lambda (VD_V-\bar{V}D_{\bar{V}})\rangle}_{I_2}.
 \end{align}
For the term $I_1$, by Lemma \ref{RD}, it holds that
 \begin{align*}
    I_1&\leq2\|\Lambda _1(U,V)-\Lambda _1(\bar{U},\bar{V})\|_F
    \big(\|\mathcal{A}^*(\omega)\|_2\|V-\bar{V}\|_F+\lambda\|UD_U-\bar{U}D_{\bar{U}}\|_F\big)
    \\&\leq2c_1\|(U,V)-(\bar{U},\bar{V})\|_F\big(\|\mathcal{A}^*(\omega)\|_2\|V-\bar{V}\|_F+\lambda L_2\|U-\bar{U}\|_F\big)\\
       &= 2c_1\big(\lambda L_2 \|U-\bar{U}\|^2_F
       +\|\mathcal{A}^*(\omega)\|_2\|V-\bar{V}\|_F^2
       \\&+\big(\lambda L_2+\|\mathcal{A}^*(\omega)\|_2\big)\|U-\bar{U}\|_F\|V-\bar{V}\|_F\big)
       \\&\le c_1((3\lambda L_2+\|\mathcal{A}^*(\omega)\|_2) \|U-\bar{U}\|^2_F
       +(\lambda L_2+3\|\mathcal{A}^*(\omega)\|_2)\|V-\bar{V}\|_F^2).
 \end{align*}
Similarly, for the term $I_2$, we have 
\begin{align*}
I_2\le c_2((\lambda L_2+3\|\mathcal{A}^*(\omega)\|_2) \|U-\bar{U}\|^2_F+(3\lambda L_2+\|\mathcal{A}^*(\omega)\|_2)\|V-\bar{V}\|_F^2).
\end{align*}
Then it follows that 
\begin{align}\label{grad}
  \|\nabla\Phi_{\lambda}(U,V)\|^2_F
  &\ge\Gamma_1(\lambda) \|U-\bar{U}\|_F^2+\Gamma_2(\lambda) \|V-\bar{V}\|_F^2
 \end{align}
where
 \begin{align*}
  &\Gamma_1(\lambda):=L_1^2\lambda^2-(3c_1+c_2+2\|\mathcal{A}^*(\omega)\|_2)L_2\lambda-(c_1+3c_2)\|\mathcal{A}^*(\omega)\|_2\\
  &\Gamma_2(\lambda):=L_1^2\lambda^2-(3c_2+c_1+2\|\mathcal{A}^*(\omega)\|_2)L_2\lambda-(3c_1+c_2)\|\mathcal{A}^*(\omega)\|_2.
 \end{align*}
 Recall that $\lambda \ge \frac{\sqrt{5}L_2}{L_1^2}   (2\bar{c}+\|\mathcal{A}^*(\omega)\|_2),$ then we have $\Gamma_1(\lambda)>0$ and $\Gamma_2(\lambda)>0$.

Since $\nabla{\Phi}_{\lambda}$ is Lipschitz continuous on $\mathbb{B}((\bar{U},\bar{V}),\varepsilon)$. Then, there exists
a constant $L>0$ such that for all $(U,V)\in\mathbb{B}((\bar{U},\bar{V}),\varepsilon)$, \begin{equation}\label{Phi}
  {\Phi}_{\lambda}(U,V)-{\Phi}_{\lambda}(\bar{U},\bar{V})
<\frac{L}{2}\big(\|U-\bar{U}\|_F^2+\|V-\bar{V}\|_F^2\big).
\end{equation}
Then combining (\ref{grad}) and (\ref{Phi}), there exists a constant
$\eta>0$ such that for all $(U,V)\in\mathbb{B}((\bar{U},\bar{V}),\varepsilon)$,
$$\|\nabla\Phi_{\lambda}(U,V)\|_F\ge\eta\sqrt{\Phi_{\lambda}(U,V)-\Phi_{\lambda}(\bar{U},\bar{V})}.$$

This implies that  $\Phi_{\lambda}$ has the KL property of exponent $\frac{1}{2}$ at $(\bar{U},\bar{V})$.
\end{proof}

\section{Inexact Proximal Alternating Linearized Minimization Method and Convergence Guarantee}\label{section5}
In this section, we propose an inexact Proximal Alternating Linearized Minimization Method (Algorithm \ref{PAMM}) for solving the factorized Schatten-$q$ norm regularized low-rank matrix recovery problem. This algorithm follows a block-coordinate update scheme with inexact prox-linear updates. For the exact proximal alternating linearized minimization method and the convergence guarantee, we refer the reader to references \cite{bcd} and \cite{xu2017globally}; however, since the proximal operator of $\|\cdot\|_2^q$ does not admit an explicit solution for general $q$, exact algorithms may not directly applicable, which motivates the study of inexact algorithms. In the following, we propose an inexact proximal operator for $\|\cdot\|_2^q$ (Proposition \ref{inexactprox}) and provide a practical criterion for its computation (Remark \ref{computeprox}). In addition, an adaptive rank adjustment strategy is incorporated in the algorithm, which can effectively and gradually reduce the initial rank; numerical experiments demonstrate that it can closely approach or even achieve the true rank. We prove that our inexact algorithm obtain the subsequence convergence guarantee (Theorem \ref{theorem4.1}) and the whole sequence convergence (Theorem \ref{theoremKL}) and a convergence rate guarantee (Theorem \ref{theoremrate}) under KL condition for the factorized Schatten-$q$ norm regularized low-rank matrix recovery problem.

\begin{algorithm}
\caption{Inexact Proximal Alternating Linearized Minimization Method}
\label{PAMM}
\begin{algorithmic}[1]
\State \textbf{Initialization:} Choose $\lambda^0>\lambda_0 > 0$, $L^0_1,L^0_2>0$, $0<\rho<1$, initial rank $d \geq 1$, $\beta_0 \in [0,~ \beta^*]$, and starting point $(U^0, V^0) \in \mathbb{R}^{m \times d} \times \mathbb{R}^{n \times d}$. Choose $\{\delta_i\}_{i=0}^{\infty}$, such that $\frac{1}{4}\ge\delta_i\to 0$ and $\sum_{i=0}^\infty \delta_i<+\infty$. Set $(U^{-1}, V^{-1}) = (U^0, V^0)$, $\theta_{-1} = \theta_0 = 1$ and $k = 0$. Set the stopping conditions $$
\|\nabla F(U^k, V^k) - \nabla F(U^{k+1}, V^{k+1}) + L^k(W^{k+1} - W^k)\|_F \leq \epsilon (1+\|W^k\|_F).$$ 
\While{the stopping conditions are not satisfied and $k\le k_{\max}$}
    \State  Set $\widetilde{U}^k = U^k + \beta_k(U^k - U^{k-1})$ and $\widetilde{V}^k = V^k + \beta_k(V^k - V^{k-1})$;
    \State  Solve the following minimization problems
    \begin{align*}
         U^{k+1} &\in \operatorname*{arg\,min}_{U \in \mathbb{R}^{m\times d}} \left\{ \langle \nabla_{U} F^{k}(\widetilde{U}^k), U - \widetilde{U}^k \rangle + \frac{L_{1}^{k}}{2} \|U - \widetilde{U}^k\|_F^2 + \frac{\lambda^k}{q} \|U\|_{2,q}^q \right\},  \\
        V^{k+1} &\in \operatorname*{arg\,min}_{V \in \mathbb{R}^{n\times d}} \left\{ \langle \nabla_{V} \bar{F}^{k}(\widetilde{V}^k), V - \widetilde{V}^k \rangle + \frac{L_{2}^{k}}{2} \|V - \widetilde{V}^k\|_F^2 + \frac{\lambda^k}{p} \|V\|_{2,p}^p \right\}. 
    \end{align*}
    If the subproblem does not have closed-form solutions, then solve it inexactly by Proposition \ref{inexactprox} with parameter $\delta_k$. 
    \State  Adaptive rank adjustment : Remove the zero-vector columns from $U^{k+1}$ and $V^{k+1}$ and update $d$ accordingly. 
    \State  Choose $\beta_{k+1}\in [0,\beta^*]$. For one possible choice, update $\beta_{k+1}$ by Nesterov extrapolation, i.e., $\beta_{k+1} = \min\{\beta^*,\frac{\theta_{k}-1}{\theta_{k+1}}\}$, where $\theta_{k+1} = \frac{1}{2} \big( 1 + \sqrt{1 + 4\theta_k^2} \big)$. 
    \State  Update $\lambda^{k+1} = \max\{\rho\lambda^k,\lambda_0\}$.
\EndWhile
\end{algorithmic}
\end{algorithm}

\begin{table}[htbp]
\centering
\caption{Summary of notation}
\begin{tabular}{l p{10cm}}
\toprule[1pt] 
Notion & Definition \\
\midrule[1pt] 
$(U^{k+1},V^{k+1})$ & The value of $(U,V)$ after the $k$-th iteration \\
$(\widetilde{U}^k,\widetilde{V}^k)$ & The extrapolation term used at the $k$-th iteration  \\
$L^k_i$ & The proximal step-size parameter at the $k$-th iteration, see (\ref{sub1}) and (\ref{sub2}) \\
$\beta_k$ & The extrapolation weight used at the $k$-th iteration\\
$\lambda^k$ & The parameter of regularization term at the $k$-th iteration \\
$\lambda^k_{1,t}$ & The parameter $\lambda$ in the perturbed subproblem, whose solution corresponds to the inexact solution at the $k$-th iteration, see (\ref{UVprox1}) and (\ref{UVprox2})  \\
$\delta_k$ & Tolerance parameter for the allowed accuracy of the inexact solution, see Proposition \ref{inexactprox} \\
$\Phi^k_\lambda(U,V)$ & The objective function with perturbed $\lambda$ at the $k$-th iteration, see (\ref{Phik}) \\
$\Phi_{\lambda}(U,V)$ & The objective function with parameter $\lambda$, see (\ref{M5}) \\
$F^k(U)$ & The function $F$ of $U$ with $V$ fixed to $V^k$ at the $k$-th iteration, see (\ref{FF}) \\
$\bar{F}^k(V)$ & The function $F$ of $V$ with $U$ fixed to $U^{k+1}$ at the $k$-th iteration, see (\ref{FF})\\
\bottomrule 
\end{tabular}
\label{tab:notation}
\end{table}

Consider the factorized Schatten-q norm regularized low-rank matrix recovery problem, where $q, p\in (0,2]$:
\begin{equation}\label{M5}
\underset{U\in \mathbb{R}^{m\times d},V\in \mathbb{R}^{n\times d}}{\min} \Phi_\lambda (U,V)
 :=F(UV^T)+\frac{\lambda}{q}\|U\|_{2,q}^q+\frac{\lambda}{p}\|V\|_{2,p}^p.
\end{equation}
Denote
\begin{equation}\label{FF}
F^{k}(U)=F(U,V^{k}) \quad \text{and}\quad
 \bar{F}^{k}(V)=F(U^{k+1},V).
\end{equation}
Here we slightly abuse the notation $F$, and define $F(U,V):=F(UV^\mathrm{T})$.


We alternatingly minimize the function with prox-linear update. Then for every iteration, we solve the following subproblems:
\begin{equation}\label{sub1}
U^{k+1} \in \operatorname*{arg\,min}_{U \in \mathbb{R}^{m\times d}} \left\{ \langle \nabla_{U} F^{k}(\widetilde{U}^k), U - \widetilde{U}^k \rangle + \frac{L_{1}^{k}}{2} \|U - \widetilde{U}^k\|_F^2 + \frac{\lambda^k}{q} \|U\|_{2,q}^q \right\},
\end{equation}
\begin{equation}\label{sub2}
V^{k+1} \in \operatorname*{arg\,min}_{V \in \mathbb{R}^{n\times d}} \left\{ \langle \nabla_{V} \bar{F}^{k}(\widetilde{V}^k), V - \widetilde{V}^k \rangle + \frac{L_{2}^{k}}{2} \|V - \widetilde{V}^k\|_F^2 + \frac{\lambda^k}{p} \|V\|_{2,p}^p \right\}.   
\end{equation}
Here $\widetilde{U}^k = U^k + \beta_k(U^k - U^{k-1})$ and $\widetilde{V}^k = V^k + \beta_k(V^k - V^{k-1})$ are the accelerated strategy and the parameter $\beta_k$ can be chosen by Nesterov extrapolation technique.

\begin{definition}[Proximal mapping]
For a proper, lower semi-continuous function $h$, its proximal mapping $\operatorname{prox}_h(\cdot)$ is defined as
\begin{equation}
\operatorname{prox}_h(x) = \arg \min_{y} \frac{1}{2} \| y-x\|^2 + h(y).
\end{equation}
\end{definition}
Using this notation, the update in (\ref{sub1}) and $(\ref{sub2})$ can be written as 
\begin{align}
&U_t^{k+1} \in \operatorname{prox}_{\frac{\lambda^k}{qL_{1}^k}\|\cdot\|_2^q} \Big( \widetilde{U}_t^k-\frac{1}{L_{1}^k}\nabla^t_{U} 
F^{k}(\widetilde{U}^k) \Big),~t=1,2,\ldots,d,\\& V_t^{k+1} \in \operatorname{prox}_{\frac{\lambda^k}{pL_{2}^k}\|\cdot\|_2^p} \Big( \widetilde{V}_{t}^k-\frac{1}{L_{2}^k}\nabla^t_{V} 
\bar{F}^{k}(\widetilde{V}^k) \Big),~t=1,2,\ldots,d,
\end{align}
where $(\cdot)_{t}$ denote the $t$-th column of matrix $(\cdot)$, and $$\nabla^t_{U} 
F^{k}(\widetilde{U}^k)=(\nabla_{U} 
F^{k}(\widetilde{U}^k))_t, ~\nabla^t_{V} 
\bar{F}^{k}(\widetilde{V}^k)=(\nabla_{V} 
\bar{F}^{k}(\widetilde{V}^k))_t.$$

Although the proximal operator of the function $\|\cdot\|_2^q$ lacks a closed-form solution for general parameter $q$, an approximate solution (\ref{approxisolution}) can be computed. Below, we present the approximate solution of the proximal operator of $\|\cdot\|_2^q$ along with its properties. To this end, we first propose Lemma \ref{Inexact}, the proof can be found in  Appendix \ref{APInexact}.

Given function $g:\mathbb{R} \mapsto \mathbb{R}$, denote \begin{equation}
\mathcal{P}_g(x)=\begin{cases}
x, &\text{if}~~g(x)\le g(0), \\ 
0, &\text{otherwise.}
\end{cases}    
\end{equation} 

\begin{lemma}\label{Inexact}
Given $y,L,\lambda>0$ and $q\in (0,2]$. Denote $g(x)=\frac{L}{2} (x-y)^2+\frac{\lambda}{q}x^q$, $x\in\mathbb{R}_{\ge 0}$. Let $\bar{x}$ be an approximate solution of $L(x-y)+\lambda x^{q-1}=0$, such that $|\bar{x}-x^*|\le \delta \bar{x}$, with $\delta \le \frac{1}{2}$, $\bar{x}\in (0,y)$ and $x^*$ is the exact positive solution. Denote $\bar{g}(x)=\frac{\bar{L}}{2}(x-y)^2+\frac{\lambda}{q}x^q$, where $\bar{L}=\frac{\lambda \bar{x}^{q-1}}{(y-\bar{x})}$. Then $\mathcal{P}_{\bar{g}}(\bar{x})$ is the minimizer of function $\bar{g}$ over the interval $[0,y]$ and we have $$(1+C\delta)^{-q-1}L\le \bar{L}\le (1-C\delta)^{-q-1} L,$$ where $C=\max\{\frac{\lambda y^{q-1}}{2^{q-1}L},1\}$ for $1\le q\le 2$ and $C=\max\{\frac{4}{q}y^{\,2/q-2}(\frac{Lq}{2\lambda})^\frac{1}{q},1\}$ for $0< q< 1$.
\end{lemma}


By applying Lemma~\ref{Inexact}, we can present an approximate solution of the operator $\operatorname{prox}_{\frac{\lambda}{qL}\|\cdot\|_2^q}(\cdot)$, which can be interpreted as the exact solution of the proximal operator with a perturbed parameter $\lambda$ and the admissible range of this perturbation is explicitly characterized in the following Proposition \ref{inexactprox}.

\begin{proposition}\label{inexactprox}
Given $u^*\in \mathbb{R}^m$, $\lambda,L>0$, $0<\delta\le \frac{1}{2}$ and $q\in(0,2]$. Denote $y=\|u^*\|_2$. Let $\bar{x}$ be an approximate solution of $L(x-y)+\lambda x^{q-1}=0$ such that $(1-C\delta)\bar{x}\le x^*\le (1+C\delta)\bar{x}$
where $C=\min\{\frac{2^{q-1}L}{\lambda y^{q-1}},1\}$ for $1\le q\le 2$ and $C=\min\{\frac{q}{4}y^{\,2-2/q}(\frac{2\lambda}{Lq})^\frac{1}{q},1\}$ for $0<q<1$. Let $f(x)=\frac{\lambda \bar{x}^{q-1}}{2(y-\bar{x})}(x-y)^2+\frac{\lambda}{q}x^q$. 
Then \begin{equation}\label{approxisolution}
\mathcal{P}_f(\bar{x})\frac{u^*}{\|u^*\|_2} \in \operatorname{prox}_{\frac{\bar{\lambda}}{qL}\|\cdot\|_2^q} (u^*),
\end{equation}
where $|(\frac{\lambda}{\bar{\lambda}})^{\frac{1}{q+1}}-1|\le \delta$ and $\bar{\lambda}=\frac{(y-\bar{x})L}{\bar{x}^{q-1}}$.
\end{proposition}
\begin{proof}
Note that $x^*\frac{u^*}{\|u^*\|_2}\in \operatorname{prox}_{\frac{\lambda}{qL}\|\cdot\|_2^q} (u^*)$, where $x^*\in \underset{x}{\arg\min} \frac{L}{2}(x-\|u\|_2)^2+\frac{\lambda}{q}x^q$. Then by Lemma \ref{Inexact}, $\mathcal{P}_f(\bar{x})\frac{u^*}{\|u^*\|_2}$ is the is the minimizer of $\frac{\bar{L}}{2}\|\cdot-u\|_2^2+\frac{\lambda}{q}\|\cdot\|_2^q$, where $(1+\delta)^{-q-1}L\le \bar{L}\le (1-\delta)^{-q-1} L$. This implies $\mathcal{P}_f(\bar{x})\frac{u^*}{\|u^*\|_2} \in \operatorname{prox}_{\frac{\bar{\lambda}}{qL}\|\cdot\|_2^q} (u^*)$,
where $|(\frac{\lambda}{\bar{\lambda}})^{\frac{1}{q+1}}-1|\le \delta$.
\end{proof}

\begin{remark}
If $q\in\{2,\frac{7}{4},\frac{5}{3},\frac{3}{2},\frac{4}{3},\frac{5}{4},1,\frac{2}{3},\frac{1}{2}\}$, then $\operatorname{prox}_{\frac{\lambda}{qL}\|\cdot\|_2^q}(u)$ admits a closed-form solution, since in this case the equation $L(\cdot-\|u\|_2)+\lambda (\cdot)^{q-1}=0$ admits a closed-form solution.
\end{remark}

\begin{remark}\label{computeprox}
To obtain $\bar{x}$ as an approximate solution of $L(x-y) + \lambda x^{q-1} = 0$ satisfying $(1 - C\delta)\bar{x} \le x^* \le (1 + C\delta)\bar{x}$, where $C$ is the same in Proposition \ref{inexactprox}, we can use numerical methods such as bisection or Newton’s method on $[0, y]$ to achieve the required accuracy.
\end{remark}

By Proposition \ref{inexactprox}, there exist $\lambda_{i,t}^k$ for $i=1,2$, $t=1,2,\ldots,d$ and $k\ge 0$, such that the inexact update $(U^{k+1},V^{k+1})$ of the algorithm satisfies
\begin{equation}\label{UVprox1}
U_t^{k+1} \in \operatorname{prox}_{\frac{\lambda_{1,t}^k}{qL_{1}^k}\|\cdot\|_2^q} \Big( \widetilde{U}_{t}^k-\frac{1}{L_{1}^k}\nabla^t_{U} 
F^{k}(\widetilde{U}^k) \Big),
\end{equation}
\begin{equation}\label{UVprox2}
V_t^{k+1} \in \operatorname{prox}_{\frac{\lambda_{2,t}^k}{pL_{2}^k}\|\cdot\|_2^p} \Big( \widetilde{V}_{t}^k-\frac{1}{L_{2}^k}\nabla^t_{V} 
\bar{F}^{k}(\widetilde{V}^k) \Big),
\end{equation}
furthermore we have $\Big|\Big(\frac{\lambda^k}{\lambda_{1,t}^k}\Big)^{\frac{1}{q+1}}-1\Big|\le \delta_k\to 0$ and  $\Big|\Big(\frac{\lambda^k}{\lambda_{2,t}^k}\Big)^{\frac{1}{q+1}}-1\Big|\le \delta_k\to 0$.
Denote \begin{equation}
R^k_1(U)=\sum_{i=1}^{d}\frac{\lambda^k_{1,t}}{q}\|U_{:,t}\|_2^q,~R^k_2(V)=\sum_{t=1}^{d}\frac{\lambda^k_{2,t}}{p}\|V_{:,t}\|_2^p,
\end{equation}
and \begin{equation} \label{Phik}
\Phi^{k+1}_\lambda(U,V)=F(UV^T)+R^k_1(U)+R^k_2(V).
\end{equation}

Next, we present the square-summable property of the iterative sequence generated by Algorithm \ref{PAMM}, and show that every limiting point is a critical point. 

\begin{assumption}\label{A1}
Function $F$ is continuously differentiable, proper and lower bounded by constant $\mathcal{M}$.
\end{assumption}
\begin{assumption}\label{A2}
$\Phi_\lambda$ has a critical point $(U^*,V^*)$.
\end{assumption}
\begin{assumption}\label{A3}
$\nabla F^{k}(U)$ and $\nabla \bar{F}^{k}(V)$ are Lipschitz continuous with Lipschitz constant $L_{1}^k/2$ and $L_{2}^k/2$, respectively, and there exist constants $0<l<L<+\infty$ such that $l\le L_{1}^k,L_{2}^k\le L$ for all $k$. 
\end{assumption}

\begin{condition}\label{C2}
The weight $\beta_k$ is chosen such that $\Phi^{k+1}_\lambda(U^{k+1},V^{k+1})\le \Phi^{k+1}_\lambda(U^{k},V^{k})$. And there exist $\delta_\beta<1$, such that $$0\le \beta^k\le \frac{1}{6}\delta_\beta \min\{\sqrt{L_{1}^{k-1}/L_{1}^{k}},\sqrt{L_{2}^{k-1}/L_{2}^{k}}\},$$ 
holds uniformly over all $k$.
\end{condition}

\begin{theorem}\label{theorem4.1}
Consider the factorized low rank matrix rocovery problem (\ref{M5}). Assume Assumptions \ref{A1}, \ref{A2}, \ref{A3}  and Condition \ref{C2} are satisfied. Let $\{(U^k,V^k)\}_{k=0}^{\infty}$ be the sequence generated by Algorithm \ref{PAMM}, then $(U^k,V^k)$ is bounded and 
\begin{equation}
\sum_{k=0}^{+\infty} \|W^{k+1}-W^{k}\|_F^2<+\infty,
\end{equation} where $W^k=(U^k;V^k)$. Any limit point $(\bar{U},\bar{V})$ of $\{(U^k,V^k)\}_{k=0}^{\infty}$ is a critical point of $\Phi_{\lambda_0}$. Furthermore if the subsequence $\{(U^k,V^k)\}_{k\in \mathcal{K}}$ converge to $(\bar{U},\bar{V})$, then 
\begin{equation}
\lim_{k\to \infty, k\in \mathcal{K}}\Phi_{\lambda_0}(U_k,V_k)=\Phi_{\lambda_0}(\bar{U},\bar{V}).
\end{equation}
\end{theorem}
\begin{proof}
Since function F is lower bounded by constant $\mathcal{M}$. Without loss of generality, assume $\mathcal{M}=0$ and $q\le p$. 

First, we show that $(U^k,V^k)$ is bounded. There exist $K>0$ such that $\lambda^k=\lambda_0$ and $\delta_k<\min\{\frac{q}{32p},\frac{1}{64}\}$, when $k\ge K$. Since $\beta_k$ is chosen such that $\Phi^{k+1}_\lambda(U^{k+1},V^{k+1})\le \Phi^{k+1}_\lambda(U^{k},V^{k})$, then for $k\ge K$, 
\begin{equation}\label{H1}
\Phi^{k+1}_\lambda(U^{k},V^{k})\ge \frac{1}{q}\sum_{t=1}^d\lambda_{1,t}^k\|U^{k+1}_t\|_2^q+\frac{1}{p}\sum_{t=1}^d\lambda_{2,t}^k\|V^{k+1}_t\|_2^p.  
\end{equation}
Furthermore, since $\delta_k\le \frac{1}{4}$, then 
\begin{align*}
\max\{1-(1+\delta_k)^{-q-1},(1-\delta_k)^{-q-1}-1\}&=(1-\delta_k)^{-q-1}-1\\&\le (1-\delta_k)^{-3}-1\le 16\delta_k.
\end{align*}
Denote $H_k=\frac{1}{q}\sum_{t=1}^d(\lambda_{1,t}^k-\lambda_0)\|U^{k+1}_t\|_2^q+\frac{1}{p}\sum_{t=1}^d(\lambda_{2,t}^k-\lambda_0)\|V^{k+1}_t\|_2^p$. Then 
\begin{align}\label{H2}
H_k& \le \frac{1}{q}\underset{t}{\max}\{|\lambda_0-\lambda_{1,t}^k|,|\lambda_0-\lambda_{2,t}^k|\}\big( \|U^{k+1}\|^q_{2,q}+\|V^{k+1}\|_{2,p}^p\big)\nonumber
\\& \le \frac{1}{q}\lambda_0\max\{1-(1+\delta_k)^{-q-1},(1-\delta_k)^{-q-1}-1\}\big( \|U^{k+1}\|^q_{2,q}+\|V^{k+1}\|_{2,p}^p\big)\nonumber
\\& \le \frac{16}{q}\lambda_0\delta_k\big( \|U^{k+1}\|^q_{2,q}+\|V^{k+1}\|_{2,p}^p\big).
\end{align}
By (\ref{H1}) and (\ref{H2}), we have \begin{equation}\label{H4}
\Phi^{k+1}_\lambda(U^{k+1},V^{k+1})\ge 
\frac{\lambda_0}{p}(1-16\frac{p}{q}\delta_k)\big( \|U^{k+1}\|^q_{2,q}+\|V^{k+1}\|_{2,p}^p\big).
\end{equation}
This implies \begin{align}\label{H3}
\Phi^{k+2}_\lambda(U^{k+1},V^{k+1})&=\Phi^{k+1}_\lambda(U^{k+1},V^{k+1})\nonumber\\&+\frac{1}{q}\sum_{t=1}^d(\lambda_{1,t}^{k+1}-\lambda_{1,t}^{k})\|U^{k+1}_t\|_2^q+\frac{1}{p}\sum_{t=1}^d(\lambda_{2,t}^{k+1}-\lambda_{2,t}^{k})\|V^{k+1}_t\|_2^p\nonumber
\\& \le \Phi^{k+1}_\lambda(U^{k+1},V^{k+1})+\frac{32}{q}\lambda_0\delta_k\big( \|U^{k+1}\|^q_{2,q}+\|V^{k+1}\|_{2,p}^p\big)\nonumber
\\& \le \big(1+\frac{32\delta_k}{\frac{q}{p}-16\delta_k}  \big)\Phi^{k+1}_\lambda(U^{k+1},V^{k+1})\nonumber
\\&\le \big(1+64\frac{p}{q}\delta_k\big)\Phi^{k+1}_\lambda(U^{k+1},V^{k+1}).
\end{align}
Then by (\ref{H3}) and Condition \ref{C2}, we have \begin{equation}
\Phi^{K+k+1}_\lambda(U^{K+k},V^{K+k})\le \prod_{i=K}^{K+k-1} (1+64\frac{p}{q}\delta_i)\Phi^{K+1}_\lambda(U^{K},V^{K}).
\end{equation}
Since $\sum_{i=0}^{+\infty}\sigma_i< +\infty$, then 
$\prod_{i=K}^{+\infty} (1+64\frac{p}{q}\delta_i)< +\infty$. This implies $\Phi^{k+1}_\lambda(U^{k},V^{k})$ is bounded, and further by (\ref{H4}), $(U^k,V^k)$ is also bounded.

Next, we show the square-summable property of the iterative sequence. By (\ref{UVprox1}), (\ref{UVprox2}) and applying Lemma 1 \cite{xu2017globally}, we have
\begin{align}\label{F1}
&F^k(U^{k})+\frac{1}{q}\sum_{t=1}^d\lambda_{1,t}^k\|U^{k}_t\|_2^q-\big(F^k(U^{k+1})+\frac{1}{q}\sum_{t=1}^d\lambda_{1,t}^k\|U^{k+1}_t\|_2^q\big)\nonumber
\\&\ge \frac{L_{1}^k}{8}\|U^{k}-U^{k+1}\|_F^2-\frac{L_{1}^{k-1}\delta_\beta^2}{8}\|U^{k-1}-U^{k}\|_F^2,
\end{align}
\begin{align}\label{F2}
&\bar{F}^k(V^{k})+\frac{1}{p}\sum_{t=1}^d\lambda_{2,t}^k\|V^{k}_t\|_2^p-\big(\bar{F}^k(V^{k+1})+\frac{1}{p}\sum_{t=1}^d\lambda_{2,t}^k\|V^{k+1}_t\|_2^p\big)\nonumber
\\&\ge \frac{L_{2}^k}{8}\|V^{k}-V^{k+1}\|_F^2-\frac{L_{2}^{k-1}\delta_\beta^2}{8}\|V^{k-1}-V^{k}\|_F^2.
\end{align}
Summing the above expression over $k$ implies that 
\begin{align}
\Phi_{\lambda^0}(U^0,V^0)&+\frac{1}{q}\sum_{k=0}^{+\infty}\sum_{t=1}^{d}\Big(\Big|\lambda_{1,t}^k-\lambda^k\Big|\|U^k_t\|_2^q+\Big|\lambda_{1,t}^k-\lambda^{k+1}\Big|\|U^{k+1}_t\|_2^q\Big)\nonumber\\&+\frac{1}{p}\sum_{k=0}^{+\infty}\sum_{t=1}^{d}\Big(\Big|\lambda_{2,t}^k-\lambda^k\Big|\|V^k_t\|_2^p+\Big|\lambda_{2,t}^k-\lambda^{k+1}\Big|\|V^{k+1}_t\|_2^p\Big)\nonumber\\& \ge
\sum_{k=0}^{+\infty}\frac{(1-\delta_\beta^2)}{8}\Big(L_{1}^k\|U^{k}-U^{k+1}\|_F^2+L_{2}^k\|V^{k}-V^{k+1}\|_F^2\Big)\nonumber\\&\ge \sum_{k=0}^{+\infty}\frac{(1-\delta_\beta^2)}{8}l\|W^k-W^{k+1}\|_F^2.
\end{align}
Since $\lambda^k=\lambda_0$ when $k\ge K$, $\delta_k\le \frac{1}{4}$, $\sum_{k=0}^{\infty}\delta_k\le +\infty$  and $(U^k,V^k)$ is bounded, then
\begin{align*}
&\sum_{k=K}^{+\infty}\sum_{t=1}^{d}\Big(\Big|\lambda_{1,t}^k-\lambda^k\Big|\|U^k_t\|_2^q+\Big|\lambda_{1,t}^k-\lambda^{k+1}\Big|\|U^{k+1}_t\|_2^q\Big)\\&\le \sum_{k=K}^{+\infty}\sum_{t=1}^{d}\lambda_{0}\Big|\frac{\lambda_{1,t}^k}{\lambda_0}-1\Big|(\|U_t^k\|_2^q+\|U_t^{k+1}\|_2^q)
\\&\le 2\sum_{k=K}^{+\infty}\lambda_{0}\max\{1-(1+\delta_k)^{-q-1},(1-\delta_k)^{-q-1}-1\}\|U^k\|_{2,q}^q\\&\le 32\sum_{k=K}^{+\infty}\lambda_{0}\delta_k\|U^k\|_{2,q}^q< +\infty.
\end{align*}
Similarly, we have 
$$\sum_{k=K}^{+\infty}\sum_{t=1}^{d}\Big(\Big|\lambda_{2,t}^k-\lambda^k\Big|\|V^k_t\|_2^p+\Big|\lambda_{2,t}^k-\lambda^{k+1}\Big|\|V^{k+1}_t\|_2^p\Big)< + \infty.$$
This implies \begin{equation}\label{W8}
\sum_{k=0}^{+\infty}\|W^k-W^{k+1}\|_F^2<+\infty.
\end{equation}

Suppose that the subsequence $\{(U^k,V^k)\}_{k\in \mathcal{K}}$ converge to $(\bar{U},\bar{V})$, and suppose that $L^{k}_{i}$ converge to $\bar{L}_{i}$ by taking another subsequence if necessary as $k\to \infty, k\in \mathcal{K}$, for $i=1,2$. Then by (\ref{W8}), we have $\underset{k\to \infty}{\lim}\|W^k-W^{k+1}\|=0$, which deduces that $\underset{k\to \infty,k\in\mathcal{K}}{\lim}\widetilde{W}^k=\bar{W}$ and $\underset{k\to \infty,k\in\mathcal{K}}{\lim}W^{k+1}=\bar{W}$.

By (\ref{UVprox1}), and let $ k\to \infty, k\in \mathcal{K}$ yields 
\begin{align}
\frac{\lambda_0}{q}\|\bar{U}\|_{2,q}^q \nonumber&= \underset{k\to \infty,k\in\mathcal{K}}{\lim}\langle \nabla_{U} F^{k}(\widetilde{U}^k), U^{k+1} - \widetilde{U}^k \rangle + \frac{L_{1}^{k}}{2} \|U^{k+1} - \widetilde{U}^k\|_F^2 + \sum_{t=1}^d\frac{\lambda_{1,t}^k}{q} \|U^{k+1}_{t}\|_{2}^q \nonumber
\\& \le \underset{k\to \infty,k\in\mathcal{K}}{\lim}\langle \nabla_{U} F^{k}(\widetilde{U}^k), U - \widetilde{U}^k \rangle + \frac{L_{1}^{k}}{2} \|U - \widetilde{U}^k\|_F^2 + \sum_{t=1}^d\frac{\lambda_{1,t}^k}{q} \|U_{t}\|_{2}^q \nonumber
\\&= \langle \nabla_{U} F(\bar{U},\bar{V}), U - \bar{U} \rangle + \frac{\bar{L}_{1}}{2} \|U - \bar{U}\|_2^2 + \frac{\lambda_{0}}{q} \|U\|_{2,q}^q.
\end{align}
Hence 
\begin{equation}
\bar{U} \in \underset{U}{\arg\min} \langle \nabla_{U} F(\bar{U},\bar{V}), U - \bar{U} \rangle + \frac{\bar{L}_{1}}{2} \|U - \bar{U}\|_2^2 + \frac{\lambda_{0}}{q} \|U\|_{2,q}^q, 
\end{equation}
and similarly we have
\begin{equation}
\bar{V} \in \underset{V}{\arg\min} \langle \nabla_{V} F(\bar{U},\bar{V}), V - \bar{V} \rangle + \frac{\bar{L}_{2}}{2} \|V - \bar{V}\|_2^2 + \frac{\lambda_{0}}{p} \|V\|_{2,q}^p.
\end{equation}
Then by the first-order optimality condition, we have $0\in \partial \Phi_{\lambda_0}(\bar{U},\bar{V})$, this implies that $(\bar{U},\bar{V})$ is a critical point of (\ref{M5}). Note that $F$, $\|\cdot\|_2^p$ and $\|\cdot\|_2^q$ are continuous, then 
\begin{equation}
\lim_{k\to \infty, k\in \mathcal{K}}\Phi_{\lambda_0}(U_k,V_k)=\Phi_{\lambda_0}(\bar{U},\bar{V}).
\end{equation}
\end{proof}
\begin{remark}
If function $F$ is block multi-convex, i.e., it is convex with respect to each block of variables while keeping the remaining variables fixed, and $q,p\in [1,2]$, then $L_{i}^k$ can be relax to $2L_{i}^k$ and $\delta_\beta$ can be relax to $6\delta_\beta$ in Condition \ref{C2}. One may refer to Remark 2 \cite{xu2017globally} for the proof of (\ref{F1}) and (\ref{F2}); the convergence results then follow from the same arguments as in Theorem \ref{theorem4.1}.
\end{remark}

\begin{remark}
By (\ref{F1}) and (\ref{F2}), if $\beta_k=0$, i.e., no extrapolation, then Condition \ref{C2} naturally holds. In practice, using extrapolation step can often accelerate the algorithm. To employ extrapolation while Condition \ref{C2} also holds, for the $k$-th iteration, one can first do the update with a positive parameter $\beta_k$ and then check that if $\Phi^{k+1}_\lambda(U^{k+1},V^{k+1})> \Phi^{k+1}_\lambda(U^{k},V^{k})$, then redo the $k$-th update by setting $\beta_k=0$. Furthermore, if the assumptions of Proposition \ref{single-value} are satisfied, then one can find a suitable $\beta_k$ by backtracking to satisfy the Condition \ref{C2}.
\end{remark}

Moreover, by leveraging the Kurdyka-\L ojasiewicz property, we next establish the global convergence and convergence rate. We first present the following result.

\begin{lemma}\label{5.2}
Let $\{(U^k,V^k)\}_{k=0}^{+\infty}$ be generated from Algorithm \ref{PAMM}. For a specific iteration $k$, assume $W^t\in \mathbb{B}_\rho(\bar{W})$, $t=k+1,k,k-1$ for some $\bar{W}$ and $\rho>0$, where $W^k=(U^k,V^k)$. If $\nabla_UF(U,V)$ and $\nabla_VF(U,V)$ are Lipschitz continuous with constant $L^*$ within $\mathbb{B}_{4\rho}(\bar{W})$ with respect to $(U,V)$, i.e.,
\begin{equation}\label{LI1}
\|\nabla_UF(U_1,V)-\nabla_UF(U_2,V)\|_F\le L^*\|U_1-U_2\|_F,
\end{equation}
\begin{equation}\label{LI2}
\|\nabla_VF(U,V_1)-\nabla_VF(U,V_2)\|_F\le L^*\|V_1-V_2\|_F,  
\end{equation}
for any $(U_i,V), (U,V_i)\in \mathbb{B}_{4\rho}(\bar{W})$,
then \begin{equation}
    \operatorname{dist}(0,\partial \Phi^{k+1}_\lambda (U^{k+1},V^{k+1})) \le 3L_HT^{k}+2\beta_kL_HT^{k-1},
\end{equation}
where $T^{t}=\|U^{t+1}-U^{t}\|_F+\|V^{t+1}-V^{t}\|_F$ and $L_H:=\max\{L^*,L\}$.
\end{lemma}
\begin{proof}
By (\ref{UVprox1}) and (\ref{UVprox2})
we have \begin{align*}
0&\in \nabla_UF(\widetilde{U}^k,V^k)+L^k_1(U^{k+1}-\widetilde{U}^k)+\partial R_1^k(U^{k+1}),\\0&\in \nabla_VF(U^{k+1},\widetilde{V}^k)+L^k_1(V^{k+1}-\widetilde{V}^k)+\partial R_2^k(V^{k+1}).
\end{align*}
Let
\begin{align*}
P^{k+1}_1&=\nabla_UF(U^{k+1},V^{k+1})-\nabla_UF(\widetilde{U}^k,V^{k})-L^k_1(U^{k+1}-\widetilde{U}^k),   \\
P^{k+1}_2&=\nabla_VF(U^{k+1},V^{k+1})-\nabla_VF(U^{k+1},\widetilde{V}^{k})-L^k_2(V^{k+1}-\widetilde{V}^k).
\end{align*}
Then $\begin{bmatrix}
P^{k+1}_1, 
P^{k+1}_2
\end{bmatrix}\in \partial \Phi_\lambda^{k+1}(U^{k+1},V^{k+1}).$
By (\ref{LI1}) and (\ref{LI2}), \begin{align*}
\|P^{k+1}_1\|_F&\le \|\nabla_UF(U^{k+1},V^{k})-\nabla_UF(\widetilde{U}^k,V^{k})\|_F+L^k_1\|U^{k+1}-\widetilde{U}^k\|_F\\& +\|\nabla_UF(U^{k+1},V^{k+1})-\nabla_UF(U^{k+1},V^{k})\|_F
\\& \le (L^*+L^k_1)\|U^{k+1}-\widetilde{U}^k\|_F+ L^*\|V^{k+1}-V^{k}\|_F
\\& \le 2L_{H}(\|U^{k+1}-U^k\|_F+\beta_k\|U^{k}-U^{k-1}\|_F)+ L^*\|V^{k+1}-V^{k}\|_F,
\end{align*}
and similarly 
$$\|P^{k+1}_2\|_F\le 2L_{H}(\|V^{k+1}-V^k\|_F+\beta_k\|V^{k}-V^{k-1}\|_F),$$
this implies 
$$\operatorname{dist}(0,\partial \Phi^{k+1}_\lambda (U^{k+1},V^{k+1})) \le 3L_HT^{k}+2\beta_kL_HT^{k-1}.$$
\end{proof}

Now we establish the global convergence under KL property.
\begin{theorem}\label{theoremKL}
Assume Assumptions \ref{A1}, \ref{A2}, \ref{A3}  and Condition \ref{C2} are satisfied, and we choose parameter $\delta_k$ such that $\sum_{t=0}^{\infty}\sqrt{\delta_t}<+\infty$. Let $(\bar{U},\bar{V})$ be a limit point of $\{(U^{k},V^{k})\}_{k=0}^{+\infty}$, if the following condition holds
\begin{enumerate}
\item $\nabla_UF(U,V)$ and $\nabla_VF(U,V)$ are Lipschitz continuous within $\mathbb{B}_{4\rho}((\bar{U},\bar{V}))$ with respect to $(U,V)$.
\item There exist sufficiently large $K$, for all $k\ge K$, 
\begin{equation}\label{Con3}
\Phi_{\lambda}^k(U^k,V^k)-\Phi_{\lambda}^k(\bar{U},\bar{V})\ge \Phi_{\lambda}^{k+1}(U^{k+1},V^{k+1})-\Phi_{\lambda}^{k+1}(\bar{U},\bar{V})>0.  
\end{equation}
\item  There exist sufficiently large $K$, for all $k\ge K$, $\Phi^k_\lambda$ has the KL property with same $\varphi$ with parameters $\rho,\eta,\theta,\mu$ at $(\bar{U},\bar{V})$.
\end{enumerate}
then $\{(U^{k},V^{k})\}_{k=0}^{+\infty}$ 
converges to $(\bar{U},\bar{V})$, which is a critical point. 
\end{theorem}

\begin{proof}
Since $(\bar{U},\bar{V})$ is a limit point of $\{(U^{k},V^{k})\}_{k=0}^{+\infty}$, and according to Theorem \ref{theorem4.1}, we have $\underset{k\to \infty}{\lim}\|W^k-W^{k+1}\|=0$, then there exist sufficiently large $K$ such that $(U^k,V^k)$ is sufficiently close to $(\bar{U},\bar{V})$ and in $\mathbb{B}_\rho((\bar{U},\bar{V}))$ for $k=K-1,K,K+1$, and also the difference $\|W^k-W^{k-1}\|_F$ can be sufficiently close to zero for $k=K,K+1$. Since $K$ is sufficiently large, then we can assume $\lambda^k=\lambda_0$ for $k\ge K-1$. Since $\delta_k\to 0$ as $k\to +\infty$, by the continuity of objective function,  both $|\Phi^{K+1}_{\lambda}(U^{K+1},V^{K+1})-\Phi^{K+1}_{\lambda}(\bar{U},\bar{V})|$ and $\phi(\Phi^{K+1}_{\lambda}(U^{K+1},V^{K+1})-\Phi^{K+1}_{\lambda}(\bar{U},\bar{V}))$ can be sufficiently small. We assume $\Phi^{K+1}_{\lambda}(U^{K+1},V^{K+1})-\Phi^{K+1}_{\lambda}(\bar{U},\bar{V})<\eta$.
Let $\phi_k=\phi(\Phi_{\lambda}^{k}(U^{k},V^{k})-\Phi_{\lambda}^{k}(\bar{U},\bar{V}))$. Denote $$S^{k}=(\frac{L_1^{k}}{8}\|U^{k}-U^{k+1}\|^2_F+\frac{L_2^{k}}{8}\|V^{k}-V^{k+1}\|^2_F)^{\frac{1}{2}}.$$
Denote 
\begin{align*}
M&=\frac{1}{1-\delta_\beta-C}(S^{K-1}+S^K+\sqrt{2\lambda_0\mathcal{W}}\sum_{t=K}^{\infty}\sqrt{\zeta^{t}}+\frac{18L_H^2}{C\sqrt{l}}\phi_{K+1}),
\end{align*}
where $\zeta^{t}=2\max\{1-(1+\delta_t)^{-q-1},(1-\delta_t)^{-q-1}-1\}$ and $\mathcal{W}$ is the upper bound of $\frac{1}{q}\|U^k\|_{2,q}^q+\frac{1}{p}\|V^k\|_{2,p}^p$ for all $k$. Then $\zeta^{t}\le 32\delta_k$ and $\sum_{t=0}^{\infty}\sqrt{\zeta^{t}}<\infty$.
Let $\bar{M}=2\sqrt{\frac{2}{l}} M+\|W^{K-1}-\bar{W}\|_F$. Here $C$ is a constant that it satisfies $C< 1-\delta_\beta$ and assume $\bar{M}\le \rho$ (note that $\bar{M}$ can be  sufficiently small).

Using KL inequality and Lemma \ref{5.2}, we have 
\begin{equation}\label{530}
\phi'(\Phi_{\lambda}^{K+1}(U^{K+1},V^{K+1})-\Phi_{\lambda}^{K+1}(\bar{U},\bar{V}))(3L_HT^{K}+2\beta_K L_HT^{K-1})\ge 1.
\end{equation}
In addition, by (\ref{F1}) and (\ref{F2}) in Theorem \ref{theorem4.1} we have
\begin{align}\label{531}
&\Phi_{\lambda}^{K+1}(U^{K+1},V^{K+1})-\Phi_{\lambda}^{K+2}(U^{K+2},V^{K+2})\nonumber\\&\ge  
\frac{L_1^{K+1}}{8}\|U^{K+1}-U^{K+2}\|^2_F+\frac{L_2^{K+1}}{8}\|V^{K+1}-V^{K+2}\|^2_F\nonumber
\\&- \frac{L_1^{K}\delta^2_\beta}{8}\|U^{K}-U^{K+1}\|^2_F-\frac{L_2^{K}\delta^2_\beta}{8}\|V^{K}-V^{K+1}\|^2_F- \zeta^{K}\lambda_0\mathcal{W}.
\end{align} 
Note that \begin{align}\label{532}
\phi_{K+1}-\phi_{K+2}\ge &\phi'(\Phi_{\lambda}^{K+1}(U^{K+1},V^{K+1})-\Phi_{\lambda}^{K+1}(\bar{U},\bar{V}))\nonumber\\&\times \big(\Phi_{\lambda}^{K+1}(U^{K+1},V^{K+1})-\Phi_{\lambda}^{K+1}(\bar{U},\bar{V}))\\&\quad\quad -(\Phi_{\lambda}^{K+2}(U^{K+2},V^{K+2})-\Phi_{\lambda}^{K+2}(\bar{U},\bar{V})))\big)\nonumber\\
\ge &\phi'(\Phi_{\lambda}^{K+1}(U^{K+1},V^{K+1})-\Phi_{\lambda}^{K+1}(\bar{U},\bar{V}))\nonumber\\&\times(\Phi_{\lambda}^{K+1}(U^{K+1},V^{K+1})-\Phi_{\lambda}^{K+2}(U^{K+2},V^{K+2})-\zeta^K\lambda_0\mathcal{W}).
\end{align}
By (\ref{530}), (\ref{531}) and (\ref{532}), we have
\begin{align*}
(3L_HT^{K}+2\beta_K L_HT^{K-1})(\phi_{K+1}-\phi_{K+2})\ge (S^{K+1})^2-\delta^2_\beta (S^{K})^2-2\zeta^{K}\lambda_0\mathcal{W}.
\end{align*}
Then
\begin{align}\label{5.3}
S^{K+1}&\le (\delta^2_\beta (S^{K})^2+(3L_HT^{K}+2\beta_K L_HT^{K-1})(\phi_{K+1}-\phi_{K+2})+2\zeta^{K}\lambda_0\mathcal{W})^{\frac{1}{2}}\nonumber
\\& \le \delta_\beta S^{K}+(3L_H(T^{K}+T^{K-1})(\phi_{K+1}-\phi_{K+2}))^{\frac{1}{2}}+(2\zeta^{K}\lambda_0\mathcal{W})^{\frac{1}{2}}\nonumber
\\& \le \delta_\beta S^{K}+C(S^{K}+S^{K-1})+\frac{18L_H^2}{C\sqrt{l}}(\phi_{K+1}-\phi_{K+2})+\sqrt{2\zeta^{K}\lambda_0\mathcal{W}}.
\end{align}
This implies
\begin{align*}
S^{K+1}+S^{K}+S^{K-1}& \le \delta_\beta S^{K}+C(S^{K}+S^{K-1})+\frac{18L_H^2}{C\sqrt{l}}(\phi_{K+1}-\phi_{K+2})\\&+\sqrt{2\zeta^{K}\lambda_0\mathcal{W}}+S^{K}+S^{K-1}.
\end{align*}
Then
\begin{align*}
(1-\delta_\beta-C)(S^{K+1}+S^{K}+S^{K-1})& \le \frac{18L_H^2}{C\sqrt{l}}\phi_{K+1}+\sqrt{2\zeta^{K}\lambda_0\mathcal{W}}+S^{K}+S^{K-1}.
\end{align*}
This implies $S^{K+1}+S^{K}+S^{K-1}\le M$. Then we have
\begin{align*}
\|W^{K+2}-\bar{W}\|_F&\le \sum_{k=0}^{2}\|W^{K+k}-W^{K+k-1}\|_F+\|W^{K-1}-\bar{W}\|_F\\& \le \sqrt{\frac{8}{l}}
(S^{K+1}+S^{K}+S^{K-1})+\|W^{K-1}-\bar{W}\|_F\le \bar{M}\le \rho .
\end{align*}
This implies $(U^{K+2},V^{K+2})\in \mathbb{B}_\rho((\bar{U},\bar{V}))$. Then, by induction, we can show that $$\sum_{t=K-1}^{K+s}S^{t}\le M ~\text{and} ~(U^{K+s},V^{K+s})\in \mathbb{B}_\rho(\bar{U},\bar{V})$$
hold for all $s\ge 2$. 
This implies $\sum_{t=K-1}^{+\infty}S^{t}\le M<+\infty$, note that
\begin{equation*}
\|W^k-\bar{W}\|_F\le \sum_{t=k}^{+\infty}\|W^k-W^{k+1}\|_F\le \sqrt{\frac{8}{l}}\sum_{t=k}^{+\infty}S^t,
\end{equation*}
then we have $W^k\to \bar{W}$ when $k\to +\infty$.
\end{proof}

\begin{remark}\label{remark5.5}
By Proposition \ref{single-value}, under mild assumption, there exist $\bar{\beta}_k>0$, $\bar{\delta}_{k-1}>0$ and $\bar{\delta}_{k}>0$ such that for any $\beta_k\in [0,\bar{\beta}_k]$, $\delta_{k-1}\in[0,\bar{\delta}_{k-1}]$ and $\delta_{k}\in[0,\bar{\delta}_{k}]$, we have \begin{equation}\label{551}
\Phi_{\lambda}^k(U^k,V^k)-\Phi_{\lambda}^k(\bar{U},\bar{V})\ge \Phi_{\lambda}^{k+1}(U^{k+1},V^{k+1})-\Phi_{\lambda}^{k+1}(\bar{U},\bar{V}).
\end{equation}
Though we do not know the value of $(\bar{U},\bar{V})$, we could have an upper bound of $\frac{\lambda_0}{q}\|\bar{U}\|^q_{2,q}+\frac{\lambda_0}{p}\|\bar{V}\|^p_{2,p}$, namely $C$. Similar to the proof of  Proposition \ref{single-value}, there exist $\tilde{\beta}_k>0$, $\tilde{\delta}_{k-1}>0$ and $\tilde{\delta}_{k}>0$ such that for any $\beta_k\in [0,\tilde{\beta}_k]$, $\delta_{k-1}\in[0,\tilde{\delta}_{k-1}]$ and $\delta_{k}\in[0,\tilde{\delta}_{k}]$, we have 
\begin{equation}\label{552}
\Phi_{\lambda}^k(U^k,V^k)\ge \Phi_{\lambda}^{k+1}(U^{k+1},V^{k+1})+16(\delta_{k-1}+\delta_{k})C,
\end{equation}
then we can find suitable parameters by backtracking to satisfy (\ref{552}), then (\ref{551}) also holds. Furthermore, if we assume $\Phi_{\lambda_0}(U^k,V^k)>\Phi_{\lambda_0}(\bar{U},\bar{V})$ for all $k$, note that $\delta_k\to 0$ when $k\to 0$, then there exist $\hat{\delta}_k>0$ such that for any $\delta_k\in[0,\hat{\delta}_k]$, we have $$\Phi_{\lambda}^{k+1}(U^{k+1},V^{k+1})-\Phi_{\lambda}^{k+1}(\bar{U},\bar{V})>0.$$
\end{remark}

\begin{remark}
Assume that function $F(U,V)$ is semi-algebraic (both two examples in Subsection \ref{subsection3.3} are satisfied). Note that $\sum_{i=1}^{d}\frac{\lambda^k_{1,t}}{q}\|U_{:,t}\|_2^q+\sum_{t=1}^{d}\frac{\lambda^k_{2,t}}{p}\|V_{:,t}\|_2^p$ are also semi-algebraic for all $k$, then $\Phi^k_{\lambda}$ are semi-algebraic for all $k$. Then by Remark 3.2 \cite{bolte2007lojasiewicz}, $\Phi^k_{\lambda}$ has KL property with parameters $\theta_k\in[0,1),\rho_k,\mu_k,\eta_k$ at $(\bar{U},\bar{V})$. Then if $\underset{k\to+\infty} {\lim\sup}~ \theta_k <1$, $\underset{k\to+\infty} {\lim\inf}~ \rho_k, \underset{k\to+\infty} {\lim\inf}~ \eta_k >0$ and $\underset{k\to+\infty} {\lim\sup}~ \mu_k <+\infty$, condition 3 in Theorem \ref{theoremKL} holds.
\end{remark}

\begin{remark}\label{remark5.6}
Assume that $\Phi_{\lambda_0}$ has the KL property with parameter $\rho,\eta,\theta,\mu$ at $(\bar{U},\bar{V})$. If for sufficiently large $k$, we have \begin{equation}\label{5.6}
\delta_{k-1}\le \min\{\frac{1}{16C_1}(3L_HT^{k-1}+2\beta_{k-1}L_HT^{k-2}),\frac{1}{32C_2}(\Phi_{\lambda_0}(U^k,V^k)-\Phi_{\lambda_0}(\bar{U},\bar{V}))\},
\end{equation} where $C_1=\lambda_0(\|U^kD_{U^k}\|_F+\|V^kD_{V^k}\|_F)$ and $C_2=\frac{\lambda_0}{q}(\|U^k\|^q_{2,q}+\|\bar{U}\|^q_{2,q})+\frac{\lambda_0}{p}(\|V^k\|^p_{2,p}+\|\bar{V}\|^p_{2,p})$, then $$\Phi^k_{\lambda}(U^k,V^k)-\Phi^k_{\lambda}(\bar{U},\bar{V})\le \frac{3}{2} (\Phi_{\lambda_0}(U^k,V^k)-\Phi_{\lambda_0}(\bar{U},\bar{V}))$$ and $$\operatorname{dist}(0,\partial \Phi_{\lambda_0} (U^{k},V^{k}))\le \operatorname{dist}(0,\partial \Phi^{k}_\lambda (U^{k},V^{k})) + 3L_HT^{k-1}+2\beta_{k-1}L_HT^{k-2}.$$ This implies \begin{equation*}
\varphi'\bigl(\Phi^k_{\lambda}(U^k,V^k)-\Phi^k_{\lambda}(\bar{U},\bar{V})\bigr) \ge \frac{2^\theta}{3^\theta}\varphi'\bigl(\Phi_{\lambda_0}(U^k,V^k)-\Phi_{\lambda_0}(\bar{U},\bar{V})\bigr).
\end{equation*}
Then (\ref{530}) holds up to a constant multiple $\frac{2^{\theta-1}}{3^\theta}$ on the right hand side. It follows that if $\delta_{k-1}$ satisfies (\ref{5.6}) for all
sufficiently large $k$, condition 3 in Theorem \ref{theoremKL} can be weakened to that $\Phi_{\lambda_0}$ has the KL property at $(\bar{U},\bar{V})$.  

Note that if we assume $W^{k-1}\ne W^{k}$ and $\Phi_{\lambda_0}(U^k,V^k)>\Phi_{\lambda_0}(\bar{U},\bar{V})$ for all
sufficiently large $k$, then the right hand side of (\ref{5.6}) is strictly larger than $0$.
\end{remark}

We next present the following proposition, which shows that under mild assumptions, Conditions \ref{C2} and (\ref{Con3}) hold. The proof of Proposition \ref{single-value} can be found in Appendix \ref{APsingle-value}. 
\begin{proposition}\label{single-value}
Assume $\operatorname{prox}_{\frac{\lambda^k}{qL_{1}^k}\|\cdot\|_2^q}$ and  $\operatorname{prox}_{\frac{\lambda^k}{pL_{2}^k}\|\cdot\|_2^p}$ are single-valued near 
$H^k(U^k,V^k)$ and $\bar{H}^k(U_{\text{E}}^{k+1}(0),V^k)$, respectively, and at least one of 
\begin{equation}
U^k\notin \operatorname{prox}_{\frac{\lambda^k}{qL_{1}^k}\|\cdot\|_2^q}\big(H^k(U^k,V^k)\big)~ \text{and}~V^k\notin \operatorname{prox}_{\frac{\lambda^k}{qL_{2}^k}\|\cdot\|_2^p}\big(\bar{H}^k(U_{\text{E}}^{k+1}(0),V^k)\big)\end{equation} 
holds. Then for sufficient large $k$, there exist $\bar{\beta}_k>0$, $\bar{\delta}_{k-1}>0$, $\bar{\delta}_{k}>0$ such that
for any $\beta_k\in [0,\bar{\beta}_k]$ and $\delta_{k}\in[0,\bar{\delta}_{k}]$, we have 
\begin{equation}
\Phi^{k+1}_\lambda(U^{k},V^{k})\ge \Phi^{k+1}_\lambda(U^{k+1}(\beta),V^{k+1}(\beta)),
\end{equation}
for any $\beta_k\in [0,\bar{\beta}_k]$, $\delta_{k-1}\in[0,\bar{\delta}_{k-1}]$ and $\delta_{k}\in[0,\bar{\delta}_{k}]$, we have
\begin{equation}
\Phi_{\lambda}^k(U^k,V^k)-\Phi_{\lambda}^k(\bar{U},\bar{V})\ge \Phi_{\lambda}^{k+1}(U^{k+1}(\beta),V^{k+1}(\beta))-\Phi_{\lambda}^{k+1}(\bar{U},\bar{V}). 
\end{equation}
Here denote \begin{equation}
H^k(U,V)=U-\frac{1}{L_{1}^k}\nabla_{U} 
F(U,V), ~\bar{H}^k(U,V)=V-\frac{1}{L_{2}^k}\nabla_{V} 
F(U,V). 
\end{equation}
Let \begin{equation}
\widetilde{U}^k(\beta_k)=U^k+\beta_k(U^k-U^{k-1}), ~\widetilde{V}^k(\beta_k)=V^k+\beta_k(V^k-V^{k-1}).
\end{equation}
Let 
\begin{align}
U_{\text{E}}^{k+1}(\beta_k)&=\operatorname{prox}_{\frac{\lambda^k}{qL_{1}^k}\|\cdot\|_2^q}(H^k(\widetilde{U}^k(\beta_k),V^k)),
\\ V_{\text{E}}^{k+1}(\beta_k)&=\operatorname{prox}_{\frac{\lambda^k}{qL_{2}^k}\|\cdot\|_2^q}(\bar{H}^k(U_{\text{E}}^{k+1}(\beta_k),\widetilde{V}^k(\beta_k))).
\end{align}
Let $U^{k+1}(\beta_k)$ and $V^{k+1}(\beta_k)$ be an approximate solution of $\operatorname{prox}_{\frac{\lambda^k}{qL_{1}^k}\|\cdot\|_2^q}(H^k(\widetilde{U}^k(\beta_k),V^k))$ and $\operatorname{prox}_{\frac{\lambda^k}{qL_{2}^k}\|\cdot\|_2^q}(\bar{H}^k(U^{k+1}(\beta_k),\widetilde{V}^k(\beta_k)))$, respectively, which are computed by Proposition \ref{inexactprox} with parameter $\delta_k$. 
\end{proposition}

Next, we propose the following lemma to study the convergence rate, the proof of Lemma \ref{SL} can be found in Appendix \ref{APSL}. This lemma also demonstrates that introducing an extra linearly convergent term in (\ref{Rat}) will not alter the convergence rate of the sequence.
\begin{lemma}\label{SL}
Suppose that a positive sequence $\{A_k\}$ satisfies $\underset{k\to +\infty}{\lim} A_k \to 0$ and there exists $K$ such that for $k \ge K$, it holds
\begin{equation}\label{Rat}
A_{k-1} \le a (A_{k-1} - A_{k+1}) + b (A_{k-1} - A_{k+1})^{(1-\theta)/\theta} + c \beta^k,
\end{equation}
where $0<\beta\le\sqrt{\frac{a+b-1}{a+b}}$, $a>1$, $b,c>0$, $\theta \in (0,1)$. Then we have the following conclusions:
\begin{enumerate}
    \item For $0<\theta \le \frac{1}{2}$, there exist $C>0$ and $0<\alpha<1$ such that for all sufficiently large $k$, we have $ A_k \le C \alpha^k$.
    \item For $\frac{1}{2} <\theta <1$, there exists $C>0$ such that for all sufficiently large $k$, we have $A_k \le C k^{-(1-\theta)/(2\theta-1)}$.
\end{enumerate}
\end{lemma}

By Lemma \ref{SL}, we establish the convergence rate under KL property.
\begin{theorem}\label{theoremrate}
Under the assumption of Theorem \ref{theoremKL}, if we choose parameter $\delta_k\le \gamma \beta^{2k}$ and $0<\beta\le\sqrt{\frac{a+b-1}{a+b}}$, where $a,b$ see (\ref{ab}),
we have 
\begin{enumerate}
    \item For $0<\theta \le \frac{1}{2}$, there exist $C>0$ and $0<\alpha<1$ such that for all sufficiently large $k$, we have $ \|W^k-\bar{W}\|_F \le C \alpha^k$.
    \item For $\frac{1}{2}<\theta <1$, there exists $C>0$ such that for all sufficiently large $k$, we have $\|W^k-\bar{W}\|_F \le C k^{-(1-\theta)/(2\theta-1)}$.
\end{enumerate}
\end{theorem}
\begin{proof}
We use the same notation as in Theorem \ref{theoremKL}. Using KL inequality and Lemma \ref{5.2}, we have 
\begin{align*}
\mu(1-\theta)(\Phi_{\lambda}^{K+1}(U^{K+1},V^{K+1})-\Phi_{\lambda}^{K+1}(\bar{U},\bar{V}))^{-\theta}&\ge (3L_HT^{K}+2\beta_K L_HT^{K-1})^{-1}\\&\ge (3L_H\sqrt{\frac{8}{l}}(S^{K}+S^{K-1}))^{-1}.
\end{align*}
Note that $\phi_k=\phi(\Phi_{\lambda}^{k}(U^{k},V^{k})-\Phi_{\lambda}^{k}(\bar{U},\bar{V}))=\mu(\Phi_{\lambda}^{k}(U^{k},V^{k})-\Phi_{\lambda}^{k}(\bar{U},\bar{V}))^\theta$. Then \begin{equation}
\phi_{K+1} \ge \mu(\mu(1-\theta)3L_H\sqrt{\frac{8}{l}}(S^{K}+S^{K-1}))^{\frac{1-\theta}{\theta}}.
\end{equation}
Then by (\ref{5.3}), we have \begin{align*}
\sum_{k=K-1}^{\infty}S^k&\le \frac{1}{1-\delta_\beta-C}(S^{K}+S^{K-1})+ \frac{\mu(\mu(1-\theta)3L_H\sqrt{\frac{8}{l}})^{\frac{1-\theta}{\theta}}}{1-\delta_\beta-C}(S^{K}+S^{K-1})^{\frac{1-\theta}{\theta}}\\&+\frac{\sqrt{2\lambda_0\mathcal{W}}}{1-\delta_\beta-C}\sum_{t=K}^{\infty}\sqrt{\zeta^{K}}.
\end{align*}
Denote \begin{equation}\label{ab}
a=\frac{1}{1-\delta_\beta-C},~ b=\frac{\mu(3\sqrt{\frac{8}{l}}\mu(1-\theta)L+H)^{\frac{1-\theta}{\theta}}}{1-\delta_\beta-C},~c= \frac{8\sqrt{\lambda_0\mathcal{W}}}{(1-\delta_\beta-C)(1-\beta)},
\end{equation}
and $A_{K-1}=\sum_{k=K-1}^{\infty}S^k$.
Then $$A_{K-1}\le a(A_{K-1}-A_{K+1})+ b(A_{K-1}-A_{K+1})^{\frac{1-\theta}{\theta}}+c\beta^K,$$
holds for sufficient large $K$. Note that 
$$\|W^K-\bar{W}\|_F\le \sum_{t=K}^{+\infty}\|W^k-W^{k+1}\|_F \le \sqrt{\frac{8}{l}}\sum_{t=K}^{+\infty}S^t.$$
Then by Lemma \ref{SL}, we obtain the result.
\end{proof}

In particular, under certain conditions, when $F$ is a least-square loss function, we can establish the linear convergence of the sequence.
\begin{proposition}\label{Convergence_result}
Consider the low rank matrix recovery problem (\ref{p1}) with least-square loss function. If the assumptions of Theorem \ref{KL} and Theorem \ref{theoremKL} hold, then $\Phi_{\lambda_0}$ has the KL property with exponent $1/2$ at $(\bar{U},\bar{V})$, if for all sufficiently large $k$, (\ref{5.6}) holds, then by Remark \ref{remark5.6} and Theorem \ref{theoremrate}, $\{\|W^k-\bar{W}\|_F\}_{k=0}^{+\infty}$ converges linearly.
\end{proposition}
\section{Experimental Results}\label{section6}
In this section, we present numerical experimental results obtained by applying the Inexact Proximal Alternating Linearized Minimization Method (Algorithm \ref{PAMM}) to solve the group-sparse factorized Schatten-q norm regularized low-rank matrix recovery problem. These results are used to validate the performance of our algorithm and to verify the theoretical properties of the factorization model.

In what follows, we conduct numerical experiments to solve this factorized low-rank matrix recovery problem:
\begin{equation}\label{M6}
\min_{U\in \mathbb{R}^{m\times d},\; V\in \mathbb{R}^{n\times d}} \Phi_\lambda (U,V):=\frac{1}{|\Omega|}\|\mathcal{P}_\Omega(UV^T-M)\|_F^2
+\frac{\lambda}{q}\|U\|_{2,q}^q+\frac{\lambda}{p}\|V\|_{2,p}^p.
\end{equation}
Given the observed matrix $M = X^* + \mathcal{N}$, denote $\mathcal{P}_\Omega$ as the projection operator onto the set of observed entries. Here, $X^*$ is the true matrix, and $\mathcal{N}$ represents additive noise with entries independently drawn from a Gaussian distribution. We use ``OR'' as the abbreviation for the observation rate. The noise level is quantified as $\sigma = \frac{\|\mathcal{P}_\Omega(\mathcal{N})\|_F}{\|X^*\|_F}$.  
For the sequence of iterates $\{(U^k, V^k)\}_k$ generated by Algorithm \ref{PAMM}, we define the relative root mean square error (RMSE) as
$$\text{RMSE} = \frac{\|U^k (V^k)^T - X^*\|_F}{\|X^*\|_F}.$$ 
All numerical experiments are conducted in MATLAB (2022a) on a laptop of 8G of memory and Intel Core i5 2.3Ghz CPU. 

\subsection{Low-rank Matrix Completion on Simulated Data}

In this subsection, we conduct numerical experiments on simulated data to validate the effectiveness of Algorithm~\ref{PAMM} and the theoretical properties of the factorization model of Schatten-$q$ norm regularized low-rank matrix recovery problem. We first investigate the influence of the parameters $(p,q)$ (i.e., different factorized formulations and the choice of parameter $q$ of the Schatten-$q$ norm) on the experimental results obtained by Algorithm~\ref{PAMM} for solving model~(\ref{M6}). To this end, we randomly generate a low-rank matrix $X^*$ of size $500 \times 400$ with true rank $r = 15$, and construct a Gaussian noise matrix $\mathcal{N}$ with noise level $\sigma = 0.1$. The observation ratio is set to $OR=\frac{|\Omega|}{mn} = 0.4$. We then perform experiments by grid search over $q,p \in (0,2]$ with a grid size of $0.1$.

\begin{figure}
    \centering
    \includegraphics[width=1\linewidth]{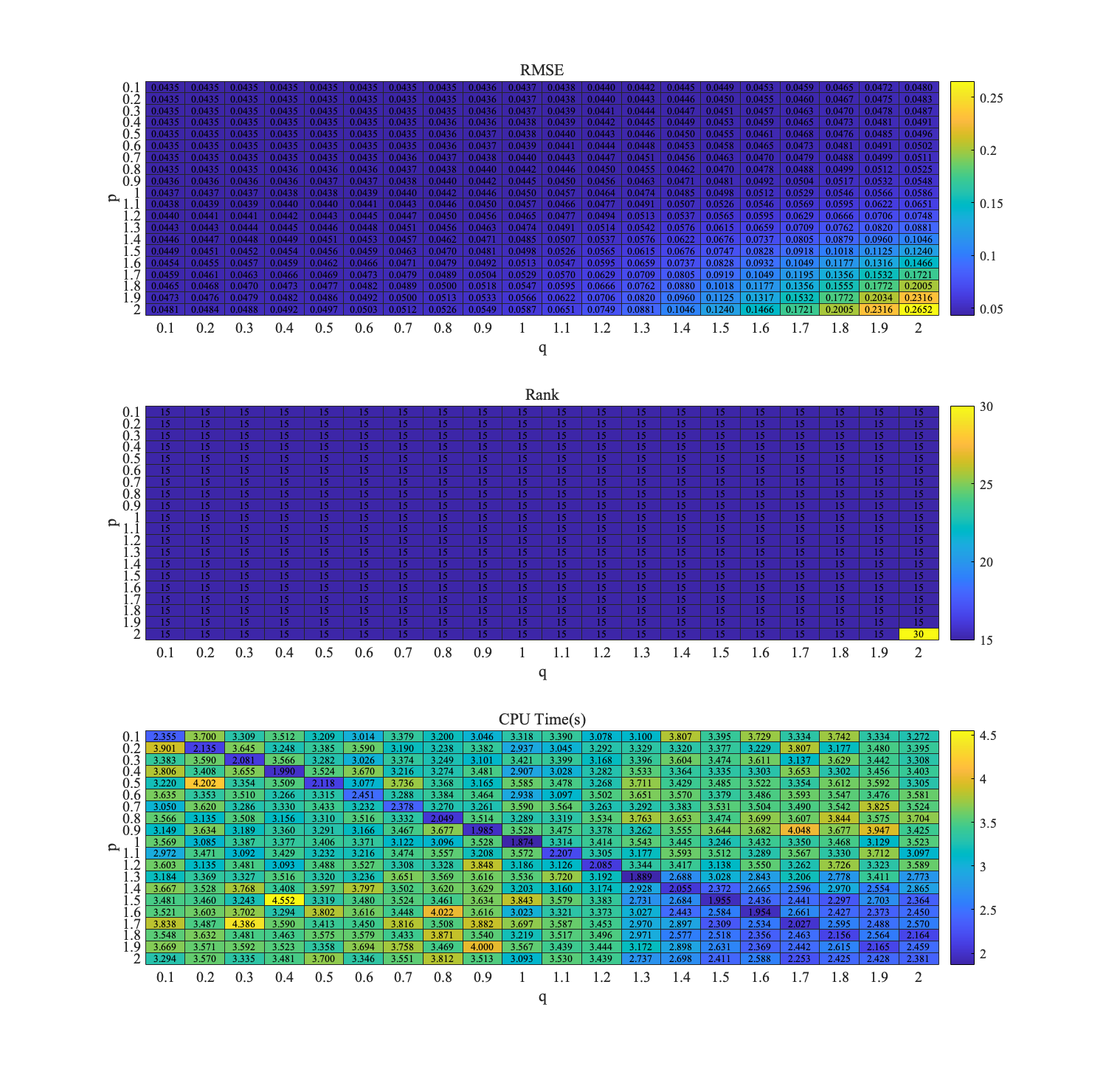}
    \caption{Heatmaps of RMSE, Rank and CPU time (s) for different choice of $(p,q)$ in model (\ref{M6}), where $(p,q)$ are selected from the grid $(0,2] \times (0,2]$ with a grid step size of 0.1. A darker color indicates a smaller error or shorter computational time, and dark blue signifies that the rank of the recovered matrix matches the true rank for the rank heat map.}
    \label{fig:pq}
\end{figure}

In Algorithm \ref{PAMM}, the parameters are configured as follows: if $q,p\ge 1$, set $\lambda^0=0.1-\frac{0.07pq}{p+q}$, else set $\lambda^0=0.25-\frac{0.37pq}{p+q}$, $\rho=0.996$, $\lambda_0=10^{-3}$, $L_1^0=L^0_2=15$ and $L^k_{1}=L^k_{2}=\max\{15\rho^k,0.1\}$, and the initial rank is set as $d=30$. Given the singular value decomposition of $\mathcal{P}_\Omega(M)$ as $\bar{U}\Sigma \bar{V}^T$, the initial iteration point is set to $(\bar{U}_{1:d}\Sigma^{\frac{1}{2}}_{d},\bar{V}_{1:d}\Sigma^{\frac{1}{2}}_{d})$, where
we suppose that the singular values $\{\Sigma_{ii}\}_i$ are decreasing and $\bar{U}_{1:d}$ and $\bar{V}_{1:d}$ denote the first $d$ columns of $\bar{U}$ and $\bar{V}$, respectively, and $\Sigma_{d}=\operatorname{diag}(\Sigma_{11},\ldots,\Sigma_{dd})$. Additionally, we use  Nesterov extrapolation and set $\beta_0=0$ and $\beta^*=\frac{1}{\sqrt{2}}$, with the maximum number of iterations $k_{\max}=2000$ and the termination condition parameter $\epsilon=10^{-7}$.

The experimental results are depicted in Figure~\ref{fig:pq},
which displays heatmaps, from left to right, the effects of different choices of $(p,q)$ on the the relative recovery error (RMSE), the rank of the recovered matrix, and the total computation time, respectively. Experimental results from the three figures above demonstrate that choosing relatively small values of $(p,q)$ yields smaller errors for the factorized model and for $(q,p)=(2,2)$, the factorized form of the nuclear norm~(\ref{nuclear}), yields the largest error; the factorized model combining our algorithm with adaptive rank adjustment technique exhibits strong rank-reduction capability, in particular, the rank can be reduced to the true rank as long as the parameters avoid values $(2,2)$; computational time is relatively short when $(p,q)$ lies within $[1.3,2] \times [1.3,2]$, and the best performance is achieved when $p=q$. This observation also justifies our choice of $p=q$ in the error bound analysis presented in Section 4. This also shows that our algorithm is adaptable to any parameters $p,q\in (0,2]$.


Next, we further demonstrate the specific rank reduction effect of our algorithm on the factorized model. To this end, we select six representative sets of parameter pairs, $(q,p)=$ $(0.1,0.1)$, $(1/2,1/2)$, $(1,1)$, $(4/3,4/3)$, $(8/5,8/5)$ and $(2,2)$, which correspond to the equivalent factorization of the Schatten‑1/20 norm, Schatten‑1/4 norm, Schatten‑1/2 norm, Schatten‑2/3 norm, and Schatten‑1 norm (nuclear norm), respectively, for further experiments.  For $\sigma=0.05,0.15$ and $OR=0.3,0.5$, we show the evolution of rank $d$ with iterations, as shown in Figure \ref{fig:pqc}. As can be seen from the figure, except for the $(q,p)=(2,2)$ case, all other parameter settings achieve good rank reduction and can recover the true rank.

\begin{figure}
    \centering
    \includegraphics[width=1\linewidth]{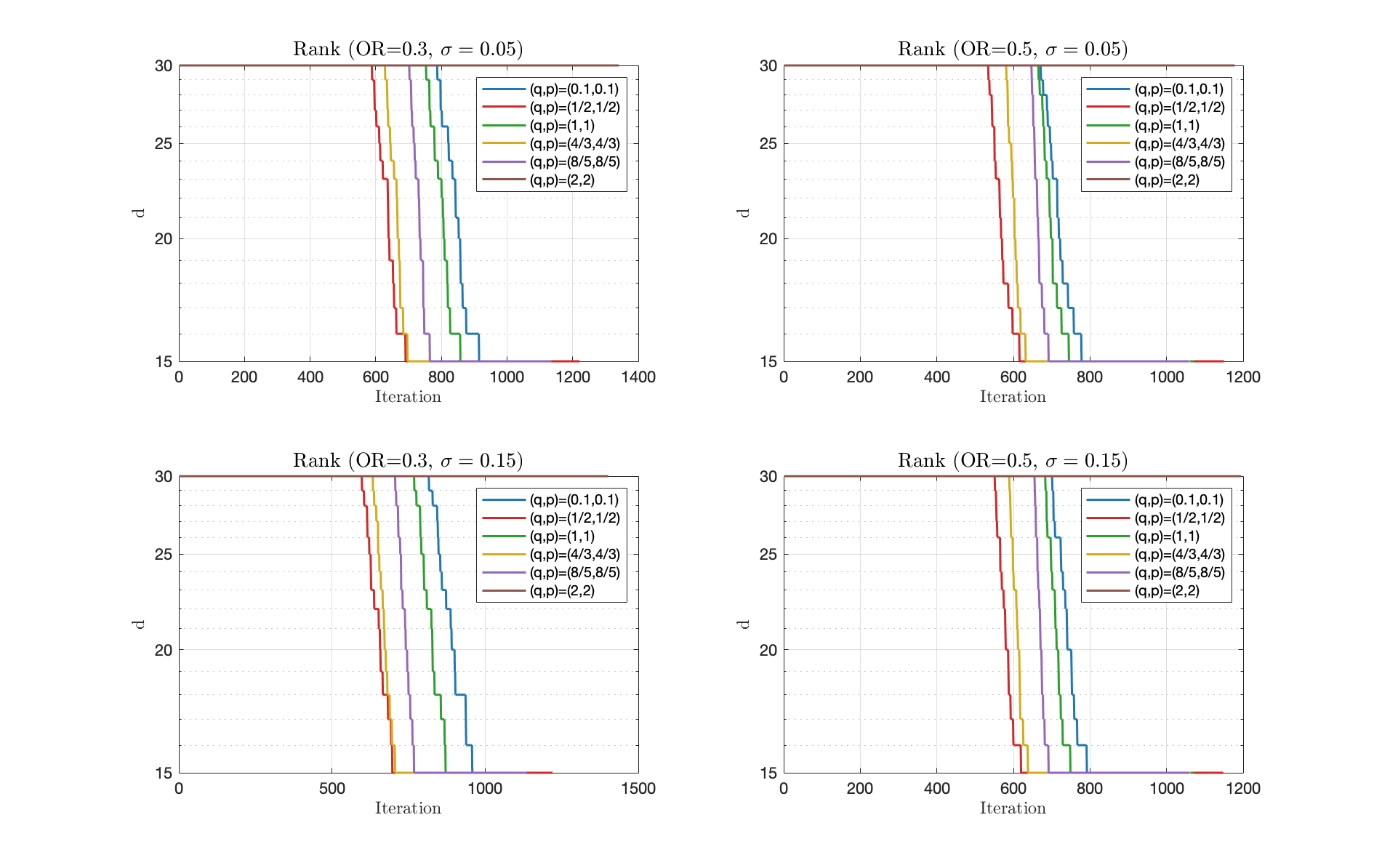}
    \caption{The evolution of Rank versus iterations for six representative parameter pairs, $(q,p)=$ $(0.1,0.1)$, $(1/2,1/2)$, $(1,1)$, $(4/3,4/3)$, $(8/5,8/5)$ and $(2,2)$.}
    \label{fig:pqc}
\end{figure}

While keeping the experimental parameters unchanged from the previous setup, we set the initial rank to $d=15$ and plot the evolution of the objective function $\Phi_\lambda$, the relative recovery error (RMSE), and the convergence indicator $\|W^k - W^*\|_F$ as functions of the iteration number. The results are shown in Figure~\ref{fig:pq1}, where $W^k = (U^k; V^k)$ and $W^*$ denotes the limit point of the sequence $\{W^k\}_k$, here taken as the iterate at termination.

\begin{figure}
    \centering
    \includegraphics[width=1\linewidth]{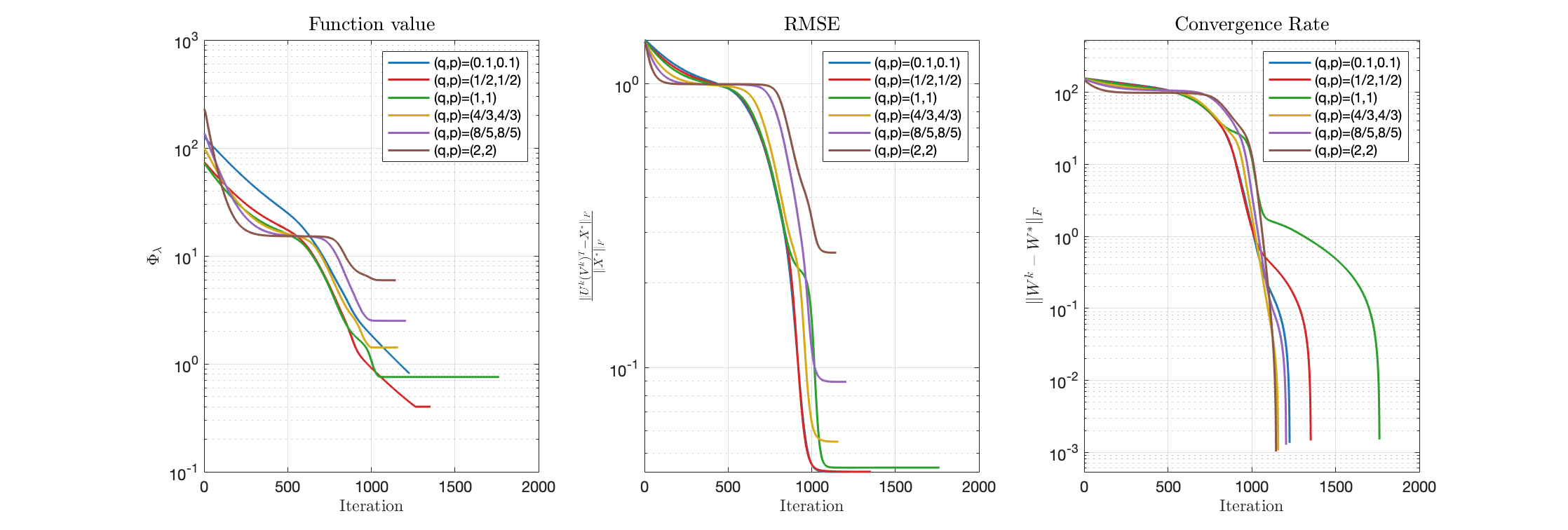}
    \caption{The curves of Function value, RMSE, and Convergence rate versus iterations for six representative parameter pairs, $(q,p)=$ $(0.1,0.1)$, $(1/2,1/2)$, $(1,1)$, $(4/3,4/3)$, $(8/5,8/5)$ and  $(2,2)$.}
    \label{fig:pq1}
\end{figure}

\begin{figure}
    \centering
    \includegraphics[width=1\linewidth]{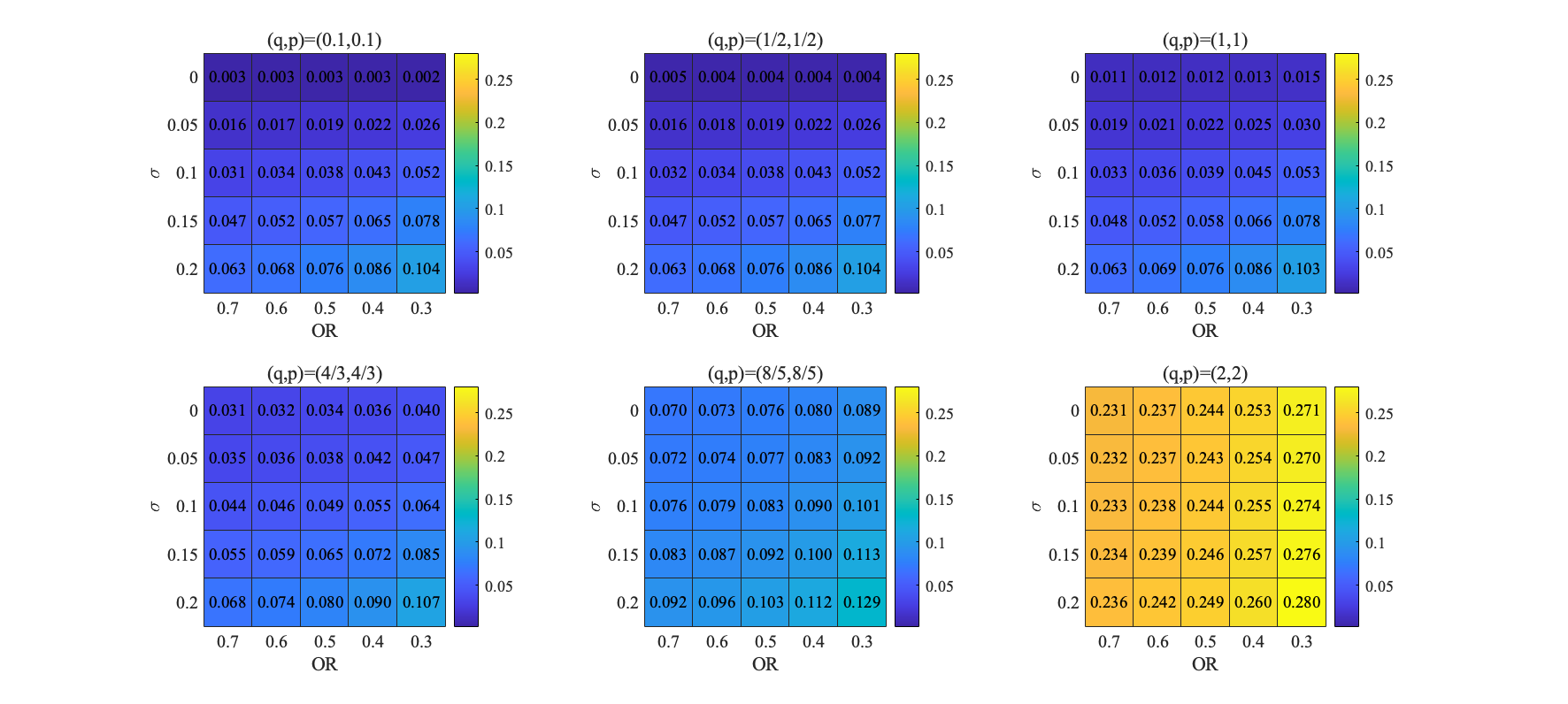}
    \caption{Comparison of RMSE with varying noise levels $\sigma$ ranging from $0$ to $0.2$ and varying observation rates $\frac{|\Omega|}{mn}$ ranging from $0.3$ to $0.7$ for six representative parameter pairs, $(q,p)=$ $(0.1,0.1)$, $(1/2,1/2)$, $(1,1)$, $(4/3,4/3)$, $(8/5,8/5)$ and $(2,2)$.}
    \label{fig:pq2}
\end{figure}

\begin{figure}
    \centering
    \includegraphics[width=1\linewidth]{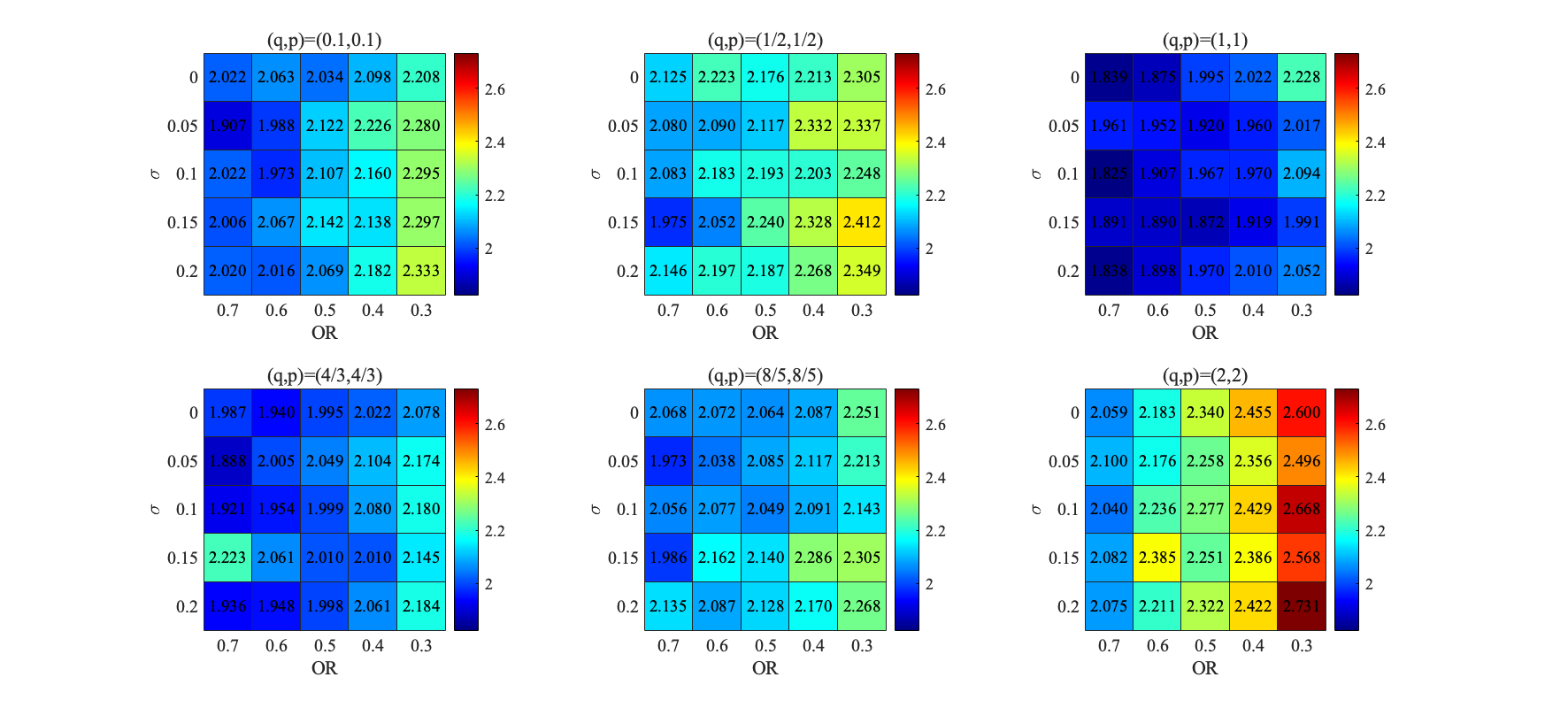}
    \caption{Comparison of CPU time(s) with varying noise levels $\sigma$ ranging from $0$ to $0.2$ and varying observation rates $\frac{|\Omega|}{mn}$ ranging from $0.3$ to $0.7$ for six representative parameter pairs, $(q,p)=$ $(0.1,0.1)$, $(1/2,1/2)$, $(1,1)$, $(4/3,4/3)$, $(8/5,8/5)$ and $(2,2)$.}
    \label{fig:pq3}
\end{figure}

From this Figure \ref{fig:pq1}, we observe that the objective function decreases monotonically as the iterations progress. This also verifies that the assumption $\Phi_{\lambda_0}(U^k,V^k)>\Phi_{\lambda_0}(\bar{U},\bar{V})$ stated in Remark \ref{remark5.5} and Remark \ref{remark5.6} is reasonable and can be satisfied. The choices $(q,p) = (0.1,0.1)$, $(1/2,1/2)$ and $(1,1)$ yield small relative recovery errors and outperform the other parameter settings. Moreover, the convergence curves further verify that near the limit point, the iterate sequence exhibits a linear convergence rate. This confirms that the function $\Phi_\lambda$ possesses the Kurdyka–{\L}ojasiewicz property of exponent $\frac{1}{2}$ at S-critical point under certain conditions (Theorem~\ref{KL}, Section~\ref{section4}), and that the iterates generated by Algorithm~\ref{PAMM} converge linearly under certain conditions (Proposition~\ref{Convergence_result}, Section~\ref{section5}).

Furthermore, we present heatmaps of the relative recovery error (RMSE) and CPU time for the above six parameter pairs under varying noise levels and observation rates. Specifically, the noise level $\sigma$ ranges from $0$ to $0.2$ in increments of $0.05$, and the observation rate $\frac{|\Omega|}{mn}$ ranges from $0.3$ to $0.7$ in increments of $0.1$. The parameter selection is the same as that of Figure~\ref{fig:pq} with $d=30$ and $k_{\max}=1500$. The corresponding results are shown in Figures~\ref{fig:pq2} and~\ref{fig:pq3}, respectively. All data are obtained from five repeated experiments and averaged.

From Figure~\ref{fig:pq2}, under low-noise and high-observation-rate settings, parameters $(q,p) = (0.1,0.1)$ and $(1/2,1/2)$ demonstrate a marked advantage in terms of estimation error over the other parameters. Under high-noise and low-observation-rate conditions, parameters $(q,p) = (0.1,0.1)$, $(1/2,1/2)$ and $(1,1)$ exhibit similar error levels, all of which outperform the remaining parameters. Parameter $(q,p) = (2,2)$, however, performing the worst. This aligns with the theoretical findings in Section~4 (Lemma~\ref{Lemma10} and Theorem~\ref{Theorem4.4}), which indicate that selecting $q,p < 2$ ensures column orthogonality at S-critical points more favorably than the choice $(2,2)$. The experimental results further corroborate the advantage of choosing $q,p < 2$. Figure~\ref{fig:pq3} shows that the parameter choices $(q,p) = (1,1)$ and $(4/3,4/3)$ achieve shorter computation times in average than the others, while $(q,p) = (2,2)$ is the most time-consuming in average.

\subsection{Low-rank Matrix Completion on Real Data}

In this subsection, we further compare the empirical performance of our algorithm to solve the factorized model with $(q,p)$= $(0.1,0.1)$, $(1/2,1/2)$, $(1,1)$, $(4/3,4/3)$, $(8/5,8/5)$ and $(2,2)$, using real-world datasets, namely the Jester joke dataset and the MovieLens dataset. The results are summarized in Tables~\ref{tab:algorithm_performance} and~\ref{tab:algorithm_performance2}, and all experimental results are averaged over ten independent trials.

Let $X^{\text{out}}=U^{\text{out}}(V^{\text{out}})^T$ be the matrix output by Algorithm \ref{PAMM}, and let $X^*$ be the truth matrix. We evaluate the performance using the normalized mean absolute error (NMAE), defined as
$$\text{NMAE} = \frac{\sum_{(i,j) \in \Omega \setminus \Gamma} |X_{i,j}^{\text{out}} - X_{i,j}^*|}{|\Omega \setminus \Gamma|(r_{\max} - r_{\min})},$$
where $\Omega$ is the index set of given entries, $\Gamma$ is the index set of observed entries, and $r_{\max}$ and $r_{\min}$ are the lower and upper bounds on the given entries, respectively. For real-world datasets, where many entries are unknown, the NMAE serves as an appropriate metric. 

In Table~\ref{tab:algorithm_performance}, we report results on the Jester joke dataset (available at \url{http://www.ieor.berkeley.edu/~goldberg/jester-data/}), which contains 4.1 million ratings of 100 jokes from 73,421 users. We consider two subsets: (1) Jester-1, consisting of 24,983 users who have rated 36 or more jokes, and (2) Jester-2, consisting of 23,500 users who have rated 36 or more jokes. From the dataset, we randomly select $n$ rows to construct the matrix $X^*\in \mathbb{R}^{n\times 100}$. Among the observed elements in $\Omega$, we randomly select entries according to the observation rate OR to form the index set of observed entries $\Gamma$. The rating range is from $r_{\min} = -10$ to $r_{\max} = 10$. The parameters for Algorithm \ref{PAMM} are configured as follows: for six parameter pairs we set $\lambda^0 = 1.4,0.6,0.2,0,11,0,076,0,076$. Additionally, we set initial rank $d=5$, and $k_{\max} = 1000$. The remaining parameters and the initial iteration point setting are consistent with those specified in the previous subsection.

\begin{sidewaystable}[htbp]
\centering
\caption{Comparison of NMAE and CPU time(s) on the Jester joke dataset for parameter pairs $(q,p)=$ $(0.1,0.1)$, $(1/2,1/2)$, $(1,1)$, $(4/3,4/3)$, $(8/5,8/5)$ and  $(2,2)$.}
\label{tab:algorithm_performance}
\begin{tabular}{cccccccccccccc}
\toprule
\multirow{2}{*}{Dataset} & \multirow{2}{*}{($n$, OR)} & \multicolumn{2}{c}{\textbf{(0.1,0.1)}} & \multicolumn{2}{c}{\textbf{(1/2,1/2)}}& \multicolumn{2}{c}{\textbf{(1,1)}}& \multicolumn{2}{c}{\textbf{(4/3,4/3)}} & \multicolumn{2}{c}{\textbf{(8/5,8/5)}} & \multicolumn{2}{c}{\textbf{(2,2)}}\\
\cmidrule(lr){3-4} \cmidrule(lr){5-6} \cmidrule(lr){7-8} \cmidrule(lr){9-10} \cmidrule(lr){11-12} \cmidrule(lr){13-14} 
& & {NMAE} & {Time} & {NMAE} &  {Time} & {NMAE} &  {Time}& {NMAE} &  {Time} & {NMAE} & {Time}& {NMAE} & {Time}\\
\midrule 
\multirow{4}{*}{Jester-1} & (2000, 0.15) & 0.1834  & 1.23 & 0.1831 & 1.21 & 0.1827 & 1.22 &  0.1821 & 1.23  & 0.1813 & 1.21 &  0.1813 & 1.22  \\
& (2000, 0.25) & 0.1726 & 1.25 & 0.1725 & 1.20 & 0.1724 & 1.23 &  0.1723 & 1.24 & 0.1722 & 1.24 &  0.1722 & 1.24  \\
& (4000, 0.15) &  0.1806 & 2.69 &  0.1803 & 2.57 & 0.1799  & 2.58 & 0.1793  & 2.62 & 0.1785  & 2.63 & 0.1785  & 2.63 \\
& (4000, 0.25) & 0.1748  & 2.65 & 0.1748 & 2.58 & 0.1747  & 2.59 & 0.1746  & 2.64 & 0.1746  & 2.62 & 0.1746  & 2.61 \\
\midrule  
\multirow{4}{*}{Jester-2} & (2000, 0.15) & 0.1821 & 1.26 & 0.1819 & 1.23 & 0.1815 & 1.19 & 0.1809 & 1.25 & 0.1802 & 1.27 & 0.1802 & 1.25 \\
& (2000, 0.25) & 0.1731  & 1.22 & 0.1730 & 1.20 & 0.1729 & 1.18 & 0.1728 & 1.22 & 0.1727 & 1.22 & 0.1727 & 1.22\\
& (4000, 0.15) & 0.1810 & 2.75 & 0.1808 &  2.65 & 0.1804  & 2.51 &  0.1799  & 2.60 & 0.1792  & 2.54 &  0.1792  & 2.64 \\
& (4000, 0.25) & 0.1718  & 2.66 & 0.1718  & 2.62 & 0.1717  & 2.59 &  0.1716  &  2.63  & 0.1716  & 2.69 &  0.1716  &  2.66 \\
\bottomrule
\end{tabular}
\end{sidewaystable}

\begin{sidewaystable}[htbp]
\centering
\caption{Comparison of NMAE and CPU time(s) on the MovieLens dataset for parameter pairs, $(q,p)=$ $(0.1,0.1)$, $(1/2,1/2)$, $(1,1)$, $(4/3,4/3)$, $(8/5,8/5)$ and  $(2,2)$.}
\label{tab:algorithm_performance2}
\begin{tabular}{cccccccccccccc}
\toprule
\multirow{2}{*}{Dataset} & \multirow{2}{*}{($m$, $n$, $OR$)} & \multicolumn{2}{c}{\textbf{(0.1,0.1)}} & \multicolumn{2}{c}{\textbf{(1/2,1/2)}}& \multicolumn{2}{c}{\textbf{(1,1)}}& \multicolumn{2}{c}{\textbf{(4/3,4/3)}} & \multicolumn{2}{c}{\textbf{(8/5,8/5)}} & \multicolumn{2}{c}{\textbf{(2,2)}}\\
\cmidrule(lr){3-4} \cmidrule(lr){5-6} \cmidrule(lr){7-8} \cmidrule(lr){9-10}  \cmidrule(lr){11-12} \cmidrule(lr){13-14} 
& & {NMAE} & {Time} & {NMAE} &  {Time} & {NMAE} &  {Time}& {NMAE} &  {Time} & {NMAE} & {Time} & {NMAE} &  {Time}\\
\midrule    
\multirow{2}{*}{Movie-100k} & (1682,943,0.15) & 0.2014 & 16.00 & 0.2017 & 16.46 & 0.2046 & 16.18& 0.2068 & 16.73 & 0.2117 & 16.56 & 0.2575 & 17.30  \\
& (1682,943,0.25) & 0.1939  & 16.32 & 0.1940 & 17.32 & 0.1958 & 17.08 & 0.1975 & 17.72 & 0.2026 & 17.16 & 2498 & 17.90\\
\midrule  
\multirow{4}{*}{Movie-1M} & (1000,1000,0.15) & 0.2160 & 9.46 & 0.2174 & 9.31 & 0.2279 & 9.22 & 0.2312 & 9.87 & 0.2341 & 9.85 & 0.2647 & 9.60  \\
& (1000,1000,0.25) & 0.1992  & 9.68 & 0.1997 &  10.22 & 0.2040  & 9.69 & 0.2055 & 10.38 & 0.2093  & 10.08 & 0.2448 & 10.63 \\
& (2000,2000,0.15) & 0.1961 & 40.80 & 0.1963 & 42.88  & 0.1994 & 42.26 & 0.2023 & 42.76 & 0.2115 & 42.68 & 0.2917 & 42.40 \\
& (2000,2000,0.25) & 0.1888  & 40.56 & 0.1890 & 42.64 & 0.1912 & 42.26 & 0.1938 &  42.58 & 0.2034 & 42.39 & 0.2845 &  43.20 \\
\bottomrule
\end{tabular}
\end{sidewaystable}

The datasets employed in Table \ref{tab:algorithm_performance2} are the MovieLens-100K and MovieLens-1M datasets, available at \url{http://www.grouplens.org/node/73}. The dataset  MovieLens-100K consists of 100,000 ratings for 1,682 movies provided by 943 users, while the dataset MovieLens-1M contains 1,000,209 ratings of 3,952 movies from 6,040 users. The rating values range from $r_{\min} = 1$ to $r_{\max} = 5$. For the MovieLens-100K dataset, all available ratings are used to form the truth matrix $X^*$. For the MovieLens-1M dataset, a subset is constructed by randomly selecting $m$ rows and $n$ columns from the full data. For six parameter pairs we set $\lambda^0 = 0.35,0.3,0.25,0,12,0,1,0,09$. Additionally, we set $k_{\max} = 1200$, the other parameters are kept the same as those for the Jester joke dataset. Since the matrix entries corresponding to the movie rating data account for less than $5\%$ of the total $mn$ elements, the initial point obtained via singular value decomposition is not effective. Therefore, we adopt a fixed initial iterative point $(U^0,V^0)$, where the elements of the first column of $U^0$ and $V^0$ are randomly selected from $[0,\sqrt{5}]$ and the rest elements are randomly generate by Gaussian distribution with zero mean and a variance of $0.05$.

From the table, it is evident that for the Jester joke dataset, the six parameter configurations achieve comparable performance in both computational time and accuracy. Under low observation rates, selecting larger parameter values leads to a slight performance gain. For the MovieLens dataset, choosing small values of parameters ($(q,p)=(0.1,0.1)$ and $(q,p)=(1/2,1/2)$) exhibit a distinct error advantage, especially under low observation rates. By contrast, $(q,p)=(2,2)$  yields substantially larger errors compared to all other parameter settings.



\section{Conclusions}\label{section7}
In this paper, we analyze the properties of the critical points of the group-sparse factorized Schatten-$q$ norm regularized low‑rank matrix recovery problem (\ref{M3}). Compared to the factorized model with nuclear norm regularization (\ref{nuclear}), the Schatten-$q$ norm implicitly encourages column orthogonality at critical points. From this insight, we introduce the concept of S-critical points, which require mild conditions yet guarantee column orthogonality of critical points. Moreover, global minimizer must be S-critical and we provide an easily operable criterion for identifying S-critical points (Remark \ref{Remark}). Furthermore, we establish an error bound between S-critical points and the true matrix (the optimal solution) for the factorized problem (\ref{M3}). In addition, for the least-squares loss (\ref{p1}), we show that for appropriate parameter $\lambda$, the objective function possesses the Kurdyka–{\L}ojasiewicz property with exponent $1/2$ at S-critical points. We also develop an inexact proximal alternating linearized minimization method \cite{bcd} to solve the factorized low-rank matrix recovery problem (\ref{M2}), allowing for inexact updates at each iteration. We prove that our inexact algorithm guarantees global convergence and a convergence rate for the factorized model under KL condition. And we show that under suitable condition, for the least-squares loss model (\ref{p1}), the iterate points converge linearly (Proposition \ref{Convergence_result}). The extensive numerical experiments validate the effectiveness of the proposed algorithm and the theoretical properties of the factorized model.

\section*{Declarations} 
\bmhead{Funding}
This work was funded by the National Natural Foundation of China (Grand number 12571323) and the Hong Kong RGC Senior Research Fellow Scheme [No. SRFS22235S02] and the GRF Grant 15307822.

\bmhead{Conflict of interest/Competing interests} 
The authors declare they have no financial interests.

\begin{appendices}
\counterwithin{theorem}{section} 
\setcounter{theorem}{0}    
\section{The proof of Theorem \ref{Theorem4.4}}\label{APTheorem4.4}
\begin{proof}
We partition the index set of the columns of the matrix $U$ into subsets $I_1, I_2, \dots, I_s$ according to the Euclidean norms of its columns, such that any two indices $i, j$ belonging to the same subset satisfy $\|U_{:,i}\|_2 = \|U_{:,j}\|_2$, while any two indices from distinct subsets satisfy $\|U_{:,i}\|_2 \neq \|U_{:,j}\|_2$.
Since $U^T U = V^T V$, the same partition $I_1, I_2, \dots, I_s$ also groups the column indices of the matrix $V$ according to whether the columns have equal Euclidean norm.

We arbitrarily select an index set $I_p$, and denote $U^p = U|_{I_p}$ and $V^p = V|_{I_p}$. Through the Gram–Schmidt column orthogonalization, there exists an upper triangular matrix $T = (t_{ij})$ with ones on the diagonal such that, by setting
$$\hat{u}_h = U^p_h - \sum_{l=1}^{h-1} t_{hl} U^p_l, \quad h = 1, 2, \ldots, |I_p|,$$
we have $\langle \hat{u}_i, \hat{u}_j \rangle = 0$ for $i \ne j$, and also $\langle U^p_{:,i}, \hat{u}_j \rangle = 0$.

Consequently, we obtain the following $\frac{1}{2}|I_p|(|I_p|-1)$ equations:
$$\langle U^p_{:,i}, U^p_{:,j} \rangle = \sum_{l=1}^{j-1} t_{lj} \langle U^p_{:,i}, U^p_{:,l} \rangle, \quad \text{where } 1 \le i < j \le |I_p|.$$
Thus, we obtain a system of linear equations with unknowns $t_{ij}$.

Similarly, there exists an upper triangular matrix $\bar{T} = (\bar{t}_{ij})$ with ones on the diagonal such that, by setting
$$\hat{v}_h = V^p_h - \sum_{l=1}^{h-1} \bar{t}_{hl} V^p_l, \quad h = 1, 2, \ldots, |I_p|,$$
we have $\langle \hat{v}_i, \hat{v}_j \rangle = 0$ for $i \ne j$. Noting that $(U^p)^T U^p = (V^p)^T V^p$, we have $\langle U^p_{:,i}, U^p_{:,j} \rangle = \langle V^p_{:,i}, V^p_{:,j} \rangle$. Therefore, the system of linear equations for the unknowns $\bar{t}_{ij}$ derived from the orthogonality condition $\langle V^p_{:,i}, \hat{v}_j \rangle = 0$ has exactly the same coefficients as that for $t_{ij}$. Hence $\bar{t}_{ij}$ and $t_{ij}$ share the same solution, which we denote as $T_p = T = \bar{T}$.

Let $\hat{U}^p = (\hat{u}_i)$ and $\hat{V}^p = (\hat{v}_i)$. We have
$$U^p (V^p)^T = \hat{U}^p T_p^{-1} T_p^{-T} (\hat{V}^p)^T = \hat{U}^p (T_p^T T_p)^{-1} (\hat{V}^p)^T.$$
Let $D_p = \operatorname{diag}(\|\hat{u}_1\|, \ldots, \|\hat{u}_{|I_p|}\|)$. Since the columns of $\hat{U}^p$ and $\hat{V}^p$ have identical Euclidean norms, we also have $D_p = \operatorname{diag}(\|\hat{v}_1\|, \ldots, \|\hat{v}_{|I_p|}\|)$. Hence,$$
U^p (V^p)^T = (\hat{U}^p D_p^{-1}) D_p (T_p^T T_p)^{-1} D_p (\hat{V}^p D_p^{-1})^T.$$

Note that $(D_p T_p^T T_p D_p)^{-1}$ is a positive definite matrix; therefore, there exist a unitary matrix $W_p$ and a diagonal matrix $\Sigma_p$ such that
$$(D_p T_p^T T_p D_p)^{-1} = W_p \Sigma_p W_p^T.$$
Thus,$$
U^p (V^p)^T = (\hat{U}^p D_p^{-1} W_p) \Sigma_p (\hat{V}^p D_p^{-1} W_p)^T.$$
Observe that both $\hat{U}^p D_p^{-1} W_p$ and $\hat{V}^p D_p^{-1} W_p$ are unitary matrices. Defining$$
\bar{U}^p = \hat{U}^p D_p^{-1} W_p \Sigma_p^{1/2}, \quad \bar{V}^p = \hat{V}^p D_p^{-1} W_p \Sigma_p^{1/2},$$
we obtain matrices $\bar{U}^p$ and $\bar{V}^p$ with orthonormal columns that satisfy
$$\bar{U}^p (\bar{V}^p)^T = U^p (V^p)^T.$$

We note that the above conclusion holds for every index set $I_p$. Therefore, there exist column-orthonormal matrices $\bar{U} = U E$ and $\bar{V} = V E$ such that $U V^T = \bar{U} (\bar{V})^T$, where the restriction of $E$ to $I_p$ is given by $E|_{I_p} = E_p = T_p D_p^{-1} W_p \Sigma_p^{1/2}$.

We then have
\begin{align*}
F\bigl(\bar{U}(\bar{V})^T\bigr) \bar{V}^p + \bar{U}^p D_{\bar{U}}|_{I_p}
&= F(U V^T) V^p E_p + U^p E_p D_{U^p E_p} \\
&= U^p E_p D_{U^p E_p} - U^p D_{U^p} E_p \\
&= U^p E_p \bigl(D_{U^p E_p} - D_{U^p}\bigr) \\
&= U^p E_p \bigl(D_{\bar{U}^p} - D_{U^p}\bigr).
\end{align*}
Thus, $F\bigl(\bar{U}(\bar{V})^T\bigr) \bar{V}^p + \bar{U}^p D_{\bar{U}}|_{I_p} = 0$ if and only if $D_{\bar{U}^p} = D_{U^p}$, i.e., $\|\bar{U}^p_{:,i}\|_2 = \|U^p_{:,i}\|_2$ for every $i$. Similarly, we also have $F\bigl(\bar{U}(\bar{V})^T\bigr)^T \bar{U}^p + \bar{V}^p D_{\bar{V}}|_{I_p} = 0$ if and only if $\|\bar{V}^p_{:,i}\|_2 = \|V^p_{:,i}\|_2$ for every $i$.

We note that if $(U,V)$ is a global minimizer of $\Phi_\lambda$ (see Eq.~\ref{M3}), then $UV^T$ is also a global minimizer of $\bar{\Phi}_\lambda$ (see Eq.~\ref{M4}). Consequently, $(\bar{U},\bar{V})$ is also a global minimizer of $\Phi_\lambda$, since we observe that in this case $(\bar{U},\bar{V})$ satisfies $\|\bar{U}(\bar{V})^T\|_{S_{\frac{q}{2}}}^{\frac{q}{2}} = \frac{1}{2}\bigl(\|\bar{U}\|_{2,q}^q + \|\bar{V}\|_{2,q}^q\bigr)$, and therefore it is certainly a critical point of $\Phi_\lambda$.

Hence, if $(\bar{U},\bar{V})$ is not a critical point of $\Phi_\lambda$, then $(U,V)$ cannot be a global minimizer of $\Phi_\lambda$.
\end{proof}

\section{The proof of Lemma \ref{RD}}\label{APRD}

\begin{proof}
Denote $z_t=x+t(y-x)$, $t\in[0,1]$. The Jacobian matrix of $f$ is $$J_f(x)=\|x\|_2^{q-2}I+(q-2)\|x\|_2^{q-4}xx^T.$$
Then 
\begin{equation}\label{Jf}
f(y)-f(x)=\int_{0}^{1} J_f(z_t) (y-x)dt=\int_{0}^{1} J_f(z_t) dt (y-x). 
\end{equation}

Given any $v\in \mathbb{R}^n$, we have 
\begin{align*}
v^T\int_{0}^{1} J_f(z_t) dt ~v&=\int_{0}^{1} v^T J_f(z_t) v dt 
\\&=\int_{0}^{1} v^T (\|z_t\|_2^{q-2}I+(q-2)\|z_t\|_2^{q-4}z_tz_t^T) v dt
\\&=\int_{0}^{1} \|z_t\|_2^{q-2} dt \|v\|_2^2+(q-2)\int_{0}^{1} \|z_t\|_2^{q-4}\|z_t^Tv\|_2^2 dt 
\\& \ge \int_{0}^{1} \|z_t\|_2^{q-2} dt \|v\|_2^2+(q-2)\int_{0}^{1} \|z_t\|_2^{q-2} dt \|v\|_2^2
\\& = (q-1)\int_{0}^{1} \|z_t\|_2^{q-2} dt \|v\|_2^2
\\& \ge (q-1) (\|x\|_2+\varepsilon)^{q-2} \|v\|_2^2,
\end{align*}
which implies
\begin{equation}\label{sigmamin}
\sigma_{\min}(\int_{0}^{1} J_f(z_t) dt)\ge (q-1) (\|x\|_2+\varepsilon)^{q-2}
\end{equation}

Furthermore, we have 
\begin{align*}
v^T\int_{0}^{1} J_f(z_t) dt ~v&=\int_{0}^{1} v^T J_f(z_t) v dt 
\\&=\int_{0}^{1} v^T (\|z_t\|_2^{q-2}I+(q-2)\|z_t\|_2^{q-4}z_tz_t^T) v dt
\\&=\int_{0}^{1} \|z_t\|_2^{q-2} dt \|v\|_2^2+(q-2)\int_{0}^{1} \|z_t\|_2^{q-4}\|z_t^Tv\|_2^2 dt 
\\& \le \int_{0}^{1} \|z_t\|_2^{q-2} dt \|v\|_2^2
 \le (\|x\|_2-\varepsilon)^{q-2} \|v\|_2^2,
\end{align*}
which implies
\begin{equation}\label{sigmamax}
\sigma_{\max}(\int_{0}^{1} J_f(z_t) dt)\le (\|x\|_2-\varepsilon)^{q-2}
\end{equation}
Combining (\ref{Jf}), (\ref{sigmamin}) and (\ref{sigmamax}), we have
 $$(q-1) (\|x\|_2+\varepsilon)^{q-2} \|x-y\|_2
\le \|f(y)-f(x)\|_2\le (\|x\|_2-\varepsilon)^{q-2} \|x-y\|_2.$$
\end{proof}

\section{The proof of Lemma \ref{Inexact}}\label{APInexact}

\begin{proof}
Denote $h(l)=l(\bar{x}-y)+\lambda \bar{x}^{q-1}$. Then we have $h(\bar{L})=0$.

Note that $\frac{L}{2}(x^* - y)^2 + \frac{\lambda}{q}(x^*)^q \le \frac{L}{2}y^2$, we obtain the two inequalities $\frac{L}{2}(x^* - y)^2 \le \frac{L}{2}y^2$ and $\frac{\lambda}{q}(x^*)^q \le \frac{L}{2}y^2$. This deduces that $$\frac{y}{2} \le x^* \le \left(\frac{Lq}{2\lambda}\right)^{\!1/q} y^{2/q}.$$ Next, we estimate the bound between $L$ and $\bar{L}$. 

\textbf{Case 1: $\bar{x} \ge x^*$.}  
We have $\bar{L}=\frac{\lambda \bar{x}^{q-1}}{(y-\bar{x})}\ge \frac{\lambda (x^*)^{q-1}}{(y-x^*)}=L$.  Using the identities  
$$\bar{L}(\bar{x} - y) + \lambda \bar{x}^{\,q-1} = 0, \quad 
L(x^* - y) + \lambda (x^*)^{q-1} = 0,$$
we derive the inequality  
$$\frac{\lambda}{L}(x^*)^q - \frac{\lambda}{\bar{L}(1-\delta)^q}(x^*)^q
   \le \frac{\lambda}{L}(x^*)^q - \frac{\lambda}{\bar{L}}\bar{x}^{\,q}
   = \bar{x} - x^*
   \le \frac{\delta}{1-\delta}x^* .$$
This leads to the following bounds on the ratio $\bar{L}/L$:

\begin{itemize}
    \item For $1 \le q \le 2$, 
    \begin{equation}\label{T1}
    1 \le \frac{\bar{L}}{L}
        \le 
           \frac{(1-\delta)^{-q}}{\,1 - \dfrac{\delta}{1-\delta}\,
                      \left(\dfrac{y}{2}\right)^{\!1-q}
                      \dfrac{L}{\lambda}\,}
        \le \frac{1}{\Bigl(1 - \dfrac{y^{\,1-q}L}{\lambda\,2^{2-q}}\,\delta\Bigr)(1-\delta)^q}.
    \end{equation}

    \item For $0 < q < 1$,    
   \begin{equation}\label{T2}
   1 \le \frac{\bar{L}}{L}
        \le 
           \frac{(1-\delta)^{-q}}{\,1 - \dfrac{\delta}{1-\delta}\,
                      \left(\dfrac{Lq}{2\lambda}y^{2}\right)^{\frac{1-q}{q}}
                      \dfrac{L}{\lambda}\,}
        \le \frac{1}{\Bigl(1 - \frac{4}{q}y^{\,2/q-2}(\frac{Lq}{2\lambda})^{\frac{1}{q}}\delta\Bigr)(1-\delta)^q}.
        \end{equation}
\end{itemize}

\textbf{Case 2: $\bar{x} \le x^*$.}  
Proceeding similarly, we obtain  
$$\frac{\lambda}{\bar{L}(1+\delta)^q}(x^*)^q - \frac{\lambda}{L}(x^*)^q
   \le \frac{\lambda}{\bar{L}}\bar{x}^{\,q} - \frac{\lambda}{L}(x^*)^q
   = x^* - \bar{x}
   \le \frac{\delta}{1+\delta}x^*.$$ 
Furthermore, $\bar{L}=\frac{\lambda \bar{x}^{q-1}}{(y-\bar{x})}\le \frac{\lambda (x^*)^{q-1}}{(y-x^*)}=L$. Hence, for the ratio $\bar{L}/L$ we have:

\begin{itemize}
    \item For $1 \le q \le 2$,
    \begin{equation}\label{T3}
    1 \ge \frac{\bar{L}}{L}
        \ge 
           \frac{(1+\delta)^{-q}}{\,1 + \dfrac{\delta}{1+\delta}\,
                      \left(\dfrac{y}{2}\right)^{\!1-q}
                      \dfrac{L}{\lambda}\,}
        \ge \frac{1}{\Bigl(1 + \dfrac{y^{\,1-q}L}{\lambda\,2^{1-q}}\,\delta\Bigr)(1+\delta)^q}.
        \end{equation}

    \item For $0 < q < 1$,
    \begin{equation}\label{T4}1 \ge \frac{\bar{L}}{L}
        \ge 
           \frac{(1+\delta)^{-q}}{\,1 + \dfrac{\delta}{1+\delta}\,
                      \left(\dfrac{Lq}{2\lambda}y^{2}\right)^{\frac{1-q}{q}}
                      \dfrac{L}{\lambda}\,}
        \ge \frac{1}{\Bigl(1 + \frac{2}{q}y^{\,2/q-2}(\frac{Lq}{2\lambda})^\frac{1}{q}\delta\Bigr)(1+\delta)^q}.   \end{equation}
\end{itemize}
Combining (\ref{T1}), (\ref{T2}), (\ref{T3}) and (\ref{T4}), this implies that $$(1+C\delta)^{-q-1}L\le \bar{L}\le (1-C\delta)^{-q-1} L.$$

Next, we demonstrate that $\mathcal{P}_{\bar{g}}(\bar{x})$ is the minimizer of $\bar{g}(x)=\frac{\bar{L}}{2} (x-y)^2 + \frac{\lambda}{q} x^q$. Note that for $1 < q \le 2$, the function $\bar{g}(x)$ is convex and differentiable. Consequently, $\mathcal{P}_{\bar{g}}(\bar{x})=\bar{x}$ is indeed the minimizer of $\bar{g}(x)$. For the case $0 < q \le 1$, since the minimum of function $\bar{g}$ must be attained either at a critical point or at a boundary point of its domain, and noting that the gradient at the right endpoint is positive, it follows that the minimum of $\bar{g}$ can only occur at a $\bar{x}$ or at the left endpoint 0. Consequently, $\mathcal{P}_{\bar{g}}(\bar{x})$ is exactly the minimizer of $\bar{g}$.
\end{proof}

\section{The proof of Proposition \ref{single-value}}\label{APsingle-value}

\begin{proof}
There exist $K>0$, such that for $k\ge K$, we have $\lambda^k=\lambda_0$. In the following proof, we consider $k\ge K$.
By Corollary 5.20 and Example 5.23 of \cite{Rockafellar2009}, since $\operatorname{prox}_{\frac{\lambda^k}{qL_{1}^k}\|\cdot\|_2^q}$ and  $\operatorname{prox}_{\frac{\lambda^k}{pL_{2}^k}\|\cdot\|_2^p}$ are single-valued near 
$H^k(U^k,V^k)$ and $\bar{H}^k(U_{\text{E}}^{k+1}(0),V^k)$, respectively, then this two proximal operator are also are continuous at $H^k(U^k,V^k)$ and $\bar{H}^k(U_{\text{E}}^{k+1}(0),V^k)$, respectively. Then by definition, we have \begin{equation}\label{D9}
U_{E}^{k+1}(\beta_k)\to U_{E}^{k+1}(0) ~\text{and}~ V_{E}^{k+1}(\beta_k)\to V_{E}^{k+1}(0), ~\beta_k\to 0.
\end{equation} 
Furthermore, there exist $\lambda_{i,t}^k(\beta_k)$ for $i=1,2$ and $t=1,2,\ldots,d$, such that the inexact solution $(U^{k+1}(\beta_k),V^{k+1}(\beta_k))$ satisfies
\begin{equation}
U_t^{k+1}(\beta_k) \in \operatorname{prox}_{\frac{\lambda_{1,t}^k(\beta_k)}{qL_{1}^k}\|\cdot\|_2^q} \Big(H^k(\widetilde{U}^k(\beta_k),V^k)\Big),
\end{equation}
\begin{equation}
V_t^{k+1}(\beta_k) \in \operatorname{prox}_{\frac{\lambda_{2,t}^k(\beta_k)}{pL_{2}^k}\|\cdot\|_2^p} \Big( \bar{H}^k(U^{k+1}(\beta_k),\widetilde{V}^k(\beta_k)) \Big),
\end{equation}
furthermore we have $\Big|\Big(\frac{\lambda^k}{\lambda_{i,t}^k(\beta_k)}\Big)^{\frac{1}{q+1}}-1\Big|\le \delta_k$. Then by continuity of objective function and (\ref{D9}), when $\beta_k\to 0$ and $\delta_k\to 0$, we have \begin{equation}\label{D1}
\Phi_{\lambda_0}(U_{E}^{k+1}(\beta),V_{E}^{k+1}(\beta))\to\Phi_{\lambda_0}(U_{E}^{k+1}(0),V_{E}^{k+1}(0)).
\end{equation}
 
Applying Lemma 1 \cite{xu2017globally}, we have
\begin{align}\label{D4}
&F(U^{k},V^{k})+\frac{\lambda^k}{q}\|U^k\|_{2,q}^q-\big(F(U_{E}^{k+1}(0),V^{k})+\frac{\lambda^k}{q}\|U_{E}^{k+1}(0)\|_{2,q}^q\big)\nonumber\\&\ge \frac{L_{1}^k}{8}\|U^{k}-U_{E}^{k+1}(0)\|_F^2,
\end{align}
and
\begin{align}\label{D5}
&F(U_{E}^{k+1}(0),V^{k})+\frac{\lambda^k}{p}\|V^k\|_{2,p}^p-\big(F(U_{E}^{k+1}(0),V_{E}^{k+1}(0))+\frac{\lambda^k}{p}\|V_{E}^{k+1}(0)\|_{2,p}^p\big)\nonumber\\&\ge \frac{L_{2}^k}{8}\|V^{k}-V_{E}^{k+1}(0)\|_F^2.
\end{align}
Since at least one of $U^{k}\ne U_{E}^{k+1}(0)$ and $V^{k}\ne V_{E}^{k+1}(0)$ holds, by (\ref{D4}) and (\ref{D5}), then 
\begin{equation}\label{D2}
\Phi_{\lambda_0}(U^k,V^k)-\Phi_{\lambda_0}(U_{E}^{k+1}(0),V_{E}^{k+1}(0))>0.
\end{equation}
Note that when $\delta_k\to 0$, $\lambda_{i,k}^k(\beta_k)\to \lambda^k$. Then by continuity, when $\delta_k\to 0$, we have 
\begin{equation}\label{D6}
\Phi_{\lambda}^{k+1}(U^{k},V^{k})\to \Phi_{\lambda_0}(U^k,V^k), ~\Phi_{\lambda}^{k+1}(U_{E}^{k+1}(\beta),V_{E}^{k+1}(\beta))\to  \Phi_{\lambda_0}(U_{E}^{k+1}(\beta),V_{E}^{k+1}(\beta))
\end{equation}
\begin{equation}
U^{k+1}(\beta) \to U_{E}^{k+1}(\beta),~ V^{k+1}(\beta) \to V_{E}^{k+1}(\beta), 
\end{equation}
\begin{equation}\label{D3}
\Phi_{\lambda}^{k+1}(U^{k+1}(\beta),V^{k+1}(\beta))\to\Phi_{\lambda}^{k+1}(U_{E}^{k+1}(\beta),V_{E}^{k+1}(\beta)).
\end{equation}
Furthermore, by continuity, when $\delta_{k-1}\to 0$, we have \begin{equation}\label{DN}
\Phi_{\lambda}^k(U^k,V^k)\to \Phi_{\lambda_0}(U^k,V^k),~ \Phi_{\lambda}^k(\bar{U},\bar{V})-\Phi_{\lambda}^{k+1}(\bar{U},\bar{V})\to 0,
\end{equation}
Then by (\ref{D1}), (\ref{D2}), (\ref{D6}), (\ref{D3}),  (\ref{DN}), there exist $\bar{\beta}_k>0$, $\bar{\delta}_{k-1}>0$ and $\bar{\delta}_{k-1}>0$ such that
for any $\beta_k\in [0,\bar{\beta}_k]$ and $\delta_{k}\in[0,\bar{\delta}_{k}]$, we have 
\begin{equation*}
\Phi^{k+1}_\lambda(U^{k},V^{k})\ge \Phi^{k+1}_\lambda(U^{k+1}(\beta),V^{k+1}(\beta))  
\end{equation*}
for any $\beta_k\in [0,\bar{\beta}_k]$, $\delta_{k-1}\in[0,\bar{\delta}_{k-1}]$ and $\delta_{k}\in[0,\bar{\delta}_{k}]$, we have
\begin{equation*}
\Phi_{\lambda}^k(U^k,V^k)-\Phi_{\lambda}^k(\bar{U},\bar{V})\ge \Phi_{\lambda}^{k+1}(U^{k+1}(\beta),V^{k+1}(\beta))-\Phi_{\lambda}^{k+1}(\bar{U},\bar{V}).   
\end{equation*}
\end{proof}

\section{The proof of Lemma \ref{SL}}\label{APSL}

\begin{proof}
If $0<\theta \le \frac{1}{2}$, since $\underset{k\to +\infty}{\lim} A_k \to 0$ and $(1-\theta)/\theta \ge 1$, then for sufficiently large $k$, $$A_{k-1} \le (a+b) (A_{k-1} - A_{k+1}) +  c \beta^k . $$
Then $A_{k+1} \le \frac{a+b-1}{a+b}A_{k-1}+ \frac{c}{a+b} \beta^k$. 
Let $\alpha=\sqrt{\frac{a+b-\frac{1}{2}}{a+b}}$. We can find sufficiently large $\bar{K}$ and $C$ such that $A_{\bar{K}} \le C \alpha^{\bar{K}}$, $A_{\bar{K}+1} \le C \alpha^{\bar{K}+1}$ and $c\beta^k\le \frac{C}{2} \alpha^{k-1}$ for $k\ge \bar{K}$. Assume that for $\bar{K} \le k \le \hat{K}$ we have $A_{k} \le C \alpha^{k}$. Next, we prove that $A_{\hat{K}+1} \le C \alpha^{\hat{K}+1}$.

We have \begin{align*}
A_{\hat{K}+1}&\le \frac{a+b-1}{a+b}A_{\hat{K}-1}+ \frac{c}{a+b} \beta^{\hat{K}}
\\ & \le C\frac{a+b-1}{a+b}\alpha^{\hat{K}-1}+ \frac{c}{a+b} \beta^{\hat{K}}
\\ & \le C\frac{a+b-1}{a+b}\alpha^{\hat{K}-1}+ \frac{C}{2(a+b)} \alpha^{\hat{K}-1}=C\alpha^{\hat{K}+1}.
\end{align*}
Thus, by induction, we obtain the conclusion 1.

If $\frac{1}{2} <\theta <1$, since $\underset{k\to +\infty}{\lim} A_k \to 0$ and $(1-\theta)/\theta < 1$, then for sufficiently large $k$, $$A_{k-1} \le (a+b) (A_{k-1} - A_{k+1})^{(1-\theta)/\theta}  +  c \beta^k . $$ Then by Jensen’s inequality ($u^p+v^p\ge 2^{1-p}(u+v)^p$, $p\ge 1$), we have $$A_{k-1}^{\theta/(1-\theta)}  \le 2^{(1-2\theta)/\theta} ((a+b) (A_{k-1} - A_{k+1}) +  c \beta^{k\theta/(1-\theta)}).$$
Denote $c_1=2^{(2\theta-1)/\theta}/(a+b)$, $c_2=c/(a+b)$ and $\bar{\beta}=\beta^{\theta/(1-\theta)}$. Then
$$A_{k+1} \le A_{k-1}- c_1 A_{k-1}^{\theta/(1-\theta)}  +  c_2 \bar{\beta}^{k}.$$
Let $C$ be sufficient large, such that $c_1C^{(2\theta-1)/(1-\theta)}\ge 4\max\{1,\frac{1-\theta}{2\theta -1}\}$.
For sufficiently large $k$, we have 
$$c_2 \bar{\beta}^{k}\le \frac{c_1}{2} C^{\theta/(1-\theta)} (k-1)^{-\theta/(2\theta-1)}.$$
We can find sufficiently large $\bar{K}$ and $\bar{N}<\bar{K}-1$, such that $$A_{\bar{K}} \le C (\bar{K}-\bar{N})^{-(1-\theta)/(2\theta-1)}, A_{\bar{K}+1} \le C (\bar{K}-\bar{N}+1)^{-(1-\theta)/(2\theta-1)}.$$ Assume that for $\bar{K} \le k \le \hat{K}$ we have $A_{k} \le C (k-\bar{N})^{-(1-\theta)/(2\theta-1)}$. Next, we prove that $A_{\hat{K}+1} \le C (\hat{K}-\bar{N}+1)^{-(1-\theta)/(2\theta-1)}$.

We have \begin{align*}
A_{\hat{K}+1} & \le A_{\hat{K}-1}- c_1 A_{\hat{K}-1}^{\theta/(1-\theta)}  +  c_2 \bar{\beta}^{\hat{K}} \\ & \le
C (\hat{K}-\bar{N}-1)^{-(1-\theta)/(2\theta-1)}- c_1 C^{\theta/(1-\theta)} (\hat{K}-\bar{N}-1)^{-\theta/(2\theta-1)}+ c_2 \bar{\beta}^{\hat{K}}
\\ & =C (\hat{K}-\bar{N}-1)^{-(1-\theta)/(2\theta-1)}(1-\frac{c_1C^{(2\theta-1)/(1-\theta)}}{\hat{K}-\bar{N}-1})+ c_2 \bar{\beta}^{\hat{K}}
\\ & \le C (\hat{K}-\bar{N}-1)^{-(1-\theta)/(2\theta-1)}(1-\frac{c_1C^{(2\theta-1)/(1-\theta)}}{2(\hat{K}-\bar{N}-1)})
\end{align*}

Since \begin{align*}
(\frac{\hat{K}-\bar{N}-1}{\hat{K}-\bar{N}+1})^{(1-\theta)/(2\theta-1)}&=(1-\frac{2}{\hat{K}-\bar{N}+1})^{(1-\theta)/(2\theta-1)}
\\& \ge 1-\frac{2\max\{1,(1-\theta)/(2\theta-1)\}}{\hat{K}-\bar{N}+1}\\&\ge 1-\frac{c_1C^{(2\theta-1)/(1-\theta)}}{2(\hat{K}-\bar{N}+1)} . 
\end{align*}
This implies \begin{equation}
A_{\hat{K}+1} \le C (\hat{K}-\bar{N}+1)^{-(1-\theta)/(2\theta-1)}.
\end{equation}
Thus, by induction, we obtain the conclusion 2.
\end{proof}





\end{appendices}


\bibliography{sn-bibliography}

\end{document}